\documentclass{amsart}
\usepackage{graphicx}
\usepackage{setspace,dsfont,enumerate,cite,color,bm,float,epstopdf,verbatim,subfiles,hyperref}    
\usepackage{amsmath}
\usepackage{amssymb}
\usepackage[outercaption]{sidecap}
\usepackage{tabularx}
\usepackage[table]{xcolor}
\usepackage{pgffor}
\usepackage{comment}
\usepackage[caption=false]{subfig}

\newtheorem{theorem}{Theorem}[section]
\newtheorem{remark}[theorem]{Remark}
\newtheorem{proposition}[theorem]{Proposition}
\newtheorem{assumptions}{Assumptions}{\bf}{\rm}

\floatstyle{ruled}
\newfloat{algorithm}{tbp}{loa}
\providecommand{\algorithmname}{Algorithm}
\floatname{algorithm}{\protect\algorithmname}

\makeatletter
\def\maketag@@@#1{\hbox{\m@th\normalfont\normalsize#1}}
\makeatother

\definecolor{darkred}{rgb}{0,0,0}
\newcommand\mi[1]{\textcolor{darkred}{#1}}
\definecolor{darkgreen}{rgb}{0,0,0}

\definecolor{darkblue}{rgb}{0,0,0}
\newcommand\mmd[1]{\textcolor{darkblue}{#1}}

\renewcommand{\Delta}{\triangle}
\newcommand{\eps}{\varepsilon} 
\newcommand{\dee}{\mathrm{d}}

\newcommand{\tr}{\mathrm{tr}}

\newcommand{\bbK}{\mathbb K}
\newcommand{\bbL}{\mathbb L}

\graphicspath{{./images/},{./Figs/}}

\begin{document}
\sloppy

\title{Hierarchical Bayesian Level Set Inversion
}


\author[M. M. Dunlop]{Matthew M. Dunlop}
\address[]{Computing \& Mathematical Sciences, California Institute of Technology, Pasadena, CA, USA}
 \email{mdunlop@caltech.edu}
 
\author[M. A. Iglesias]{Marco A. Iglesias}
\address[]{School of Mathematical Sciences, University of Nottingham, Nottingham, UK}
 \email{marco.iglesias@nottingham.ac.uk} 
 
\author[A. M. Stuart]{Andrew M. Stuart}
\address[]{Computing \& Mathematical Sciences, California Institute of Technology, Pasadena, CA, USA}
\email{astuart@caltech.edu} 

\maketitle

\begin{abstract}
The level set approach has proven widely successful in the study of
inverse problems for interfaces, 
since its systematic development in the 1990s. Recently it
has been employed in the context of Bayesian inversion, allowing for the
quantification of uncertainty within the reconstruction of interfaces. 
However the
Bayesian approach is very sensitive to the length and amplitude scales
in the prior probabilistic model. This paper demonstrates how
the scale-sensitivity can be circumvented by means of a hierarchical approach,
using a single scalar parameter. Together with careful consideration of
the development of algorithms which encode probability measure
equivalences as the hierarchical parameter is varied, this leads to 
well-defined Gibbs based MCMC methods found by alternating 
Metropolis-Hastings updates of the level set function and the hierarchical
parameter. These methods demonstrably outperform non-hierarchical
Bayesian level set methods.
\end{abstract}

\section{Introduction}
\label{sec:I}

\subsection{Background}
\label{ssec:bg}

The level set method has been pervasive as a tool for the
study of interface problems since its introduction 
in the 1980s \cite{Osher1988}. In a seminal paper in the 1990s, 
Santosa demonstrated the power of the approach for the study of
inverse problems with unknown interfaces \cite{Santosa1996}. 
The key benefit of adopting the level set parametrization 
of interfaces is that topological changes are permitted. In
particular for inverse problems the number of connected components 
of the field does not need to be known a priori. The idea is illustrated
in Figure \ref{fig:idea}. The type of unknown functions that we might
wish to reconstruct are piecewise continuous functions, illustrated in
the bottom row by piecewise constant ternary functions. However in the
inversion we work with a smooth function, shown in the top row and
known as the {\em level-set function}, which
is thresholded to create the desired unknown function in the bottom row.
This allows the inversion to be performed on smooth functions, and
allows for topological changes to be detected during the course of
algorithms.  After Santosa's paper there were many subsequent
papers employing the level set representation for classical
inversion, and examples include \cite{Burger2001, Tai2004, Chung2005, Dorn2006},
and the references therein.

In many inverse problems arising in modern day science and engineering, 
the data is noisy and prior regularizing information is naturally
expressed probabilistically since it contains uncertainties. In this
context, Bayesian inversion is a very attractive conceptual approach
\cite{kaipio}. Early adoption of the Bayesian approach within level set 
inversion, especially in the context of history matching for reservoir 
simulation, includes the papers 
\cite{Xie2011,Ping2014,Lorentzen2012,Lorentzen2012_2}. 
In a recent paper \cite{levelset} the mathematical foundations of 
Bayesian level set inversion were developed, and a well-posedness theorem
established, using the infinite dimensional Bayesian framework
developed in \cite{inverse,lasanen1,lasanen2,lecturenotes}.
An ensemble Kalman filter method has also been applied in the Bayesian level 
set setting \cite{reg_enkf} to produce estimates of piecewise constant 
permeabilities/conductivities in groundwater flow/electrical impedance 
tomography (EIT) models.

For linear Bayesian inverse problems, the adoption of Gaussian priors leads
to Gaussian posteriors, formulae for which can be explicitly computed  
\cite{linear1,linear2,linear3}. However the {\em level set map}, which takes
the smooth underlying level set function (top row, Figure \ref{fig:idea})
into the physical unknown function (bottom row, Figure \ref{fig:idea}) is
nonlinear; indeed it is discontinuous. As a consequence, Bayesian level set inversion, even for
inverse problems which are classically-speaking `linear', does not
typically admit closed form solutions for the posterior distribution
on the level set function. 
Thus, in order to produce samples from the posterior arising in the Bayesian 
approach, MCMC methods are often used. Since the posterior is typically 
defined on an infinite-dimensional space in the context of inverse problems, 
it is important that the MCMC algorithms used are well-defined on such spaces. 
A formulation of the Metropolis-Hastings algorithm on general state spaces 
is given in \cite{tierney}. A particular case of this algorithm, well-suited 
to posterior distributions on function spaces and Gaussian priors, is the 
preconditioned Crank-Nicolson (pCN) method introduced (although not named
this way) in \cite{diff_bridges}. As the method is defined directly 
on a function space, it has desirable properties related 
to discretization -- in particular the method is robust with respect to 
mesh refinement (discretization invariance) -- 
see \cite{Cotter2013} and the references therein.
On the other hand, the need for hierarchical models in Bayesian
statistics, and in particular in the context of non-parametric
(i.e. function space) methods in machine learning, 
is well-established \cite{bishop2006pattern}.
However, care is needed when using hierarchical methods in order
to ensure that discretization invariance is not 
lost \cite{agapiou2014analysis}.  In this paper we demonstrate 
how hierarchical methods can be employed in the context of 
discretization-invariant MCMC methods for Bayesian level
set inversion.

\begin{figure*}
\begin{center}
\includegraphics[width=0.9\linewidth]{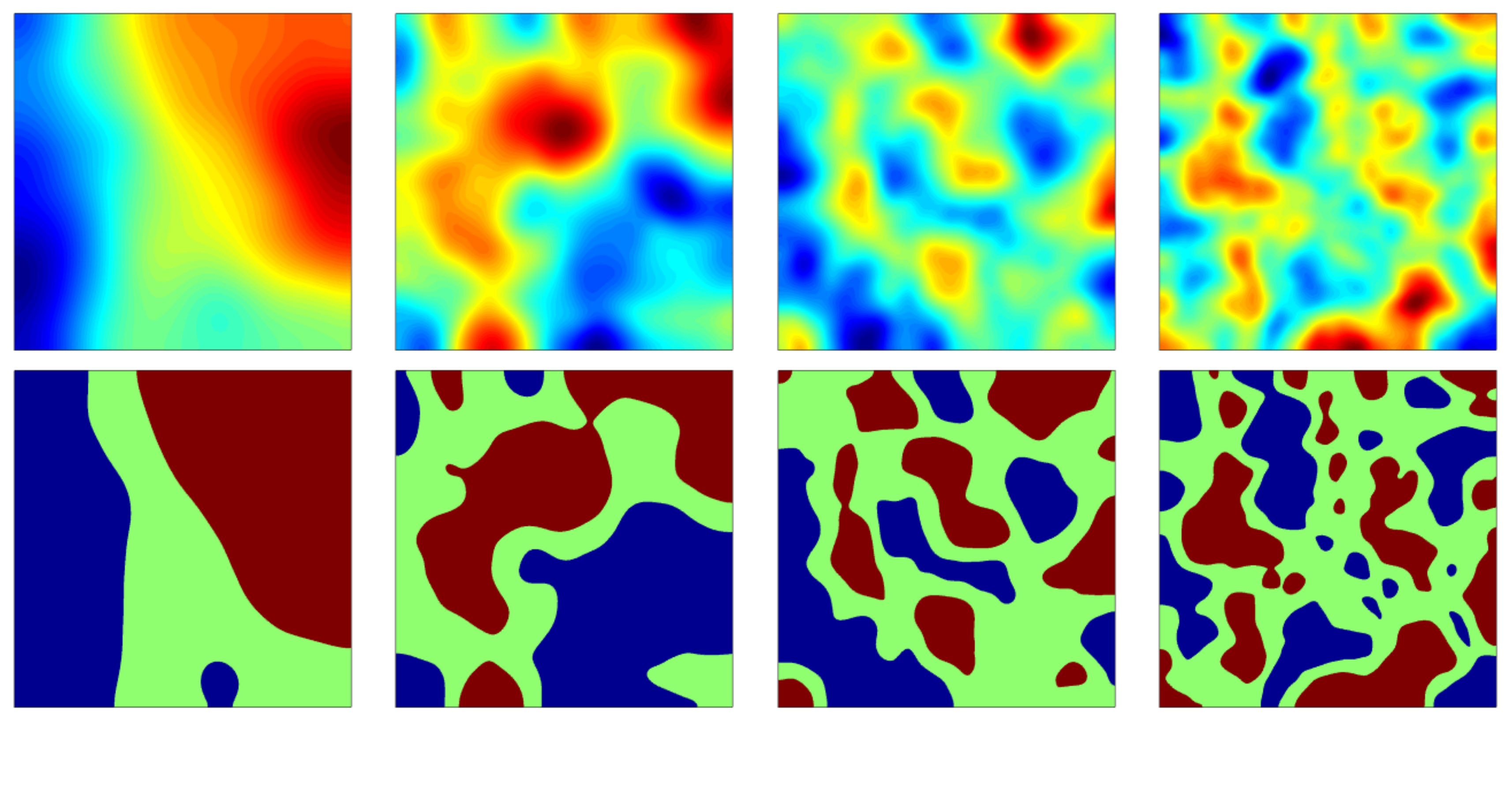}
\caption{Four continuous scalar fields (top) and the corresponding ternary fields formed by thresholding these fields at two levels (bottom). The smooth 
function in the top row is known as the {\em level-set function} and is used
in the inversion procedure. The discontinuous function in the
bottom row is the physical unknown.}
\label{fig:idea}
\end{center}
\end{figure*}

\subsection{Key Contributions of the Paper}
\label{ssec:key}

\mmd{The key contribution of this paper is in computational statistics: we develop a Metropolis Hastings method with mesh-independent mixing properties that makes an order of magnitude of improvement in the Bayesian level set method as introduced in \cite{levelset}}.

Study of Figure \ref{fig:idea} suggests that the ability 
of the level set representation to accurately reconstruct
piecewise continuous fields depends on two important scale parameters:

\begin{itemize}

\item the length-scale of the level set function, and its relation to
the typical separation between discontinuities;

\item the amplitude-scale of the level set function, and its relation to
the levels used for thresholding.

\end{itemize}

If these two scale parameters are not set correctly then MCMC methods
to determine the level set function from data can perform poorly. This
immediately suggests the idea of using hierarchical Bayesian methods in
which these  parameters are learned from the data. However there is a
second consideration which interacts with this discussion. From the
work of Tierney \cite{tierney} it is known that absolute continuity of
certain measures arising in the definition of Metropolis-Hastings methods 
is central \mmd{to} their well-definedness, and hence to discretization
invariant MCMC methods \cite{Cotter2013}. \mmd{In fact it appears algorithms defined on infinite dimensional spaces have spectral gaps that are bounded independently of the mesh, and so their convergence rates are bounded below in the limit \cite{spectralgap}}.
The key contribution of our paper
is to show how enforcing absolute continuity links the two scale 
parameters, and hence leads to the construction of a hierarchical 
Bayesian level set method with a {\em single scalar} 
hierarchical parameter which deals with the scale and absolute continuity 
issues simultaneously, resulting in effective sampling algorithms. 

The hierarchical parameter is an inverse length-scale within a Gaussian
random field prior for the level set function. In order to preserve
absolute continuity of different priors on the level set function as the
length-scale parameter varies, and relatedly to make
well-defined MCMC methods, the mean square amplitude of this Gaussian random 
field must decay proportionally to a power of the inverse length-scale. 
It is thus natural that the level values used for thresholding
should obey this power law relationship with respect 
to the hierarchical parameter.  As a consequence the 
likelihood depends on the hierarchical parameter,
leading to a novel form of posterior distribution.

We construct this posterior distribution and demonstrate how to sample from it
using a Metropolis-within-Gibbs algorithm which alternates between
updating the level set function and the inverse length scale. 
As a second contribution of the paper, we
demonstrate the applicability of the algorithm 
on three inverse problems, by means of simulation studies.
The first concerns reconstruction of a ternary piecewise constant field
from a finite noisy set of point measurements\mmd{: in this context, the Bayesian level set method is very closely related to a spatial probit model \cite{rasmussen-williams}.} \mi{This relation is discussed in in subsection \ref{ssec:probit}.} The other two concern
reconstruction of the coefficient of a divergence form
elliptic PDE from measurements
of its solution; in particular, groundwater flow (in which measurements
are made in the interior of the domain) and EIT (in which measurements
are made on the boundary). 

\subsection{Structure of the Paper}
\label{ssec:structure}

In section \ref{sec:P} we describe a family of prior distributions on the 
level set function, indexed by an inverse length scale parameter, which remain 
absolutely continuous with respect to one another when we vary this parameter; 
we then place a hyper-prior on this parameter. We describe an appropriate 
level set map, dependent on the length-scale parameter because length
and amplitude scales are intimately connected through absolute continuity
of measures, to transform these fields into piecewise constant ones, 
and use this level set map in the construction of the likelihood. We end by 
showing existence and well-posedness
of the posterior distribution on the level set function 
and the inverse length scale parameter.
In section \ref{sec:M} we describe a Metropolis-within-Gibbs MCMC algorithm 
for sampling the posterior distribution, taking advantage of existing 
state-of-the-art function space MCMC, and the absolute continuity of our 
prior distributions with respect to changes in the inverse length
scale parameter, established in the previous section.
Section \ref{sec:N} contains numerical experiments for three different 
forward models: a linear map comprising pointwise observations, 
groundwater flow and EIT; these
illustrate the behavior of the algorithm and, in particular, demonstrate
significant improvement with respect to 
non-hierarchical Bayesian level set inversion.

\section{Construction of the Posterior}
\label{sec:P}
In subsection \ref{ssec:prior} we recall the definition of the 
Whittle-Mat\'ern covariance functions, and define a related family 
of covariances parametrized by an inverse length scale parameter $\tau$. 
We use these covariances to define our prior on the level set function $u$, 
and also place a hyperprior on the parameter $\tau$, yielding a prior 
$\mathbb{P}(u,\tau)$ on a product space. In subsection \ref{ssec:like} 
we construct the level set map, taking into account the amplitude
scaling of prior samples with $\tau$, and incorporate this into the 
forward map. The inverse problem is formulated, and the resulting 
likelihood $\mathbb{P}(y|u,\tau)$ is defined. Finally in 
subsection \ref{ssec:post} we construct the posterior $\mathbb{P}(u,\tau|y)$ 
by combining the prior $\mathbb{P}(u,\tau)$ and likelihood 
$\mathbb{P}(y|u,\tau)$ using Bayes' formula.  
Well-posedness of this posterior is established.

\subsection{Prior}
\label{ssec:prior}

As discussed in the introduction it can be important, within the
context of Bayesian level set inversion, to attempt to learn the
length-scale of the level set function whose level sets determine
interfaces in piecewise continuous reconstructions. This is because
we typically do not know a-priori the typical separation of interfaces. 
It is also computationally expedient to work with Gaussian random 
field priors for the level set function, as demonstrated in 
\cite{levelset,eit}. A family of covariances parameterized by 
length scale is hence required.

A widely used family of distributions, 
allowing for control over sample regularity,
amplitude and length scale, are Whittle-Mat\'{e}rn distributions. 
These are a family of stationary Gaussian distributions with covariance function
\[
c_{\sigma,\nu,\ell}(x,y) = \sigma^2\frac{2^{1-\nu}}{\Gamma(\nu)}\left(\frac{|x-y|}{\ell}\right)^\nu K_\nu\left(\frac{|x-y|}{\ell}\right)
\] 
where $K_\nu$ is the modified Bessel function of the second kind of order $\nu$ \cite{matern1,matern2}. These covariances interpolate between exponential covariance, for $\nu = 1/2$, and Gaussian covariance, for $\nu\rightarrow\infty$. As a consequence, 
the regularity of samples increases as the parameter $\nu$ increases. The parameter $\ell > 0$ acts as a characteristic length scale (sometimes referred to as the spatial
range) and $\sigma$ as an amplitude scale 
($\sigma^2$ is sometimes referred to as the marginal variance).
On $\mathbb{R}^d$, samples from a Gaussian distribution with covariance function $c_{\sigma,\nu,\ell}$ correspond to the solution of a particular 
stochastic partial differential equation (SPDE). This SPDE can be derived 
using the Fourier transform and the spectral representation of covariance 
functions --  the paper \cite{whittlematern} derives 
the appropriate SPDE for the covariance function above: 
\begin{align}
\label{eq:spde}
\frac{1}{\sqrt{\beta\ell^d}}(I-\ell^2\Delta)^{(\nu+d/2)/2}v = W
\end{align}
where $W$ is a white noise on $\mathbb{R}^d$, and
\[
\beta = \sigma^2 \frac{2^d\pi^{d/2}\Gamma(\nu+d/2)}{\Gamma(\nu)}.
\]
Computationally, implementation of this SPDE approach requires
restriction  to a bounded subset $D\subseteq\mathbb{R}^d$, and 
hence the provision of boundary conditions for the SPDE in order to obtain a unique solution. Choice of these boundary conditions may significantly affect the autocorrelations near the boundary. The effects for different boundary conditions are discussed in \cite{whittlematern}. Nonetheless,
the computational expediency of the SPDE formulation makes the approach
very attractive for applications and, if necessary, boundary effects
can be ameliorated by generating the random fields on larger domains 
which are a superset of the domain of interest.

From (\ref{eq:spde}) it can be seen that the covariance operator corresponding to the covariance function $c_{\sigma,\nu,\ell}$ is given by
\begin{align}
\label{wmcov}
\mathcal{D}_{\sigma,\nu,\ell} = \beta\ell^d(I - \ell^2\Delta)^{-\nu-d/2}.
\end{align}
The fact that the scalar multiplier in front of the covariance operator
$\mathcal{D}_{\sigma,\nu,\ell}$ changes with the length-scale means that
the family of measures $\{N(0,\mathcal{D}_{\sigma,\nu,\ell})\}_{\ell}$, 
for fixed $\sigma$ and $\nu$, are mutually singular. This leads to
problems when trying to design hierarchical methods based around 
these priors. We hence 
work instead with the modified covariances
\[
\mathcal{C}_{\alpha,\tau} = (\tau^2 I - \Delta)^{-\alpha}
\]
where $\tau = 1/\ell > 0$ now represents an \emph{inverse} length scale, and $\alpha = \nu+d/2$ still controls the sample regularity. To be
concrete we will always assume that the domain of the Laplacian is chosen
so that $\mathcal{C}_{\alpha,\tau}$ is well-defined for all $\tau \ge 0$;
for example we may choose a periodic box, with domain restricted to functions
which integrate to zero over the box, Neumann boundary conditions on a box, 
again with domain restricted to functions which integrate to zero over the box, 
or Dirichlet boundary conditions. We have the following theorem concerning the family of Gaussians \mi{$\{N(0,\mathcal{C}_{\alpha,\tau})\}_{\tau\geq 0}$}, proved in \mmd{the} Appendix.

\begin{theorem}
\label{t:wmequiv}
Let $D = \mathbb{T}^d$ be the $d$-dimensional torus, and fix $\alpha > 0$. Define the family of Gaussian measures \mi{$\mu^{\tau}_0 = N(0,\mathcal{C}_{\alpha,\tau})$}, $\tau \geq 0$. Then
\begin{enumerate}[(i)]
\item for $d \leq 3$, the $\{\mu_0^\tau\}_{\tau \geq 0}$ are mutually equivalent; 
\item if $u \sim \mu_0^\tau$, then $\mu_0^\tau$-a.s. we have \mi{$u \in H^s(D)$} and 
\mi{$u \in C^{\lfloor s \rfloor,s-\lfloor s \rfloor}(D)$} for all $s< \alpha - d/2.$ 
\footnote{i.e. the function has $s$ weak (possibly fractional) 
derivatives in the Sobolev sense, and the 
$\lfloor s \rfloor^{th}$ classical derivative is H\"older with
exponent $s-\lfloor s \rfloor$;}
\mi{\item if $u \sim \mu_0^\tau$ and $v \sim N(0,\mathcal{D}_{\sigma,\nu,\ell})$, then
\[
\mathbb{E}\|u\|^2 \propto \tau^{d-2\alpha}\cdot\mathbb{E}\|v\|^2
\]
with constant of proportionality independent of $\tau.$}
\end{enumerate}
\end{theorem}

\begin{remark}
\label{rem:thm1}
\begin{enumerate}[(a)]
\item Proof of this theorem is driven by the smoothness of the eigenfunctions of the
Laplacian subject to periodic boundary conditions, together with
the growth of the eigenvalues, which is like $j^{2/d}.$ These
properties extend to Laplacians on more general domains and with more general
boundary conditions, and to Laplacians with lower order perturbations,
and so the above result still holds in these cases. 
For discussion of this in relation to (ii) see \cite{lecturenotes};
for parts (i) and (iii)
the reader can readily extend the proof given in the Appendix.

\mmd{\item The proportionality in part (iii) above could be simplified if it were the case that \mi{$\mathbb{E}\|v\|^2$} were independent of $\tau$. However since we restrict to a bounded domain $D\subset\mathbb{R}^d$, boundary effects mean that this isn't necessarily true. Neumann boundary conditions for example inflate the variance up to a distance of approximately $\ell\sqrt{8\nu} = \sqrt{8\nu}/\tau$ from the boundary \cite{lindgren-rue15}. Nonetheless, at points $x \in D$ sufficiently far away from the boundary we have \mi{$\mathbb{E}|v(x)|^2\approx\sigma^2$} independently of $x$. At these points we would hence expect that, for $u \sim \mu_0^\tau$,
\mi{\[
\mathbb{E}|u(x)|^2 \propto \tau^{d-2\alpha}.
\]}Note also that numerically, we may produce samples on a larger domain $D^*$ that contains the domain of interest $D$, in order to minimize the boundary effects within $D$.$\hfill\qed$}

\end{enumerate}
\end{remark}

Let \mi{$X = C(D)$} denote the space of continuous real-valued
functions on domain $D$. In what follows we will always assume that $\alpha - d/2 > 0$ in order that the measures have samples in $X$ almost-surely. Additionally we shall write $\mathcal{C}_\tau$ in place of $\mathcal{C}_{\alpha,\tau}$ when the parameter $\alpha$ is not of interest.

In subsection \ref{ssec:like}, we pass the inverse length scale parameter $\tau$ to the forward map and treat it as an additional unknown in the inverse problem. We therefore require a joint prior $\mathbb{P}(u,\tau)$ on both the \mmd{level set} field and on $\tau$. We will treat $\tau$ as a hyper-parameter, so that $\mathbb{P}(u,\tau)$ takes the form $\mathbb{P}(u,\tau) = \mathbb{P}(u|\tau)\mathbb{P}(\tau)$. Specifically, we will take the conditional distribution $\mathbb{P}(u|\tau)$ to be given by \mi{$\mu_0^\tau = N(0,\mathcal{C}_{\tau})$}, and the hyper-prior $\mathbb{P}(\tau)$ to be any probability measure $\pi_0$ on $\mathbb{R}^+$, the set of positive reals; in practice it will always have a Lebesgue density
on $\mathbb{R}^+$. The joint prior $\mu_0$ on $X\times\mathbb{R}^+$ is therefore assumed to be given by
\begin{align}
\label{eq:prior}
\mu_0(\dee u,\dee \tau) = \mu_0^\tau(\dee u)\pi_0(\dee \tau).
\end{align}
\mi{Non-zero means could also be considered via a change of coordinates.} Discussion of prior choice for the hierarchical parameters in
latent Gaussian models may be found in \cite{fuglstad2015interpretable}.

\subsection{Likelihood}
\label{ssec:like}
In the previous subsection we defined a prior distribution $\mu_0$ on $X\times\mathbb{R}^+$. We now define a way of constructing  a piecewise constant field from a sample $(u,\tau)$. In \cite{levelset}, where the Bayesian level set method was introduced, the piecewise constant field was constructed purely as a function of $u$ as follows. Let $n \in \mathbb{N}$ and fix constants $-\infty = c_0 < c_1 < \ldots < c_n = \infty$. Given $u \in X$, define $D_i(u)\subseteq D$ by
\[
D_i(u) = \{x \in D\,|\, c_{i-1} \leq u(x) < c_i\},\;\;\;i=1,\ldots,n
\]
so that\footnote{\mi{For any subset $A\subset\mathbb{R}^d$ we will denote by $\overline{A}$ its closure in $\mathbb{R}^d$.}} $\overline{D} = \bigcup_{i=1}^n \overline{D}_i(u)$ and $D_i(u)\cap D_j(u) = \varnothing$ for $i\neq j$, $i,j \geq 1$. 
Then given $\kappa_1,\ldots,\kappa_n \in \mathbb{R}$, define the map $F:X\rightarrow Z$ by
\begin{align}
\label{eq:lvl_original}
F(u) = \sum_{i=1}^n \kappa_i\mathds{1}_{D_i(u)}.
\end{align}
\mmd{Thus $F$ maps the level set field to the geometric field, which is the field of interest, even though inference is performed on the level set field}. We may take \mi{$Z = L^p(D)$}, the space of $p$-integrable functions 
on $D$, for any $1\leq p \le \infty$. \mmd{$F(u)$ then defines a piecewise constant function on $D$; the interfaces defined by the jumps are given by the level sets $\{x\in D\,|\, u(x) = c_i\}$.}
\mmd{
\begin{remark}
One of the constraints of this construction, discussed in \cite{levelset}, is that in order for $F(u)$ to pass from $\kappa_i$ to $\kappa_j$, it must pass through all of $\kappa_{i+1},\ldots,\kappa_{j-1}$ first. Thus this construction cannot represent, for example, a triple junction. \mi{This also means that that it must be known a priori that, for example, level $i$ is typically found near levels $i-1$ and $i+1$, but unlikely to be found near levels $i+3$ or $i+4$. This is potentially a significant constraint;} we discuss how this may be dealt with in the conclusions. $\hfill\qed$
\end{remark}
}
This construction is effective for a 
fixed value of $\tau$, but in light of Theorem \ref{t:wmequiv}(iii), 
the amplitude of samples from \mi{$N(0,\mathcal{C}_{\alpha,\tau})$},
varies with $\tau$. More specifically, since $d-2\alpha < 0$ by assumption, 
samples will decay towards \mi{zero} as $\tau$ increases.  For this reason,
employing fixed levels $\{c_i\}_{i=0}^n$ and then changing the value of $\tau$
during a sampling method may render the levels  out of reach. We can 
compensate for this by allowing the levels to change with $\tau$, so that 
they decay towards \mi{zero} at the same rate as the samples.

From Theorem \ref{t:wmequiv}(iii) \mmd{and Remark \ref{rem:thm1}(b)} we deduce that samples $u$ from \mi{$N(0,\mathcal{C}_{\alpha,\tau})$ decay towards zero} at a rate of \mmd{approximately} $\tau^{d/2-\alpha}$ with respect to $\tau$. This suggests allowing for the following dependence of the levels
on the length scale parameter $\tau$:  
\mi{\begin{align}
\label{eq:ctau}
c_i(\tau) = \tau^{d/2-\alpha}c_i,\;\;\;i=1,\ldots,n.
\end{align}}In order to update these levels, we must pass the parameter $\tau$ to the level set map $F$. We therefore \mi{redefine the} level set map $F:X\times\mathbb{R}^+\rightarrow Z$ as follows. Let $n \in \mathbb{N}$, fix initial levels $-\infty = c_0 < c_1 < \ldots < c_n = \infty$ and define $c_i(\tau)$ by (\ref{eq:ctau}) for $\tau > 0$.  Given $u \in X$ and $\tau > 0$, define $D_i(u,\tau)\subseteq D$ by
\begin{align}
\label{eq:di}
D_i(u,\tau) = \{x \in D\;|\;c_{i-1}(\tau) \leq u(x) &< c_i(\tau)\},\;\;\;i=1,\ldots, n,
\end{align}
so that $\overline{D} = \bigcup_{i=1}^n \overline{D}_i(u,\tau)$ and $D_i(u,\tau)\cap D_j(u,\tau) = \varnothing$ for $i\neq j$, $i,j \geq 1$. Now given $\kappa_1,\ldots,\kappa_n \in \mathbb{R}$, we define the map $F:X\times\mathbb{R}^+\rightarrow Z$ by
\begin{align}
\label{lvlsetmap}
F(u,\tau) = \sum_{i=1}^n \kappa_i\mathds{1}_{D_i(u,\tau)}.
\end{align}

We can now \mmd{define the likelihood}. Let $Y = \mathbb{R}^J$ be the data space, and let $S:Z\rightarrow Y$ be a forward operator. 
Define $\mathcal{G}:X\times\mathbb{R}^+\rightarrow Y$ 
by $\mathcal{G} = S\circ F$. Assume we have data 
$y \in Y$ arising from observations of some 
$(u,\tau) \in X\times\mathbb{R}^+$ under $\mathcal{G}$, 
corrupted by Gaussian noise $\eta \sim \mathbb{Q}_0 := N(0,\Gamma)$ on $Y$:
\begin{equation}
y = \mathcal{G}(u,\tau) + \eta.
\label{eq:be}
\end{equation}
We now construct the likelihood $\mathbb{P}(y|u,\tau)$. In the Bayesian formulation, we place a prior $\mu_0$ of the form (\ref{eq:prior}) on the pair $(u,\tau)$. Assuming $\mathbb{Q}_0$ is independent of $\mu_0$, the conditional distribution $\mathbb{Q}_{u,\tau}$ of $y$ given $(u,\tau)$ is given by
\begin{align}
\label{eq:qudensity}
\frac{\dee \mathbb{Q}_{u,\tau}}{\dee\mathbb{Q}_0}(y) =\exp\bigg(-\Phi(u,\tau;y) + \frac{1}{2}|y|_\Gamma^2\bigg)
\end{align}
where the potential (or negative log-likelihood) $\Phi:X\times\mathbb{R}^+\rightarrow\mathbb{R}$ is defined by
\begin{align}
\label{eq:phi}
\Phi(u,\tau;y) = \frac{1}{2}|y - \mathcal{G}(u,\tau)|_\Gamma^2.
\end{align}
and $|\cdot|_\Gamma := |\Gamma^{-1/2}\cdot|$.

Denote $\mathrm{Im}(F) \subseteq Z$ the image of $F:X\times\mathbb{R}^+\rightarrow Z$. In what follows we make the following assumptions on $S:Z\rightarrow Y$. 
\begin{assumptions}
\label{ass:forward}

\begin{enumerate}[(i)]
\item $S$ is continuous on $\mathrm{Im}(F)$.
\item For any $r > 0$ there exists $C(r) > 0$ such that for any $z \in \mathrm{Im}(F)$ with $\|z\|_{L^\infty} \leq r$, $|S(z)| \leq C(r)$.
\end{enumerate}
\end{assumptions}
In the next subsection we show that, under the above assumptions, the posterior distribution $\mu^y$ of $(u,\tau)$ given $y$ exists, and study its properties.

\subsection{Posterior}
\label{ssec:post}

Bayes' theorem provides a way to construct the posterior distribution $\mathbb{P}(u,\tau|y)$ using the ingredients of the prior $\mathbb{P}(u,\tau)$ and the likelihood $\mathbb{P}(y|u,\tau)$ from the previous two subsections. Informally 
we have
\begin{align*}
\mathbb{P}(u,\tau|y) &\propto \mathbb{P}(y|u,\tau)\mathbb{P}(u,\tau)\\
&\propto \exp\left(-\Phi(u,\tau;y)\right)\mu_0^\tau(u)\pi_0(\tau)
\end{align*}
after absorbing $y-$dependent constants from the likelihood into the 
normalization constant. In order to make this formula rigorous some care must 
be taken, since $\mu_0^\tau$ does not admit a Lebesgue density. The 
following is proved in the Appendix.

\begin{theorem}
\label{t:2}
Let $\mu_0$ be given by (\ref{eq:prior}), $y$ by \eqref{eq:be}
and $\Phi$ be given by (\ref{eq:phi}). Let Assumptions \ref{ass:forward} hold. If $\mu^y(du,d\tau)$ is the regular conditional probability
measure on $(u,\tau)|y$, then $\mu^y \ll \mu_0$ with Radon-Nikodym derivative
\[
\frac{\dee \mu^y}{\dee \mu_0}(u,\tau) = \frac{1}{Z}\exp\big(-\Phi(u,\tau;y)\big)
\]
where, for $y$ almost surely,
\[
Z:= \int_{X\times\mathbb{R}^+}\exp\big(-\Phi(u,\tau;y)\big)\,\mu_0(\dee u,\dee \tau) > 0.
\]
Furthermore $\mu^y$ is locally Lipschitz with respect to $y$, in the Hellinger distance: for all $y,y'$ with $\max\{|y|_\Gamma,|y'|_\Gamma\} < r$, there exists a $C = C(r) > 0$ such that
\[
d_{\mathrm{Hell}}(\mu^y,\mu^{y'}) \leq C|y-y'|_\Gamma.
\]
This implies that, for all $f \in L^2_{\mu_0}(X\times\mathbb{R}^+;E)$ 
for separable Banach space $E$,
\[
\|\mathbb{E}^{\mu^y}f(u,\tau) - \mathbb{E}^{\mu^{y'}}f(u,\tau)\|_E \leq C|y-y'|.
\]
\end{theorem}

To the best of our knowledge this form of Bayesian \mmd{posterior distribution},
in which the prior hyper-parameter appears in the likelihood because
it is natural to scale a thresholding function with that parameter, \mmd{for algorithmic reasons,}
is novel. A different form of thresholding is studied in the paper
\cite{bolin2015excursion} where boundaries defining regions in
which certain events occur with a specified (typically close to $1$)
probability is studied. 

\subsection{Relation to Probit Models}
\label{ssec:probit}

\mi{
The Bayesian level set method has a close relation with an ordered probit model in the case that the state space $X$ is finite dimensional. Suppose that $X = \mathbb{R}^N$, then neglecting the length scale parameter, the data $y_{\mathrm{level}}$ in the level set method is assumed to arise via 
\[
y_{\mathrm{level}} = \mathcal{G}(F(u)) + \eta,\;\;\;\eta \sim N(0,\Gamma)
\]
where $F$ denotes the original thresholding function as defined by (\ref{eq:lvl_original}). In an ordered probit model, the data $y_{\mathrm{prob}}$ is assumed to arise via\footnote{The thresholding function $F$ is defined pointwise, so can be considered to be defined on either $\mathbb{R}^N$ or $\mathbb{R}$, with $F(u)_n = F(u_n)$.}
\begin{align*}
y_{\mathrm{prob}} &= \mathcal{G}(z),\\
z_n &= F(u_n + \eps_n),\;\;\;\eps_n \sim N(0,1),\;\;\;n=1,\ldots,N.
\end{align*}
Note that in the case of probit, the noise is applied before the thresholding $F$ so that the geometric field takes values in the discrete set $\{\kappa_1,\ldots,\kappa_n\}$. In contrast in the case of the level set model the noise is applied after thresholding. If $\mathcal{G}$ is linear then the probit model results in categorical data, whilst in the level set case the data can take any real value. Depending on the forward model either probit or level set may be more appropriate: the former in cases where the data is genuinely discrete and interpolation between phases doesn't have a meaning, such as categorical data, and the latter when it is continuous, such as when corrupted by measurement noise. The two models could also be combined, which may be interesting in some applications. In the small noise limit the models are seen to be equivalent.
}

\mi{Placing a prior upon $u$ leads to a well-defined posterior distribution in both cases. Dimension-robust sampling of both distributions can be performed using a prior-reversible MCMC method, such as the preconditioned Crank-Nicolson (pCN) method \cite{Cotter2013}. The spatial version of probit, that is when $X$ is a function space rather than $\mathbb{R}^N$, is of interest to study further.}

\mi{Once we introduce the hierarchical length scale dependence, significant problems arise in terms of sampling the probit posterior in high dimensions, due to the issues associated with measure singularity discussed above. With the level set method it is possible to circumvent through the choice of prior and rescaling discussed in this section; a well-defined Metropolis-within-Gibbs sampling algorithm on function space is outlined in the next section.}

\section{MCMC Algorithm for Posterior Sampling}
\label{sec:M}

Having constructed the posterior distribution on $(u,\tau)|y$ we
are now faced with the task of sampling this probability distribution.
We will use the Metropolis-within-Gibbs formalism, as described in for 
example \cite{casella-robert}, section 10.3. This
algorithm constructs the Markov chain $(u^{(k)},\tau^{(k)})$ with
the structure
\begin{itemize}
\item $u^{(k+1)} \sim \bbK^{\tau^{(k)},y}(u^{(k)},\cdot)$,
\item $\tau^{(k+1)} \sim \bbL^{u^{(k+1)},y}(\tau^{(k)},\cdot)$,
\end{itemize} 
where $\bbK^{\tau,y}$ is a Metropolis-Hastings Markov kernel
reversible with respect to $u|(\tau,y)$ and
$\bbL^{u,y}$ is a Metropolis-Hastings Markov kernel
reversible with respect to $\tau|(u,y).$ 
The Metropolis-Hastings method is outlined in chapter 7 
of \cite{casella-robert}. See \cite{geirsson2015mcmc} for related blocking methodologies for
Gibbs samplers in the context of latent Gaussian models.

In defining the conditional distributions, and the
Metropolis methods to sample from them, a key design principle is
to ensure that all measures and algorithms are well-defined in 
the infinite-dimensional setting, so that the resulting algorithms 
are robust to mesh-refinement \cite{Cotter2013}.
This thinking has been behind the form of the prior and posterior
distributions developed in the previous section, as we now demonstrate.

In subsection \ref{ssec:pu} we define the kernel $\bbK^{\tau,y}$ 
and in subsection \ref{ssec:pt} we define the kernel $\bbL^{u,y}.$
Then in the final subsection \ref{ssec:alg} we put all these
building blocks together to specify the complete algorithm used.

\subsection{Proposal and Acceptance Probability for $u|(\tau,y)$}
\label{ssec:pu}

Samples from the distribution of $u|(\tau,y)$ can be produced using a 
pCN Metropolis Hastings method \cite{Cotter2013}, with proposal
and acceptance probability as follows:
\mi{\begin{enumerate}
\item Given $u$, propose 
\[
v = (1-\beta^2)^{1/2}u + \beta \xi,\;\;\;\xi\sim N(0,\mathcal{C}_\tau).
\]
\item Accept with probability 
\[
\alpha(u,v) = \min\big\{1,\exp\big(\Phi(u,\tau;y)-\Phi(v,\tau;y)\big)\big\}
\]
or else stay at $u$.
\end{enumerate}}

\subsection{Proposal and Acceptance Probability for $\tau|(y,u)$}
\label{ssec:pt}

Producing samples of $\tau |(u,y)$ is more involved, since we must first make sense of this conditional distribution. To do this, define the three measures 
$\eta_0$, $\nu_0$, and $\nu$ on $X\times\mathbb{R}^+\times Y$ by
\begin{align*}
\eta_0(\dee u, \dee \tau,\dee y) &= \mu_0^0(\dee u)\pi_0(\dee \tau)\mathbb{Q}_0(\dee y),\\
\nu_0(\dee u, \dee \tau,\dee y) &= \mu_0^\tau(\dee u)\pi_0(\dee \tau)\mathbb{Q}_0(\dee y),\\
\nu(\dee u,\dee \tau,\dee y) &= \mu_0^\tau(\dee u)\pi_0(\dee \tau)\mathbb{Q}_{u,\tau}(\dee y).
\end{align*}
Here $\mathbb{Q}_0 = N(0,\Gamma)$ is the distribution of the noise, and $\mathbb{Q}_{u,\tau}$ is as defined in (\ref{eq:qudensity}). Then we have the chain of 
absolute continuities $\nu \ll \nu_0 \ll \eta_0$, with
\begin{align*}
\frac{\dee \nu_0}{\dee \eta_0}(u,\tau,y) &= \frac{\dee \mu_0^\tau}{\dee \mu_0^0}(u) =: L(u,\tau),\\
\frac{\dee \nu}{\dee \nu_0}(u,\tau,y) &= \frac{\dee \mathbb{Q}_{u,\tau}}{\dee \mathbb{Q}_0}(y) = \exp\left(-\Phi(u,\tau;y)+\frac{1}{2}|y|_\Gamma^2\right),
\end{align*}
and so by the chain rule we have $\nu \ll \eta_0$ and
\[
\frac{\dee \nu}{\dee \eta_0}(u,\tau,y) = \frac{\dee \mathbb{Q}_{u,\tau}}{\dee \mathbb{Q}_0}(y)\cdot\frac{\dee \mu_0^\tau}{\dee \mu_0^0}(u) =: \varphi(u,\tau,y).
\]
We use the conditioning lemma, Theorem 3.1 in \cite{lecturenotes}, to prove the existence of the desired conditional distribution.

\begin{theorem}
Assume that \mmd{$\Phi:X\times \mathbb{R}^+\times Y\rightarrow\mathbb{R}$ is $\eta_0$ measurable and $\eta_0$-a.s. finite. Assume also that, 
for $(u,y)$ $\mu_0^0\times\mathbb{Q}_0$-a.s.},
\[
Z_\pi := \int_{\mathbb{R}^+} \exp\big(-\Phi(u,\tau;y)\big)L(u,\tau)\,\pi_0(\dee \tau) > 0.
\]
Then the regular
conditional distribution of $\tau |(u,y)$ exists under $\nu$, and is denoted by $\pi^{u,y}$. Furthermore, $\pi^{u,y} \ll \pi_0$ and, for $(u,y)$ $\nu$-a.s,
\[
\frac{\dee \pi^{u,y}}{\dee \pi_0}(\tau) = \frac{1}{Z_\pi} \exp\big(-\Phi(u,\tau;y)\big)L(u,\tau).
\]
\end{theorem}

\begin{proof}
The conditional random variable $\tau |(u,y)$ exists under $\eta_0$, and its distribution is just $\pi_0$ since $\eta_0$ is a product measure. Theorem 3.1 in \cite{lecturenotes} then tells us that the conditional random variable $\tau |(u,y)$ exists under $\nu$. We denote its distribution $\pi^{u,y}$. Define
\begin{align*}
c(u,y) &= \int_{\mathbb{R}^+}\varphi(u,\tau,y)\pi_0(\dee \tau)\\ 
&= \exp\left(\frac{1}{2}|y|_\Gamma^2\right)\int_{\mathbb{R}^+}\exp\big(-\Phi(u,\tau;y)\big)L(u,\tau)\pi_0(\dee \tau).
\end{align*}
Now since $\exp\big(\frac{1}{2}|y|_\Gamma^2\big) \in (0,\infty)$ $\mu_0^0\times\mathbb{Q}_0$-a.s., we deduce that $c(u,y) > 0$ $\mu_0^0\times\mathbb{Q}_0$-a.s. by the $\mu_0^0\times\mathbb{Q}_0$-a.s. positivity of $Z_\pi$. By the absolute continuity $\nu \ll \eta_0$, we deduce that $c(u,y) > 0$ $\nu$-a.s. Therefore, again by Theorem 3.1 in \cite{lecturenotes}, we have $\pi^{u,y}\ll\pi_0$ and, for $(u,y)$ $\nu$-a.s.,
\begin{align*}
\frac{\dee \pi^{u,y}}{\dee \pi_0}(\tau) &= \frac{1}{c(u,y)}\varphi(u,\tau,y)\\
&= \frac{1}{Z_\pi} \exp\big(-\Phi(u,\tau;y)\big)L(u,\tau).
\end{align*}\qed
\end{proof}

\begin{remark}
Above we have used $\mu_0^0$ as a reference measure, and the
function $L(u,\tau)$ enters our expression for the posterior. 
But any $\mu_0^\lambda$ will suffice since the entire family
of measures $\{\mu_0^\tau\}_{\tau \ge 0}$ are equivalent
to one another. A straightforward calculation with the 
chain rule gives
\begin{align*}
\frac{\dee \pi^{u,y}}{\dee \pi_0}(\tau) 
&= \frac{1}{Z_{\pi,\lambda}}\frac{\dee \mu_0^\tau}{\dee \mu_0^\lambda}(u)\exp\big(-\Phi(u,\tau;y)\big)\\
&:=\frac{1}{Z_{\pi,\lambda}}L_{\lambda}(u,\tau)\exp\big(-\Phi(u,\tau;y)\big).
\end{align*}
$\hfill\qed$
\end{remark}

We now wish to sample from $\pi^{u,y}$ using a Metropolis-Hastings algorithm. We assume from now on that $\pi_0$ admits a Lebesgue density, so that $\pi^{u,y}$ also admits a Lebesgue density. Abusing notation
and using $\pi^{u,y},\pi_0$ to denote Lebesgue densities as well as
the corresponding measures we have
\[
\pi^{u,y}(\tau) \propto \exp\big(-\Phi(u,\tau;y)\big)L(u,\tau)\pi_0(\tau).
\]
Take a proposal kernel $Q(\tau,\dee\gamma) = q(\tau,\gamma)\,\dee\gamma$. Define the two measures $\rho, \rho^T$ on $(\mathbb{R}\times\mathbb{R},\mathcal{B}(\mathbb{R})\otimes\mathcal{B}(\mathbb{R}))$ by
\begin{align*}
\rho(\dee \tau, \dee \gamma) &= \pi^{u,y}(\dee \tau)Q(\tau, \dee \gamma)\\
&\propto \exp\big(-\Phi(u,\tau;y)\big)L(u,\tau)\pi_0(\tau)q(\tau,\gamma)\,\dee\tau\dee\gamma,\\
\rho^T(\dee\tau,\dee\gamma) &= \mu(\dee\gamma,\dee\tau).
\end{align*}
Then under appropriate conditions on $\pi_0$ and $q$, these two measures are equivalent. Define $r(\tau,\gamma)$ to be the Radon-Nikodym derivative
\begin{align*}
r(\tau,\gamma) &:=\frac{\dee \rho^T}{\dee \rho}(\tau,\gamma)\\
&= \exp\big(\Phi(u,\tau;y)-\Phi(u,\gamma;y)\big)\cdot\frac{\dee \mu_0^\gamma}{\dee \mu_0^\tau}(u)\cdot \frac{\pi_0(\gamma)q(\gamma,\tau)}{\pi_0(\tau)q(\tau,\gamma)}.
\end{align*}
The general form of the Metropolis-Hastings algorithm, as for example given in \cite{tierney}, says that we produce samples from $\pi^{u,y}$ by iterating the follow two steps:
\begin{enumerate}
\item Given $\tau$, propose $\gamma \sim Q(\tau,\dee\gamma)$.
\item Accept with probability $\displaystyle \alpha(\tau,\gamma) = \min\big\{1, r(\tau,\gamma)\big\}$, or else stay at $\tau$.
\end{enumerate}
In order to implement this algorithm, we need an expression for the Radon-Nikodym derivative $\frac{\dee \mu_0^\gamma}{\dee \mu_0^\tau}(u)$. \mmd{Denote by $\{\lambda_j(\tau)\}_{j\geq 1}$ the eigenvalues of the covariance $\mathcal{C}_\tau$, and $\{\varphi_j\}_{j\geq 1}$ their corresponding eigenvectors. Note that because of the structure of the family $\{\mathcal{C}_\tau\}_{\tau\geq 0}$, the eigenvectors are independent of $\tau$.} Using Proposition \ref{prop:fhdensity}, we see that
{\fontsize{0.3cm}{0.5cm}
\mi{\begin{align}
\label{trueratio}
\frac{\dee \mu_0^\gamma}{\dee \mu_0^\tau}(u) &= \prod_{j=1}^\infty\frac{\lambda_j(\tau)^{1/2}}{\lambda_j(\gamma)^{1/2}}\times\exp\Bigg(\frac{1}{2}\sum_{j=1}^\infty\bigg(\frac{1}{\lambda_j(\tau)} - \frac{1}{\lambda_j(\gamma)}\bigg)\langle u,\varphi_j\rangle^2\Bigg)\\
&\notag= \exp\Bigg(\frac{1}{2}\Bigg[\sum_{j=1}^\infty\left(\frac{1}{\lambda_j(\tau)} - \frac{1}{\lambda_j(\gamma)}\right)\langle u,\varphi_j\rangle^2 + \log\left(\frac{\lambda_j(\tau)}{\lambda_j(\gamma)}\right)\Bigg]\Bigg).
\end{align}
}}

From Theorem \ref{t:wmequiv} we know that $\mu_0^\tau$ and $\mu_0^\gamma$ are equivalent, and so it must be the case that the expressions for the derivative above are almost-surely finite. However this is not immediately clear from inspection of the expression; thus we provide some intuition about why it is so in the following theorem. The proof is given in the Appendix.

\mi{\begin{theorem}
\label{t:exponent}
Assume that $u \sim N(0,\mathcal{C}_0)$. Then for each $\tau > 0$,
\begin{enumerate}[(i)]
\item {\small $\displaystyle\sum_{j=1}^\infty\left(\frac{1}{\lambda_j(\tau)} - \frac{1}{\lambda_j(0)}\right)\langle u,\varphi_j\rangle^2$} is almost-surely finite if and only if $d = 1$; and
\item {\small $\displaystyle\sum_{j=1}^\infty\left[\left(\frac{1}{\lambda_j(\tau)} - \frac{1}{\lambda_j(0)}\right)\langle u,\varphi_j\rangle^2 + \log\left(\frac{\lambda_j(\tau)}{\lambda_j(0)}\right)\right]$} is almost-surely finite if $d \leq 3$.
\end{enumerate}
\end{theorem}}

A consequence of part (i) of this result is that in dimensions $2$ and $3$, both the product and the sum in (\ref{trueratio}) diverge, 
despite the whole expression being finite.
This means that care is required when numerically implementing 
the Gibbs update of $\tau.$

\subsection{The Algorithm}
\label{ssec:alg}
Putting the theory above together, we can write down a Metropolis-within-Gibbs algorithm for sampling the posterior distribution. Recall that we assumed the proposal kernel $Q$ admitted a Lebesgue density $q$: $Q(\tau,\dee \gamma) = q(\tau,\gamma)\dee\gamma$.

Let $\{\lambda_j(\tau),\varphi_j\}_{j \geq 1}$ denote the eigenbasis associated with $\mathcal{C}_\tau$. Define
\mi{\begin{align*}
w(\tau,\gamma) &= \exp\Bigg(\frac{1}{2}\sum_{j=1}^\infty\bigg[\left(\frac{1}{\lambda_j(\tau)} - \frac{1}{\lambda_j(\gamma)}\right)\langle u,\varphi_j\rangle^2+ \log\left(\frac{\lambda_j(\tau)}{\lambda_j(\gamma)}\right)\bigg]\Bigg)
\end{align*}}and set
\begin{align*}
\alpha^\tau(u,v) &= \min\Big\{1,\exp\big(\Phi(u,\tau;y) - \Phi(v,\tau;y)\big)\Big\},\\
\alpha^u(\tau,\gamma) &= \min\bigg\{1,\exp\big(\Phi(u,\tau;y) - \Phi(u,\gamma;y)\big)\cdot w(\tau,\gamma)\cdot \frac{\pi_0(\tau)q(\tau,\gamma)}{\pi_0(\gamma)q(\gamma,\tau)}\bigg\}.
\end{align*}
Fix jump parameter $\beta \in (0,1]$, and generate $\{u^{(k)},\tau^{(k)}\}_{k\geq 0}$ as follows:

\begin{algorithm}[ht]
\begin{singlespace}
\caption{Metropolis-within-Gibbs}
\end{singlespace}
\begin{enumerate}
\item Set $k = 0$ and pick initial state $(u^{(0)},\tau^{(0)}) \in X \times \mathbb{R}^+$.
\item \mi{Propose $v^{(k)} = (1-\beta^2)^{1/2}u^{(k)} + \beta\xi^{(k)}$, where $\xi^{(k)} \sim N(0,\mathcal{C}_\tau)$.}
\item Set $u^{(k+1)} = v^{(k)}$ with probability $\alpha^{\tau^{(k)}}(u^{(k)},v^{(k)})$, or else set $u^{(k+1)} = u^{(k)}$.
\item Propose $\gamma^{(k)} \sim Q(\tau^{(k)},\cdot)$.
\item Set $\tau^{(k+1)} = \gamma^{(k)}$ with probability $\alpha^{u^{(k+1)}}(\tau^{(k)},\gamma^{(k)})$, or else set $\tau^{(k+1)} = \tau^{(k)}$.
\item $k\rightarrow k+1$ and return to 2.
\end{enumerate}
\end{algorithm}

Then $\{u^{(k)},\tau^{(k)}\}_{k\geq 0}$ is a Markov chain which is
invariant with respect to $\mu^y(du,d\tau)$.

\section{Numerical Results}
\label{sec:N}
We perform a variety of numerical experiments to illustrate the performance of the hierarchical algorithm described in section \ref{sec:M}. We focus on three different forward models. The first is pointwise observations composed with the identity -- the simplicity of this model allows us to probe the behavior 
of the algorithm at low computational cost, 
and such models are also of interest in
applications such as image reconstruction -- see for example \cite{alvarez1994formalization,sapiro2006geometric} and the references therein.
The other two, groundwater flow and EIT, are physical models which have previously been studied extensively, including study of non-hierarchical Bayesian level set methods \cite{levelset,eit}. A review of studies on inverse problems associated with EIT is given in \cite{borcea}.

\mi{The code used for simulations is available on GitHub at \url{https://github.com/mattdunlop/bayes-hier/releases/v1.0}}.

\mi{\subsection{Discretization of the problem}
\label{ssec:disc}
There are two spaces that we must discretize in order to implement the algorithm. The first is the state space, where the samples will be generated, and the second is the function space associated with the evaluation of the forward model. We briefly outline how this is done.}

\mi{Our discretization for the state space relies on the Karhunen-Lo\'eve expansion of the prior. Suppose we wish to produce samples from a Gaussian measure $N(0,\mathcal{C})$, where $\mathcal{C}$ has associated eigenbasis $\{\lambda_j,\varphi_j\}_{j\in\mathbb{N}}$. Then a sample $u$ from this distribution may be represented as
\[
u(x) = \sum_{j=1}^\infty \sqrt{\lambda_j}\xi_j\varphi_j(x),\;\;\;\xi_j\sim N(0,1)\text{ i.i.d.}
\]
We discretize the space by truncating and approximating this basis, so that elements of the space are represented as
\[
u^N(x) = \sum_{j=1}^N u_j^N\varphi_j^N(x).
\]
The inference is then performed on the random variables $\{u_j^N\}_{j=1}^N$.
Additionally, in the cases we consider, the eigenvectors associated with all covariances are given by the Fourier basis and so we may use the Fast Fourier Transform for efficient implementation.}

\mi{The second discretization occurs in the solution of the differential equations. In the EIT example a finite element discretization is used, in which the functions are approximated by expansion in a finite basis. The coefficients of the expansion of the solution to the PDE in this basis are then solved for numerically. The basis is chosen such that each basis element is locally supported -- this ensures that matrices arising in the implementation of the method are sparse.
}

\mi{The groundwater flow model uses a finite difference discretization, in which derivatives are approximated by difference quotients. For example, given a uniform grid $\{x_i,y_j\}_{i,j=1}^N$ with spacing $x_{i+1}-x_i = \delta$, we may approximate
\[
\frac{\partial h}{\partial x}(x_i,y_j) \approx \frac{h(x_i + \delta,y_j) - h(x_i - \delta,y_j)}{2\delta}.
\] 
This leads to an approximate solution to the PDE defined on the grid $\{x_i,y_j\}_{i,j=1}^N$.
}

\mi{Finite element, finite difference and even spectral methods outlined above can all be used for any PDE examples; what we use for illustrative purposes in this paper (EIT with finite element and groundwater flow with finite difference) are just examples of numerous possible forward models and discretization combinations.}

\subsection{Identity Map}
\label{ssec:id}

The first inverse problem is based on reconstruction of a piecewise constant
field from noisy pointwise observations. 

\subsubsection{The forward model}
\label{sssec:id_fwd}
Let $D = [0,1]^2$ and define a grid of observation points \mmd{$\{q_j\}_{j=1}^J\subseteq D$}. Let $Z = L^p(D)$ for some $1\leq p < \infty$ and let $Y = \mathbb{R}^J$. The forward operator $S:Z\rightarrow Y$ is defined by
\mmd{\[
S(\kappa) = (\kappa(q_1),\ldots,\kappa(q_J)).
\]
}We are then interested in finding \mmd{$\kappa$}, given the prior information that it
is piecewise constant, and taking a number of known prescribed values.
Let $\mathcal{G} = S\circ F:X\times\mathbb{R}^+\rightarrow Y$.
We reconstruct $(u,\tau)$ and hence \mmd{$\kappa=F(u,\tau)$.}
The map $S$ is not continuous, and so Assumptions \ref{ass:forward} do not hold. However Proposition \ref{prop:point_obs_cts} \mmd{in the Appendix} shows that the map $\mathcal{G}$ is uniformly bounded, and almost-surely continuous under the priors considered. From this the conclusions of Proposition \ref{p:app} \mmd{in the Appendix} follow, and it is possible
to deduce the conclusions of Theorem \ref{t:2}.

\subsubsection{Simulations and results}
\label{sssec:id_results}
We study the effect of different length scales, for both hierarchical
and non-hierarchical methods, demonstrating the advantages of the former
over the latter. To this end
we define $\tau_i^\dagger = 5i$, $i=1,\ldots,10$, and generate $10$ 
different true level set fields $u_i^\dagger \sim \mu_0^{\tau_i^\dagger}$ on a mesh of $2^{10}\times 2^{10}$ points. This leads to 10 sets of data $y_i$, given by
\[
y_i = \mathcal{G}(u_i^\dagger,\tau_i^\dagger) + \eta_i,\;\;\;\eta_i \sim N(0,\Gamma)\text{ i.i.d.}
\]
where we take the noise covariance \mmd{$\Gamma = 0.2^2\cdot I$ to be white}. The level set map $F$ is defined such that there are $3$ phases, taking the constant values 
$1, 3$ and $5.$ The mean relative error on the generated data sets ranges 
from $6\%$ to $9\%.$

One of the motivations for developing a hierarchical method is that little knowledge may be known a priori about the length scale associated with the unknown \mmd{geometric} field. We therefore sample from each hierarchical posterior distribution associated with each $y_i$ using a variety of initial values for the length scale parameter. This allows us to check that, computationally, we can recover a good approximation to the true length scale even if our initial guess is poor. 
Specifically, for each set of data we run 10 hierarchical MCMC simulations started at the different values of $\tau = \tau_k^\dagger$, giving a total of 100 hierarchical MCMC chains. For all chains we place a relatively flat prior of $N(20,10^2)$ on $\tau$. \mmd{On the prior for the level set function $u$ we take Neumann boundary conditions and fix the smoothness parameter $\alpha = 5$. The thresholding levels in the level set map are chosen such that there is an order one amount of prior mass in all levels -- specifically we take $c_1 = -0.1$ and $c_2 = 0.1$.}

We also wish to compare how the hierarchical method compares with the non-hierarchical method. We therefore look at the 10 different posterior distributions that arise from each set of data $y_i$ when using each of 10 fixed prior inverse length scales $\tau_k^\dagger$, which gives another 100 MCMC chains.
 
We perform all sampling on a mesh of $2^7\times 2^7$ points to avoid an inverse crime, \mmd{discretizing via the discrete Fourier transform (DFT) and retaining all $2^{14}$ modes}. The observation grid \mmd{$\{q_j\}_{j=1}^{100}$} is taken to be a uniformly spaced grid of 100 points. \mmd{We use a Gaussian random walk proposal distribution for the length scale parameter. We make this choice as it is the canonical starting point for MCMC, and it works in this case. It is possible however that something more sophisticated may be beneficial}. We produce $5\times 10^6$ samples for each chain, and discard the first $10^6$ samples as burn-in when calculating quantities of interest.

In Figure \ref{fig:id_taumean} we look at the recovery of the true value of $\tau$ with the hierarchical method. For large enough $\tau_0$, the mean of $\tau$ after the burn-in period is roughly constant with respect to varying the initialization point, for each posterior. This makes sense from a theoretical point of view since these means arise from the same posterior distribution, for a fixed truth, but it is also reassuring from a computational point of view since the output is close to independent of the initial guess for the length scale. There does however appear to be an issue with initializing the value of $\tau$ at too low a value, with the value $\tau$ tending to get stuck far from the truth when initialized at $=5$. This effect has been detected in several other experiments and models -- initializing the value of $\tau$ much lower than the true inverse length can cause the parameter to become stuck in a local minimum. Such an effect has not been observed however when the parameter is initialized significantly larger than the true value. Table \ref{tab:id_taumean} shows that recovery of the true value of $\tau$ is very good for $\tau^\dagger \leq 35$, though becomes slightly worse for larger values of $\tau^\dagger$.  The means here are calculated without the $\tau_0 = 5$ sample means since they are clearly outliers for most of the posteriors. One possible explanation for the lack of recovery in the cases $\tau^\dagger = 40$, $45$ and $50$ is to do with the structure of the observation map $S$. The observation grid has a length scale associated with it, related to distances between observation points, and so issues could arise when trying to detect the length scale of the \mmd{geometric} field that is significantly shorter than this. Additionally, the length scales $1/\tau$ are closer for larger $\tau$ and so it may be more difficult to distinguish between particular values.

For brevity we now focus on the case where $\tau^\dagger = 15$. The traces of the values of $\tau$ along the hierarchical chains corresponding to this truth is shown in Figure \ref{fig:id_tautrace}. After approximately $10^6$ samples, all chains have become centered around the true length scale. This convergence appears to be roughly linear for each chain. 

Figure \ref{fig:id_tau15} shows the push forwards of the sample means from the different chains under the level set map, that is, approximations of $F(\mathbb{E}(u),\mathbb{E}(\tau))$.  This figure also shows approximations of $\mathbb{E}(F(u,\tau))$ and typical samples of $F(u,\tau)$ coming from the different chains. We see that these conditional means for the hierarchical method appear to agree with one other. This is reassuring for the reason mentioned above -- they are all estimates of the mean of the same distribution. The figures for the non-hierarchical posteriors admit greater variation, especially near the boundary for higher values of $\tau$. Moreover, not all inclusions are detected when the length scale parameter is taken to be $\tau = 5$. Note that the mean from the hierarchical posterior agrees closely with that from the non-hierarchical posterior using the fixed true length-scale $\tau = 15$. Additionally, even though the means are reasonable approximations to the truth in most cases, the typical samples are much worse when using the non-hierarchical method with an incorrect length scale parameter.

We can also consider the sample variance of the pushforward of the samples by the level set map, i.e. approximations of the quantity $\text{Var}(F(u,\tau))$. In Figure \ref{fig:id_var_tau15} we show this quantity for both the hierarchical and non-hierarchical priors. Note that for the non-hierarchical priors, the variance increases both at the boundary and away from the observation points for larger values of $\tau$. Variance is also higher along the interfaces and within the central phase, since points in these locations are more likely to switch between all three phases. The hierarchical approximations all appear to agree. Whilst the hierarchical means are very similar to the non-hierarchical means using the true length scale, as seen in Figure \ref{fig:id_tau15}, the hierarchical variances are smaller away from the observation points.

Additionally, we look at the level set function $u$ itself in Figure \ref{fig:id_tau15_u}. In these plots we rescale the level set function by $\tau^{\alpha-d/2} = \tau^4$ so that they are all of approximately the same amplitude. The means for both the hierarchical and non-hierarchical methods are again quite similar to one another, though the difference between the typical samples is much more stark.

Finally, in Figure \ref{fig:id_densities_tau15}, we look at the joint densities of the inverse length scale parameter $\tau$ and first five Karhunen-Lo\`{e}ve (KL) modes of the level set function $u$.\footnote{KL modes are the eigenfunctions of the covariance operator, here ordered by decreasing eigenvalue.} Non-trivial correlations are evident between $\tau$ and each of these modes, with the support of the densities appearing non-convex. This is likely related to the non-linear scaling between the length-scale and the amplitude of the level-set function under the prior. Conversely the KL modes, whilst still correlated with one-another other, have simpler joint densities.
Note, also, that the posterior on the length scale is centered close to the
true value of the inverse length scale parameter $\tau.$ 

\mmd{
\begin{remark}
In this section we studied the ability to recover the true length scale parameter $\tau^\dagger$, given a finite number of direct noisy observations of the geometric field. The question arises of how the quality of this recovery depends upon the spatial resolution of the data. As would be expected, learning this parameter becomes more difficult when this resolution is poor due to the lack of information in the data. However it is interesting to note that, even in the limit of an infinite number of distinct observation points, it is unlikely that we would be able to identify $\tau^\dagger$ perfectly. This is suggested by a result of Zhang \cite{zhang2004inconsistent} which states that, in the context of generalized linear mixed models, the marginal variance and length-scale parameters of a Mat\'ern field cannot be consistently estimated in this limit \mi{where as in our case the domain is fixed. This is in contrast to the case of additional data points increasing the domain, where consistent estimation is possible \cite{mardia_marshall}.} $\hfill\qed$
\end{remark}
}

\begin{figure}
\begin{center}
\includegraphics[width=0.7\linewidth]{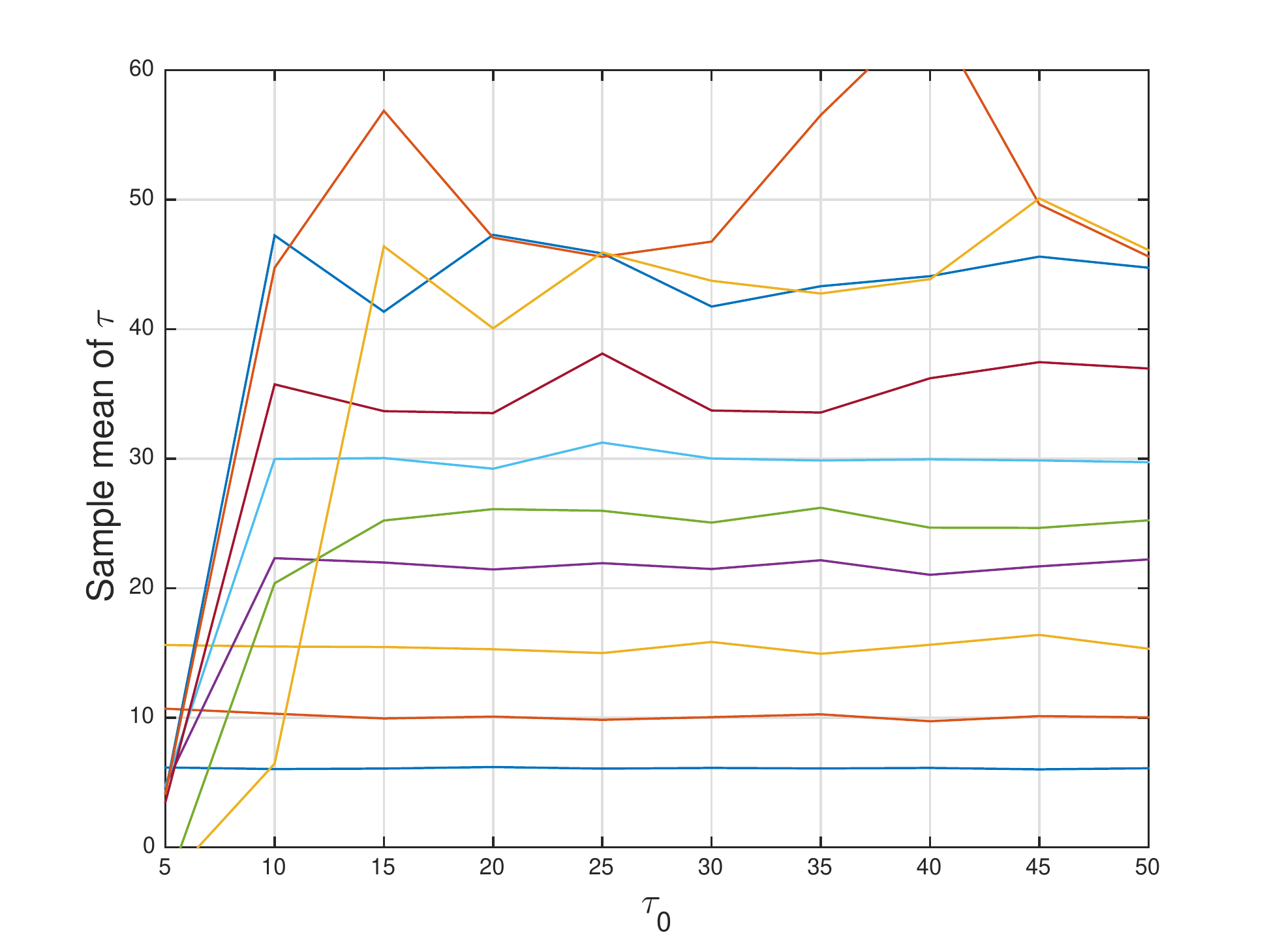}
\caption{(Identity model) The sample mean of $\tau$ along each hierarchical MCMC chain, against the initial value of $\tau$. The different curves arise from using different data $y_i$.}
\label{fig:id_taumean}
\end{center}
\end{figure}

\begin{table}
\caption{(Identity model) The value of $\tau$ used to create the data $y_i$, and the mean value of $\tau$ across the MCMC chains and the different initial values of $\tau$.}

\label{tab:id_taumean}
\begin{tabularx}{0.7\linewidth}{XX}
\hline\noalign{\smallskip}
$\tau^\dagger$ & Mean sample mean of $\tau$ \\
\noalign{\smallskip}\hline\noalign{\smallskip}
5 & 6.10\\
10 & 10.0\\
15 & 15.5\\
20 & 21.8\\
25 & 24.8\\
30 & 30.0\\
35 & 35.4\\
40 & 44.6\\
45 & 50.8\\
50 & 40.6\\
\noalign{\smallskip}\hline
\end{tabularx}
\end{table}

\begin{figure}
\begin{center}
\includegraphics[width=0.7\linewidth,trim=0cm 0cm 0cm 0cm]{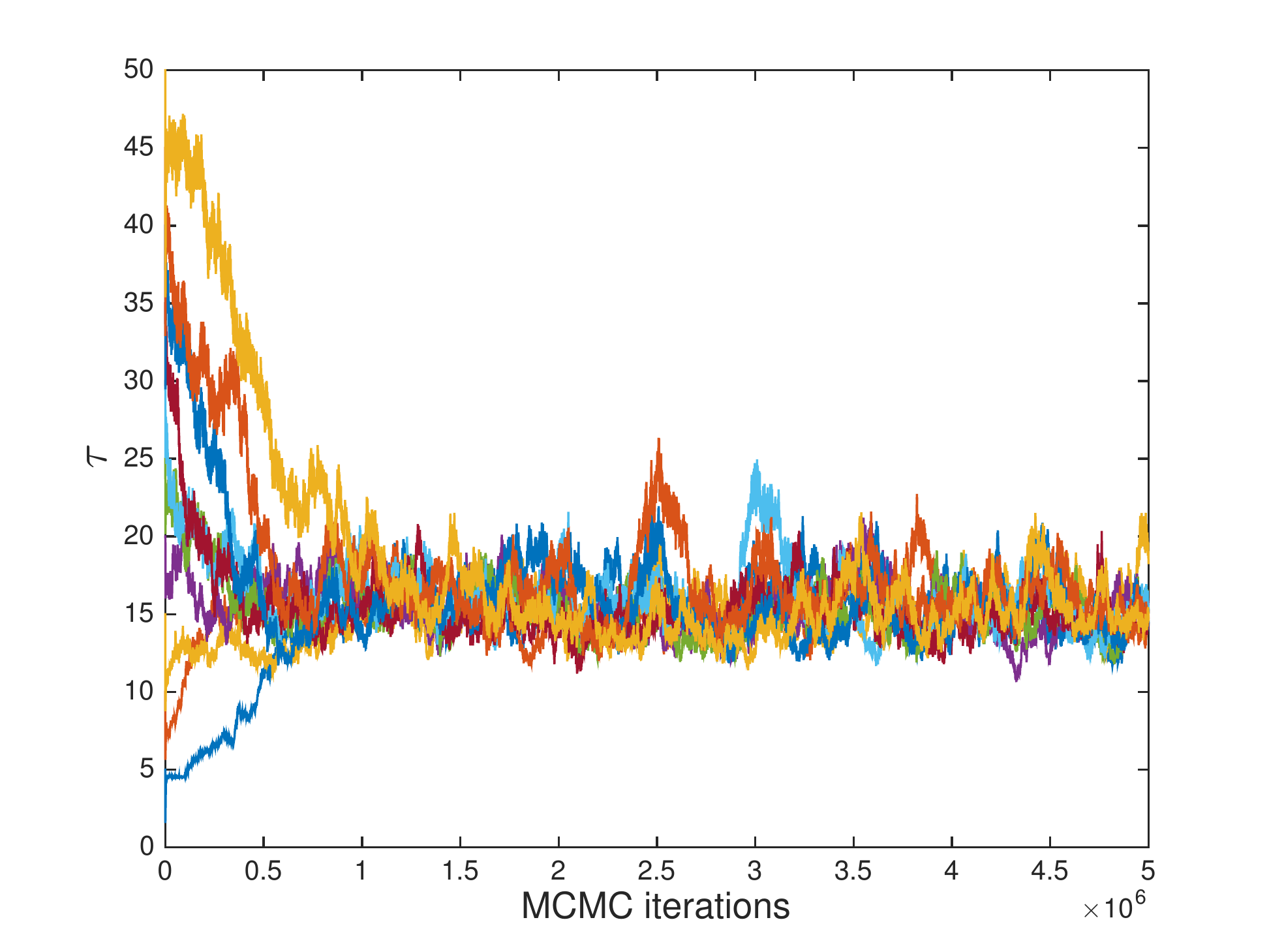}
\caption{(Identity model) The trace of $\tau$ along the MCMC chain, when initialized at the 10 different initial values. True inverse length scale is $\tau = 15$.}
\label{fig:id_tautrace}
\end{center}
\end{figure}

\begin{figure*}
\begin{center}
\subfloat[\mmd{The true geometric field used to generate the data $y$, with true inverse length scale $\tau = 15$}]{\hspace{3.5cm}\includegraphics[width=0.35\textwidth]{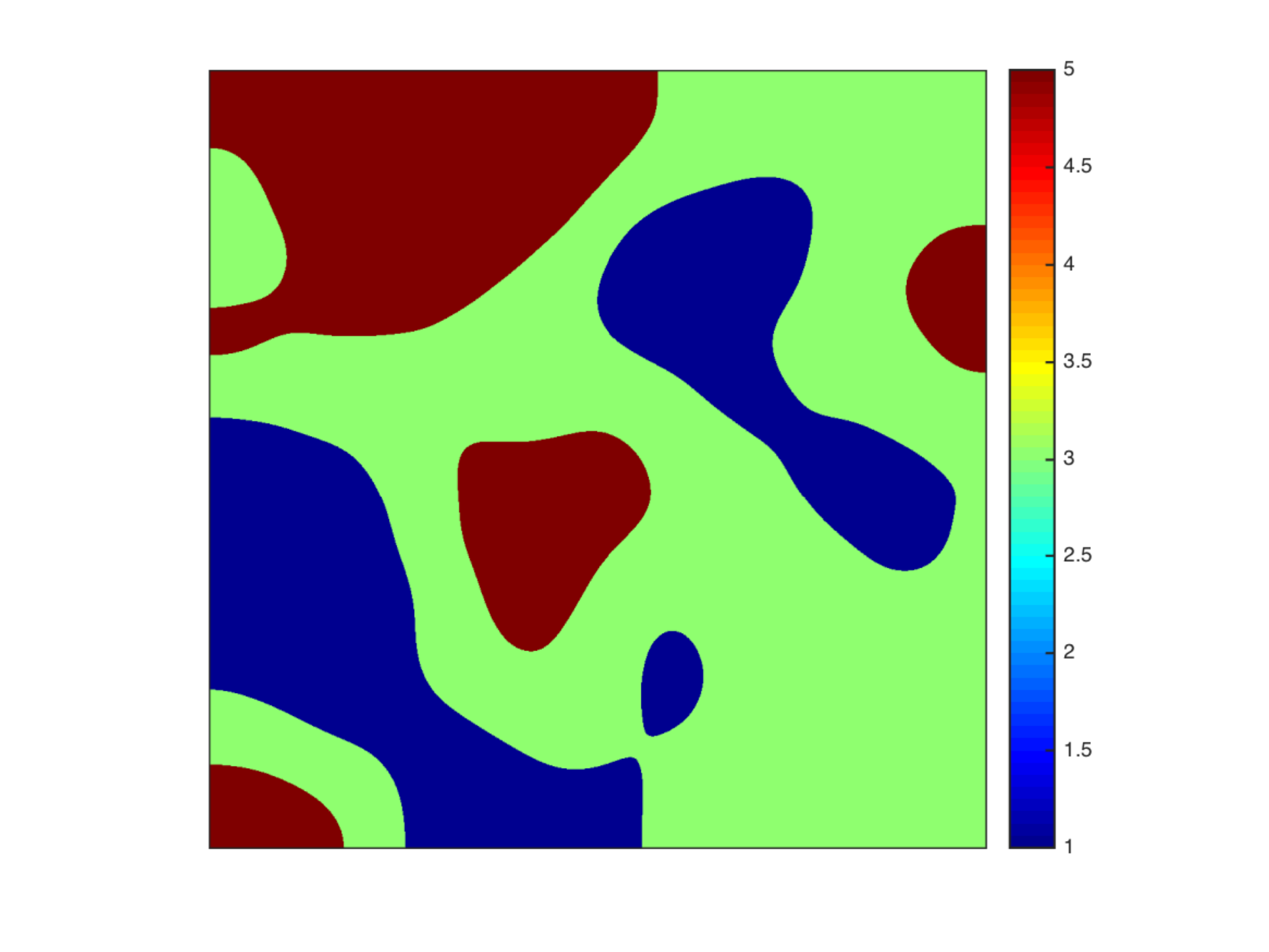}\includegraphics[width=0.28\textwidth,trim=0cm 0cm 0cm 4cm,clip]{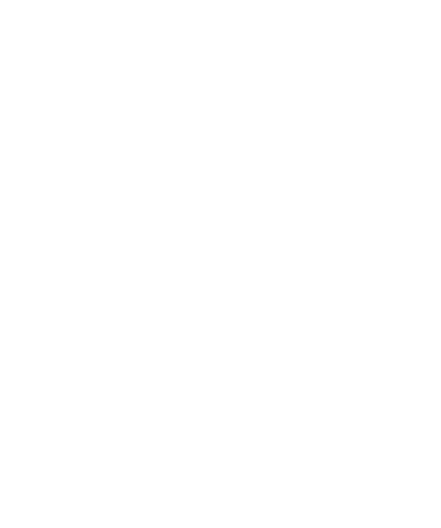}}\\
\subfloat[\mmd{(Top) Representative samples of $F(u,\tau)$ under the hierarchical posterior. (Middle) Approximations of $F(\mathbb{E}(u),\mathbb{E}(\tau))$. (Bottom) Approximations of $\mathbb{E}(F(u,\tau))$. From left-to-right, $\tau$ is initialized at $\tau = 5,15,25,35,45$.}]{\includegraphics[width=0.9\linewidth]{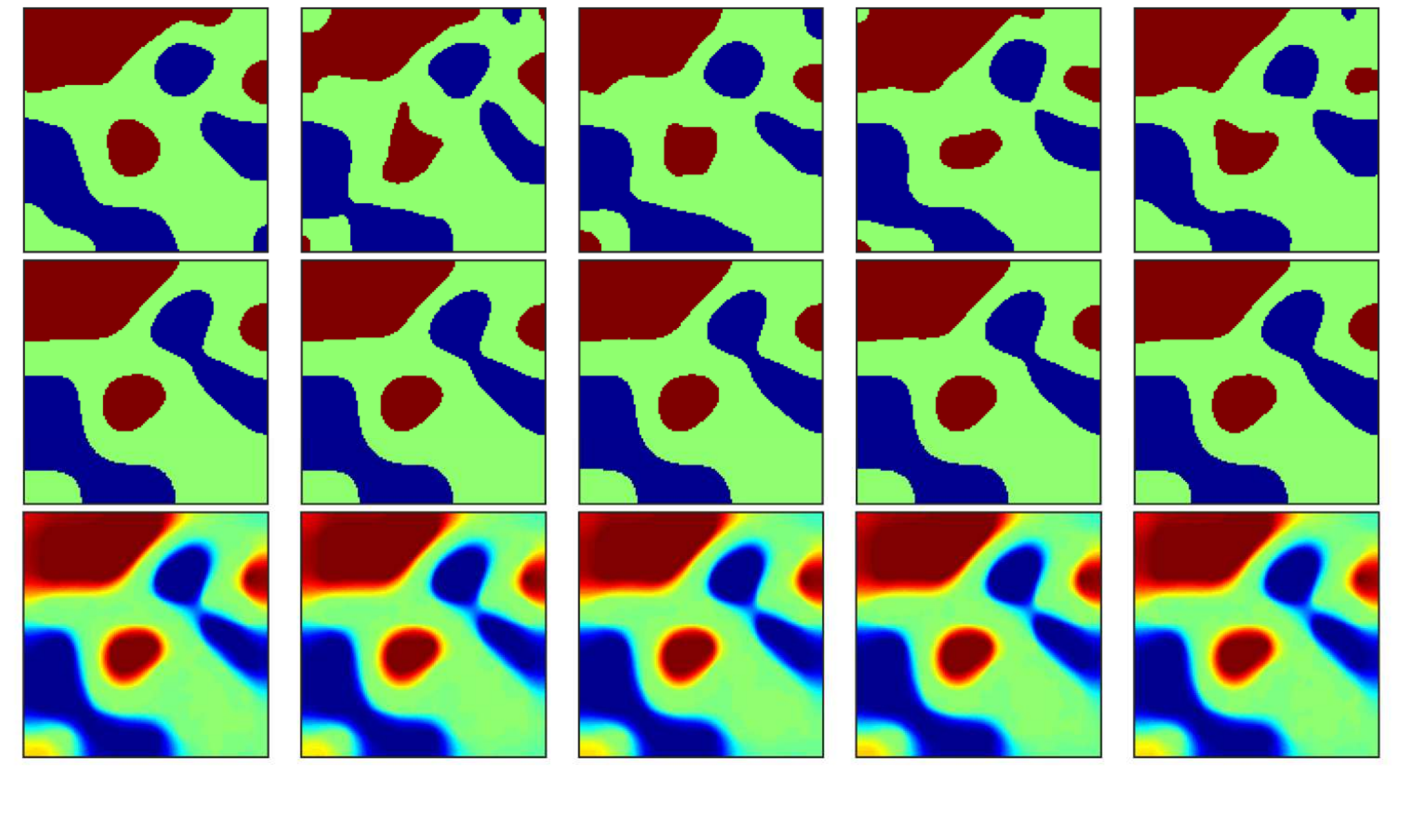}}\\
\subfloat[\mmd{As in (b), using the non-hierarchical method. From left-to-right, $\tau$ is fixed at $\tau = 5,15,25,35,45$.}]{\includegraphics[width=0.9\linewidth]{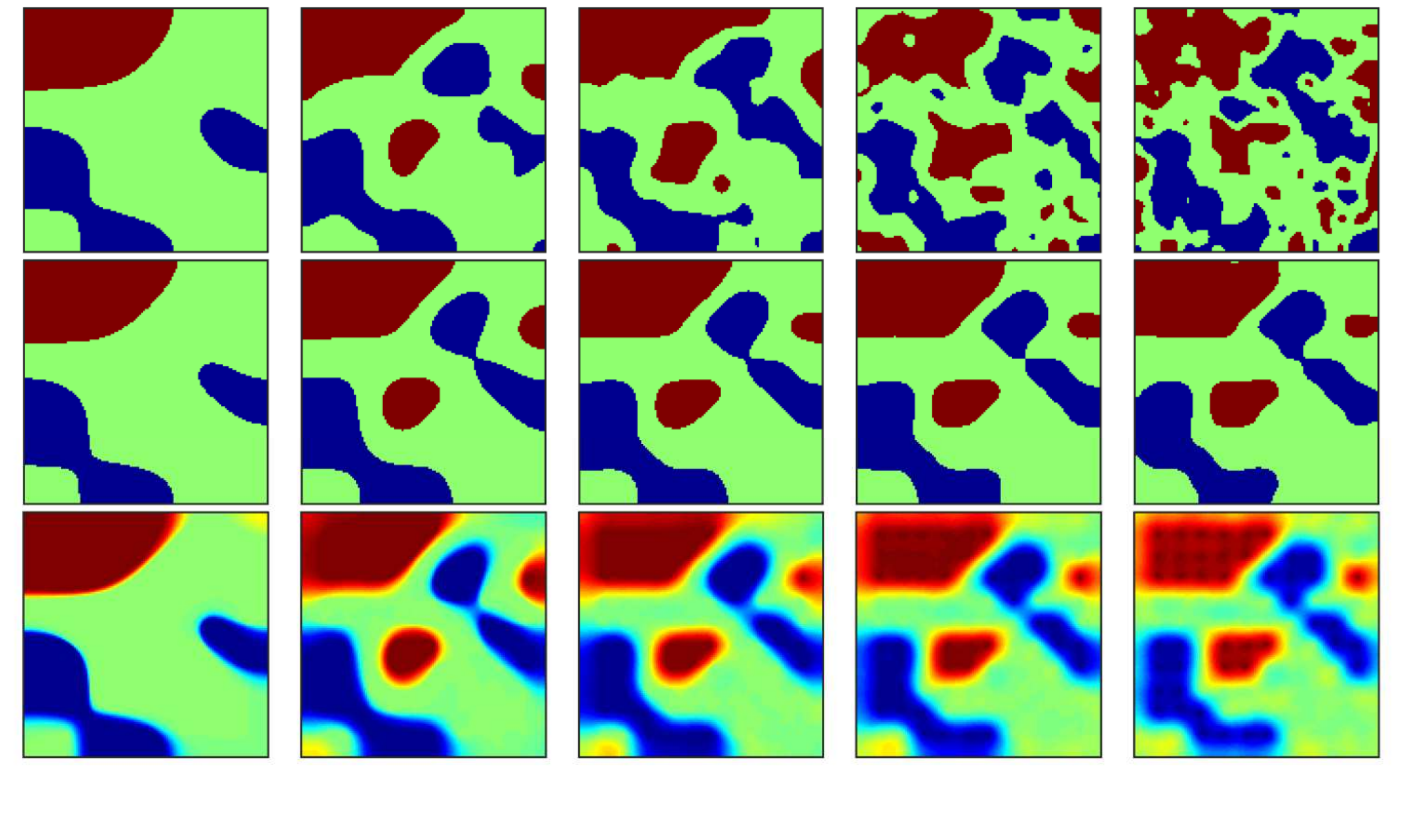}}
\caption{\mmd{Simulations for the identity model.}}
\label{fig:id_tau15}
\end{center}
\end{figure*}

\begin{figure*}
\begin{center}
\includegraphics[width=0.9\linewidth, trim=0cm 0cm 2cm 0cm]{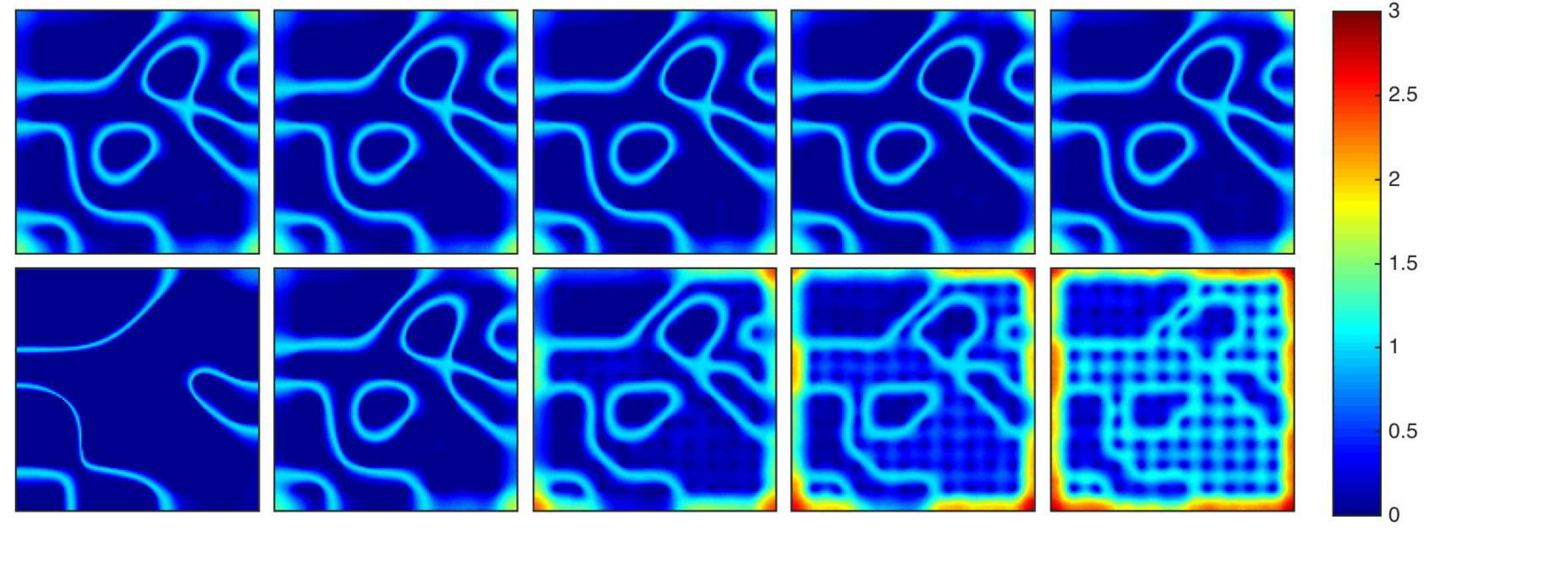}
\caption{(Identity model) Approximations of $\text{Var}(F(u,\tau))$ using the hierarchical (top) and fixed (bottom) priors, initialized or fixed at $\tau = 5,15,25,35,45$, from left-to-right. True inverse length scale is $\tau = 15$.}
\label{fig:id_var_tau15}
\end{center}
\end{figure*}

\begin{figure*}
\begin{center}
\includegraphics[width=0.85\linewidth, trim=0cm 0cm 2cm 0cm]{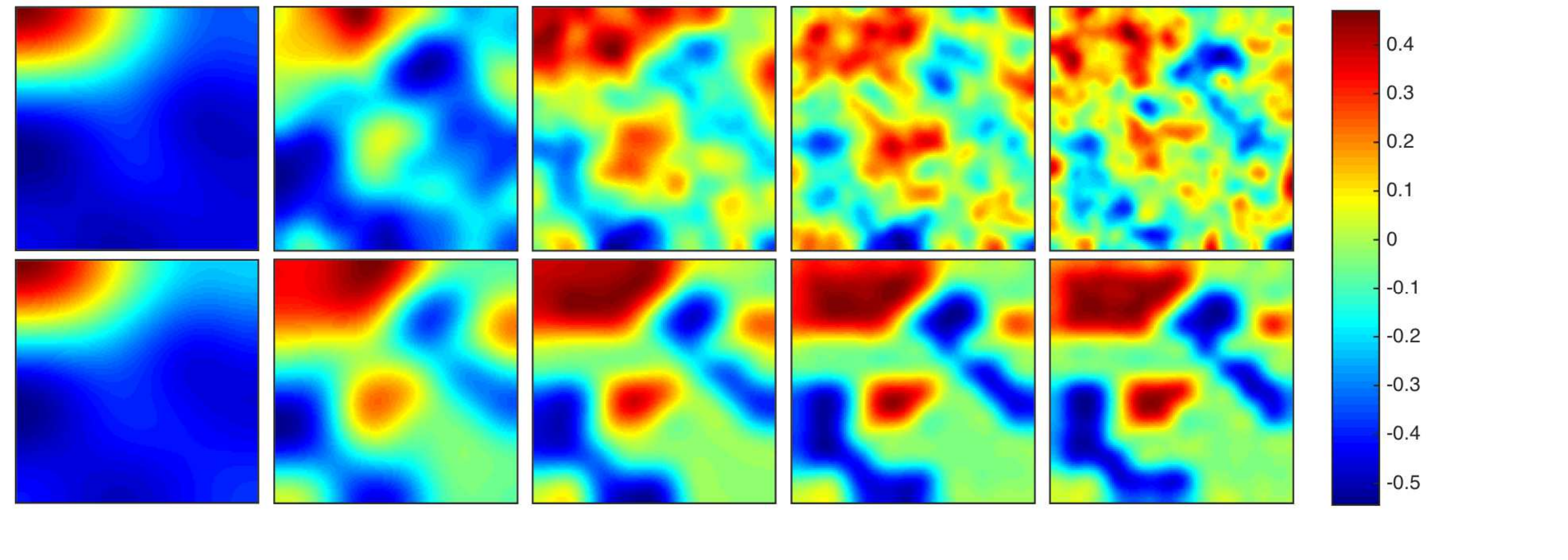}\includegraphics[width=0.248\linewidth, trim=22.2cm 0cm 2cm 0cm,clip]{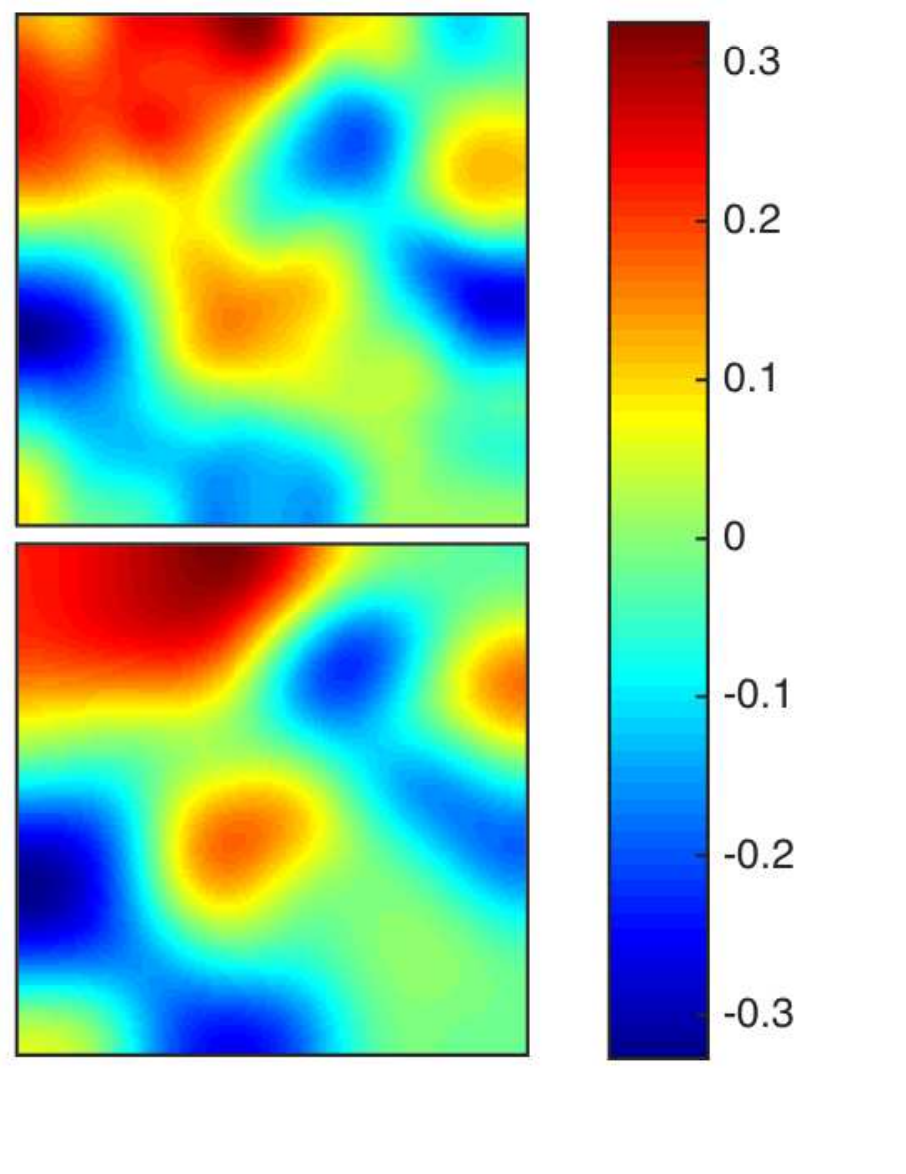}
\caption{\mi{(Identity model) Representative samples $\tau^4\cdot u$ (top) and sample means $\mathbb{E}(\tau^4\cdot u)$ (bottom) of the level set function. The rescaling $\tau^4$ means that the above quantities have the same approximate amplitude. True inverse length scale is $\tau = 15$.
(Left) Using the non-hierarchical method; from left-to-right $\tau$ is fixed at $\tau = 5,15,25,35,45$.
(Right) Corresponding quantities for the hierarchical method.}}
\label{fig:id_tau15_u}
\end{center}
\end{figure*}

\begin{figure*}
\begin{center}
\includegraphics[width=\linewidth]{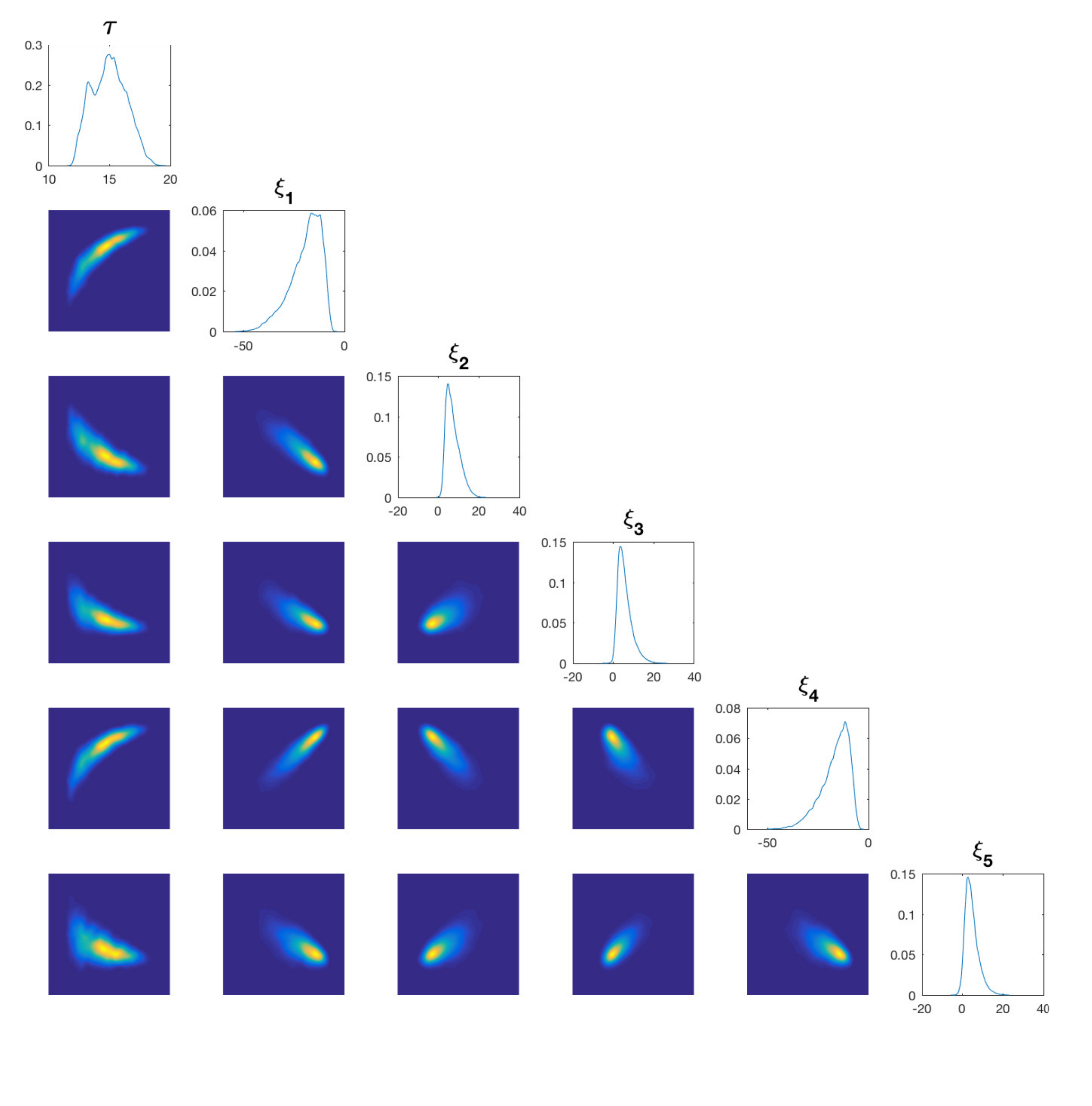}
\caption{(Identity model) (diagonal) Empirical densities of $\tau$ and the first five KL modes of $u$. (off-diagonal) Empirical joint densities. True inverse length scale is $\tau = 15$.}
\label{fig:id_densities_tau15}
\end{center}
\end{figure*}

\subsection{Identification of Geologic Facies in Groundwater Flow}
\label{ssec:gf}
The identification of geologic facies in subsurface flow applications is a common example of a large scale inverse problem that involves the recovery of unknown interfaces. In the case of groundwater flow, for example, the inverse problem concerns the recovery of the interface between regions with different hydraulic conductivity given measurements of hydraulic head. Geometric inverse problems of this type have recently received a lot of attention by the research community \cite{Xie2011,Ping2014,Lorentzen2012,Lorentzen2012_2}. Indeed, it has been recognized that the geometry determined by the aforementioned interfaces constitutes one of the main sources of uncertainty that must be quantified and reduced by means of Bayesian inversion.

In the context of groundwater flow, the identification of interfaces between regions associated with different types of geological properties can be posed as the recovery of a piecewise constant conductivity field parameterized with a level set function. A fully Bayesian level set framework for the solution of the aforementioned type of inverse problems has been recently developed in \cite{levelset}. The MCMC method applied in \cite{levelset} performs well when the prior of the level set function properly encodes the intrinsic length-scales of the unknown interfaces. Clearly, in practical applications such length-scales are most likely unknown and their incorrect specification may result in inaccurate and uncertain estimates of the unknown interfaces. The purpose of this section is to show that the proposed hierarchical Bayesian framework enables us to determine an optimal length-scale in the prior of the level set function which, in turn, captures more accurately the intrinsic length-scale of the unknown interface.

\subsubsection{The forward model}
\label{sssec:gwf_fwd}

We are interested in the identification of a piecewise constant hydraulic conductivity, denoted by $\kappa$, of a two-dimensional  confined aquifer whose physical domain is $D=[0,6]\times [0,6]$. We assume single-phase steady-state Darcy flow. The piezometric head, denoted by $h(x)$ ($x\in D$), which describes the flow within the aquifer can be modeled by the solution of \cite{Bear}
\begin{eqnarray}\label{eq:a1}
-\nabla\cdot \kappa \nabla h&=f &\qquad\textrm{in}~~D
\end{eqnarray}
where $f$ represents sources/sinks and where boundary conditions need to be specified. For the present work we consider the setup from the Benchmark used in \cite{Carrera,Hanke,Repre,enkf,reg_enkf,levelset}. In concrete, we assume that $f$ is a recharge term of the form
\begin{eqnarray}\label{eq:a2}
f(x_{1},x_{2})=\left\{\begin{array}{ccc}
0 &\textrm{if}& 0< x_{2}\leq 4,\\
137&\textrm{if}& 4< x_{2}< 5,\\
274&\textrm{if}& 5\leq x_{2} < 6.\end{array}\right.
\end{eqnarray}
and we consider the following boundary conditions
\begin{equation}
\begin{aligned}
\label{eq:a3}
 h(x_1,0)=100, \qquad \frac{\partial h}{\partial x_1}(6,x_2)=0,\\
 -\kappa\frac{\partial h}{\partial x_1}(0,x_2)=500,\qquad   \frac{\partial h}{\partial x_2}(x_1,6)=0.
\end{aligned}
\end{equation}

We consider the inverse problem of recovering $\kappa$ from observations $\{\ell_j(h)\}_{j=1}^{64}$ of $h$ given by (\ref{eq:a1})-(\ref{eq:a3}). We assume we have smoothed point observations given by
\[
\ell_j(h) = \int_D \frac{1}{2\pi\eps^2}e^{-\frac{1}{2\eps^2}(x-q_j)^2}h(x)\,\dee x
\]
where $\eps > 0$ and \mmd{$\{q_j\}_{j=1}^{64}\subseteq D$} is a grid of 64 observation points equally distributed on $D$. Let $Z = L^p(D)$ for some $1\leq p < \infty$ and $Y = \mathbb{R}^{64}$. \mmd{Given $\kappa \in Z$, let $h$ be given by (\ref{eq:a1})-(\ref{eq:a3}).} Then the forward map $S:Z\rightarrow Y$ is given by
\[
\kappa\mapsto (\ell_1(h),\ldots,\ell_{64}(h)).
\]
We assume that each $\kappa_i$ in the definition of the level set map $F$ is strictly positive. The image of $F$ is contained in the set of bounded fields on $D$ bounded below by $\min_i\kappa_i > 0$. In \cite{levelset} the map $S$ is shown to be continuous and uniformly bounded on such fields, with respect to $\|\cdot\|_{L^p(D)}$ for some $p$, and so Assumptions \ref{ass:forward} hold. As a consequence
Theorem \ref{t:2} applies directly.

\subsubsection{Simulations and results}
\label{sssec:gwf_results}

In the previous example we illustrate, with a simple model, the capabilities of the proposed framework to recover a specified true length-scale and a true level set function that defines a true discontinuous field from which synthetic data are generated. However, we must reiterate that, in practice, we wish to recover the true discontinuous field; the level set function is merely an artifact that we use for the parameterization of such a field. In practical applications the aim of the proposed hierarchical Bayesian level set framework is to infer a length-scale alongside with a level set function which, by means of expression (\ref{lvlsetmap}), produces a discontinuous field that captures the desired piecewise constant
field as accurately as possible and, in particular, the intrinsic length-scale 
separation of the interfaces determined by the discontinuities of the true \mmd{geometric} field. Therefore, in order to test our methodology in the applied setting of groundwater flow, rather than a true level set function, in this subsection we consider the true hydraulic conductivity $\kappa^{\dagger}$ whose logarithm is displayed in Figure \ref{fig:gw_0}\mmd{(a)}. This $\kappa^\dagger$ is defined such that it takes the constant values $e^{1.5}$, $e^{4}$ and $e^{6.5}$. This is channelized conductivity typical of fluvial environments and often used as Benchmarks for subsurface flow inversion \cite{Ping2014,Lorentzen2012,Xie2011,levelset}. Note that the values that the conductivity can take on the three different regions differ by at least one order of magnitude,
due to the logarithmic transformation. While there is indeed an intrinsic length-scale in the channelized structure, this true conductivity field does not come from a specified level set prior.

Synthetic data are generated by means of 
\[
y= (\ell_1(h^{\dagger}),\ldots,\ell_{64}(h^{\dagger}))+ \eta,\;\;\;\eta \sim N(0,\Gamma)\text{ i.i.d.}
\]
where $h^{\dagger}$ is the solution to (\ref{eq:a1})-(\ref{eq:a3}) for $\kappa=\kappa^{\dagger}$. Equations (\ref{eq:a1})-(\ref{eq:a3}) have been solved with cell-centered finite differences \cite{Wheeler}. In order to avoid inverse crimes, synthetic data are generated on a grid finer ($160\times 160$ cells) than the one used for the inversion ($80\times 80$ cells). \mmd{The discretization is performed via the DFT, and we retain all modes.} In addition, $\Gamma$ is a diagonal matrix given by $\Gamma_{i,i}= 0.0175\ell_i(h^{\dagger})$. In other words, we add noise that corresponds to $1.75\%$ of the size of the noise-free observations. \mmd{On the prior for the level set function $u$ we take Neumann boundary conditions and fix the smoothness parameter $\alpha = 5$.}

We consider a Gaussian prior $N(35,10^2)$ for $\tau$, \mmd{and use a Gaussian random walk proposal distribution for this parameter}. We then apply the hierarchical MCMC method from subsection \ref{ssec:alg} initialized with the following six different choices of $\tau=1,10,30,50,70,90$ and a sample of the prior (with that given $\tau$) of the level set function $u$. We thus produce six MCMC chains of length $4\times 10^6$ and discard the first $10^6$ as burn-in for the computation of quantities of interest. The trace plots of $\tau$ are displayed in Figure \ref{fig:gw_1} from which we clearly observe that all chains, regardless of their initial point, seem to stabilize and produce samples around $\tau=18$. In the \mmd{top row of Figure \ref{fig:gw_0}(b)} we display the logarithm of some representatives samples of $F(u,\tau)$ under the hierarchical posterior.  The \mmd{middle row of Figure \ref{fig:gw_0}(b)} shows the logarithm of $F(\mathbb{E}(u),\mathbb{E}(\tau))$, i.e., the pushforward of the posterior means obtained using the hierarchical method. The \mmd{bottom row of Figure \ref{fig:gw_0}(b)} displays the logarithm of the approximations of $\mathbb{E}(F(u,\tau))$. That is, the expected value of the pushforward samples under the posterior. The aforementioned results corresponds to five MCMC chains with $\tau$ initialized $\tau=10,30,50,70,90$ (the results for $\tau=1$ have been omitted). Similarly, Figure \ref{fig:gw_3} (top) shows the approximations of the variance of the pushforward samples of the posterior, i.e. $\text{Var}\big(F(u,\tau)\big)$. Clearly, both $\mathbb{E}(F(u,\tau))$ and $F(\mathbb{E}(u),\mathbb{E}(\tau))$ result in fields that provide a reasonable approximation of the true \mmd{geometric} field. Note that, as expected, the largest uncertainty in the distribution of the pushforward samples is around the interface between the regions with different conductivity. In \mmd{Figure \ref{fig:gw_4}(a) we show some representative samples of $u$ (top) and approximations to $\mathbb{E}(u)$ (bottom)}. In these plots, as before, we rescale the level set function by $\tau^{\alpha-d/2} = \tau^4$ so that they are all of approximately the same amplitude. In Figure \ref{fig:gw_5} we display the empirical densities of $\tau$ and the first five KL modes of $u$. A key observation
is that, although the true hydraulic conductivity is not generated
by thresholding a Gaussian random field, and hence there is no
``true'' length scale, the posterior nonetheless settles on a narrow
range of values of $\tau$ which are consistent with the data.

From the aforementioned results we can also clearly see that the hierarchical MCMC algorithm produces similar outcomes regardless of the initialization of the inverse of the length-scale $\tau$, reflecting ergodicity of the Markov chain.
The results from $\tau=1$ are not shown but they are very similar to the ones from other chains. As with the results from the previous subsection, the similarity in outcomes between all six chains is not surprising as these are aimed at sampling from the same posterior distribution; but the fact that this posterior distribution on $\tau$ concentrate\mmd{s} near to
a single value is of particular interest because it shows that the true \mmd{geometric}
field has an intrinsic length-scale, even though it was no\mmd{t} constructed via the
map $F(u,\tau).$ Furthermore, this similarity of outcomes between chains showcases the main advantage of the proposed framework with respect to the non-hierarchical one. Indeed, as stated earlier, the proposed method has the ability to recover a distribution for the intrinsic length-scale which gives rise to reasonably accurate estimates (i.e. $F(\mathbb{E}(u),\mathbb{E}(\tau))$ and  $\mathbb{E}(F(u,\tau))$) of the true
\mmd{geometric} field. We now present the numerical results from applying a non-hierarchical MCMC algorithm in which the inverse of length-scale $\tau$ is fixed. We consider again six MCMC chains as before with the (now fixed) values of $\tau=1,10,30,50,70,90$ that we used to initialized the hierarchical chains used before.  Analogous results to the ones presented for the hierarchical method can be found in the bottom panels of Figure \ref{fig:gw_0} as well as the bottom of Figures 
\ref{fig:gw_3} and \ref{fig:gw_4}. Clearly, the lack of properly prescribing the intrinsic length-scale in the non-hierarchical method results in inaccurate estimates of the true \mmd{geometric} field. We clearly observe that for $\tau\ge 30$ the estimates of the truth given by $F(\mathbb{E}(u),\mathbb{E}(\tau))$ and  $\mathbb{E}(F(u,\tau))$ are substantially inaccurate and the uncertainty measured by $\text{Var}\big(F(u,\tau)\big)$ is large. The non-hierarchical MCMC for $\tau=1$ did not converge; the results are not shown. The non-hierarchical MCMC only provides reasonable estimates for $\tau=10$ and $\tau=30$. However, we can visually appreciate that these results are still suboptimal when compared to the results from the hierarchical framework.

\begin{figure}
\begin{center}
\includegraphics[width=0.7\linewidth]{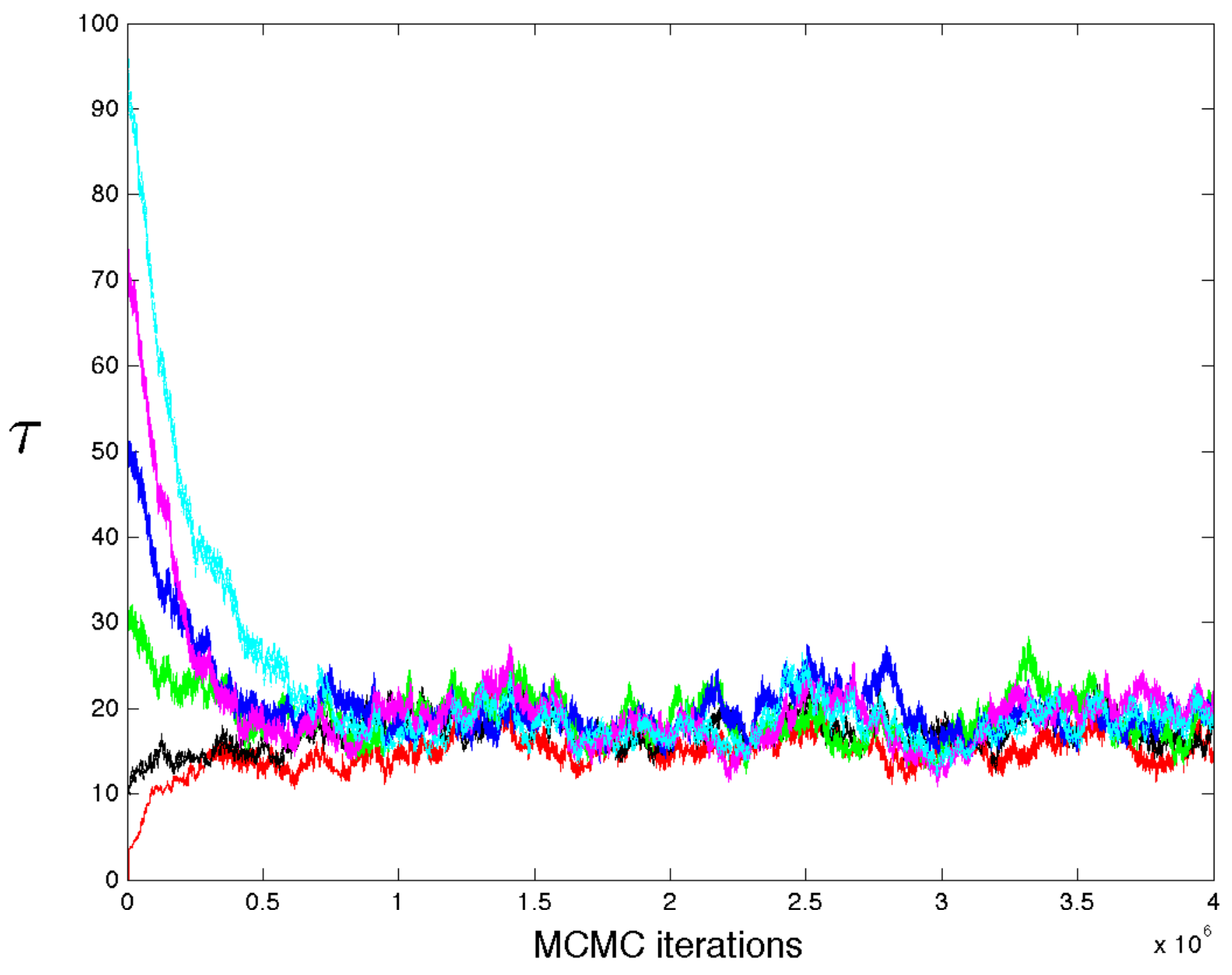}
\caption{(Groundwater flow model) Trace plots of $\tau$ obtained from six hierarchical MCMC chains.}
\label{fig:gw_1}

\end{center}
\end{figure}

\begin{figure*}
\begin{center}
\begingroup
\captionsetup[subfigure]{width=0.9\textwidth}
\subfloat[\mmd{(Left) Logarithm of the true hydraulic conductivity field used to generate the data $y$. (Right) True pressure field, and the grid of observation points.}]{\includegraphics[width=0.22\linewidth]{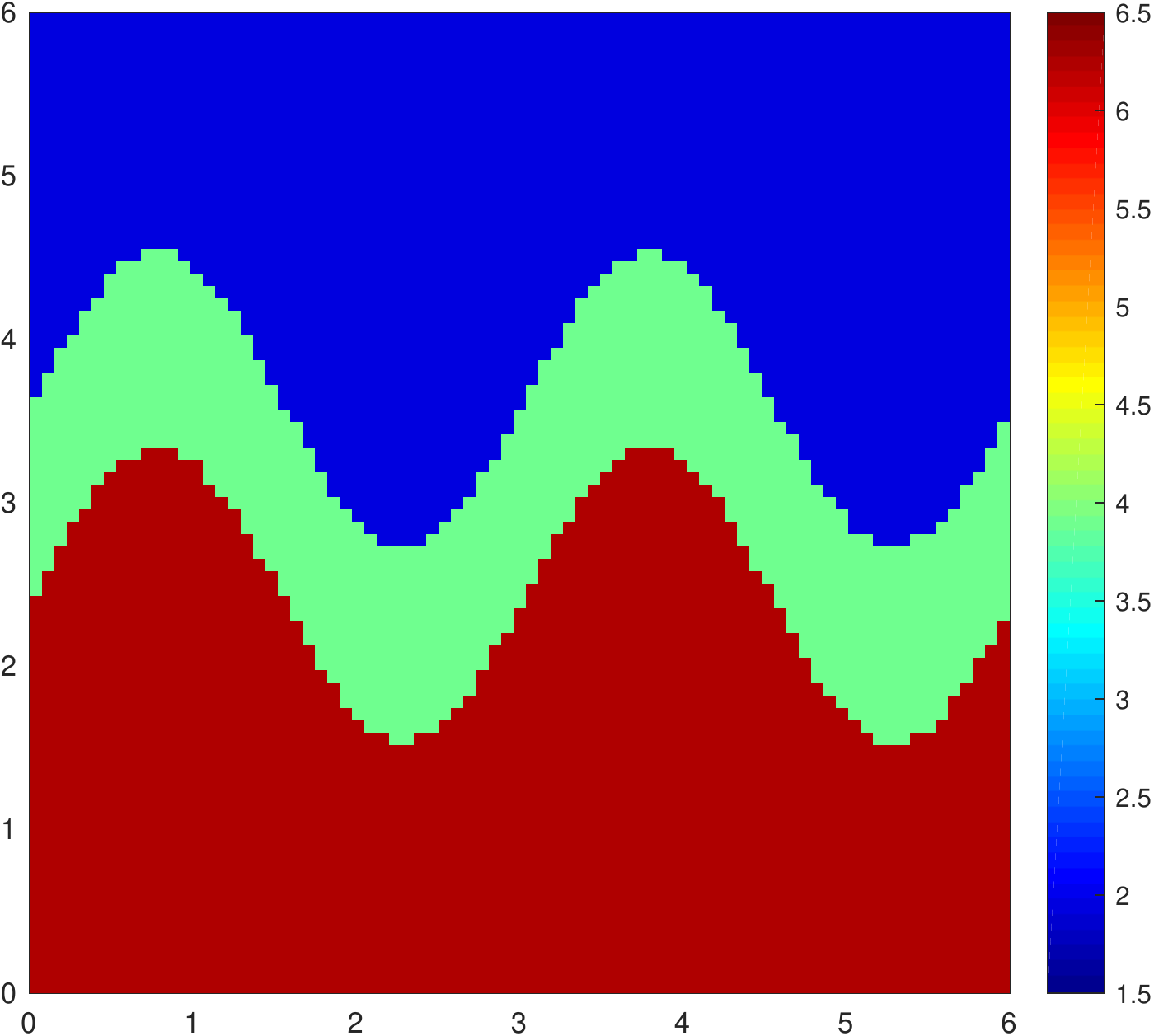}\includegraphics[width=0.3\linewidth, trim=0cm 1cm 0cm 0cm clip]{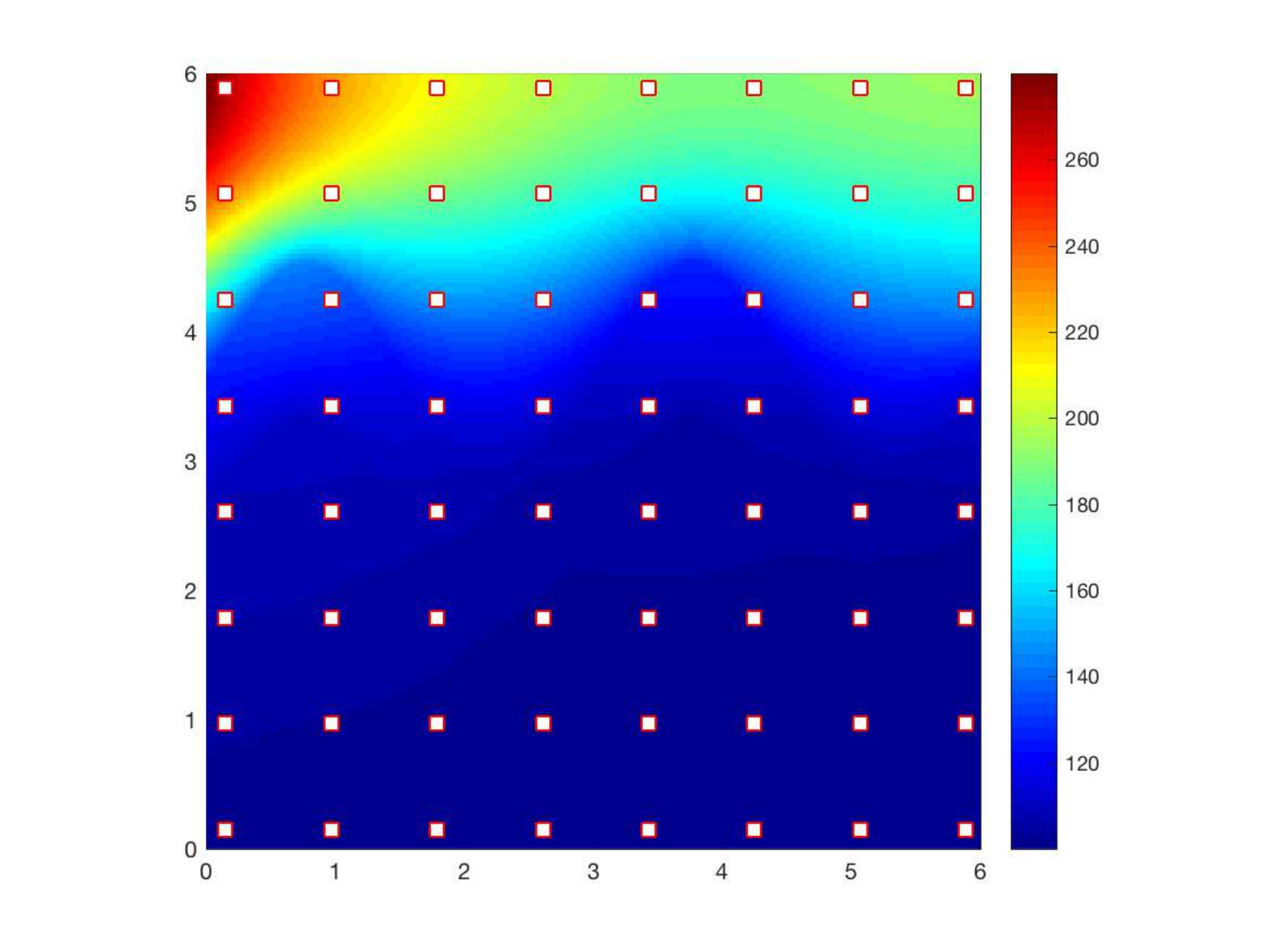}}\\
\endgroup
\subfloat[\mmd{(Top) Logarithm of representative samples of $F(u,\tau)$ under the hierarchical posterior. (Middle) Logarithm of the  approximations of $F(\mathbb{E}(u),\mathbb{E}(\tau))$. (Bottom) Logarithm of the approximations of $\mathbb{E}(F(u,\tau))$. From left-to-right, $\tau$ is initialized at $\tau = 10,30,50,70,90$.}]{\includegraphics[width=0.9\linewidth]{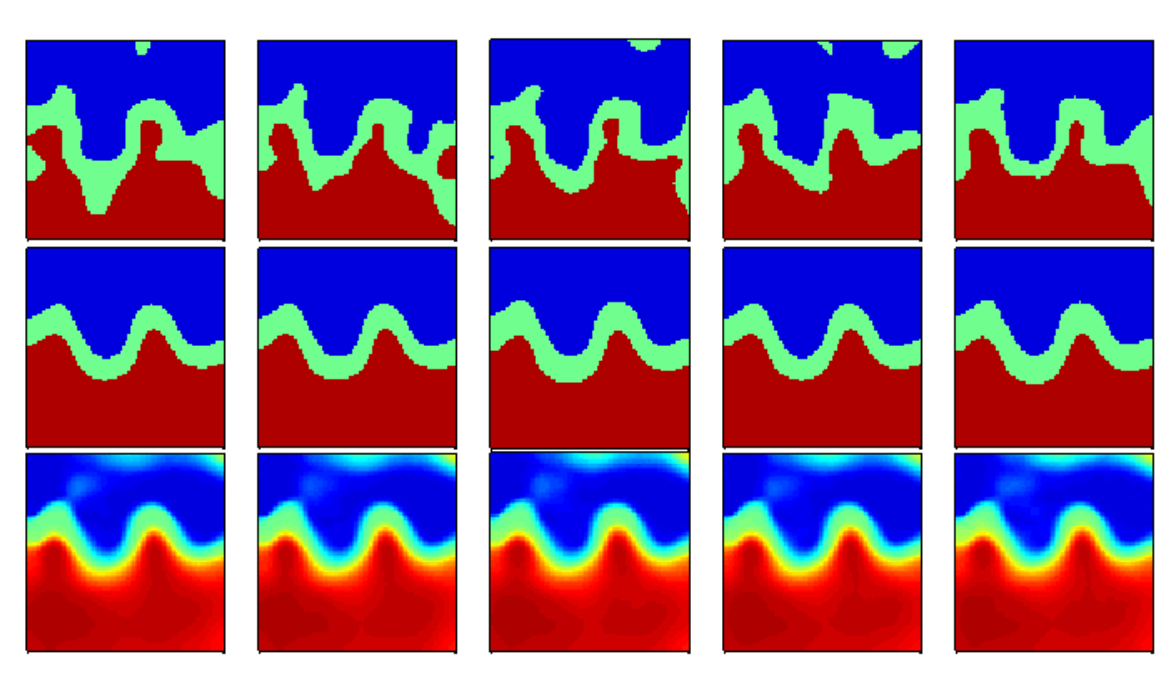}}\\
\subfloat[\mmd{As in (b), using the non-hierarchical method. From left-to-right, $\tau$ is fixed at $\tau = 10,30,50,70,90$.}]{\includegraphics[width=0.9\linewidth]{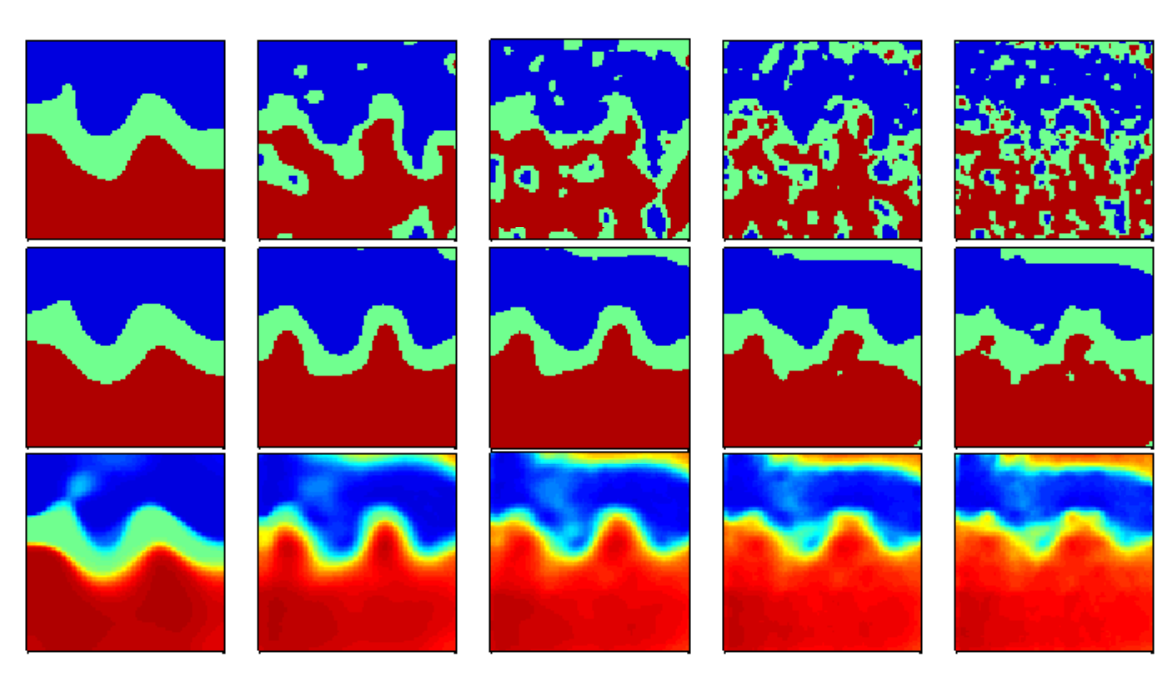}}
\caption{\mmd{Simulations for the groundwater flow model.}}
\label{fig:gw_0}
\end{center}
\end{figure*}

\begin{figure*}
\begin{center}
\includegraphics[width=0.9\linewidth,trim=0cm 0cm 0.15cm 0cm,clip]{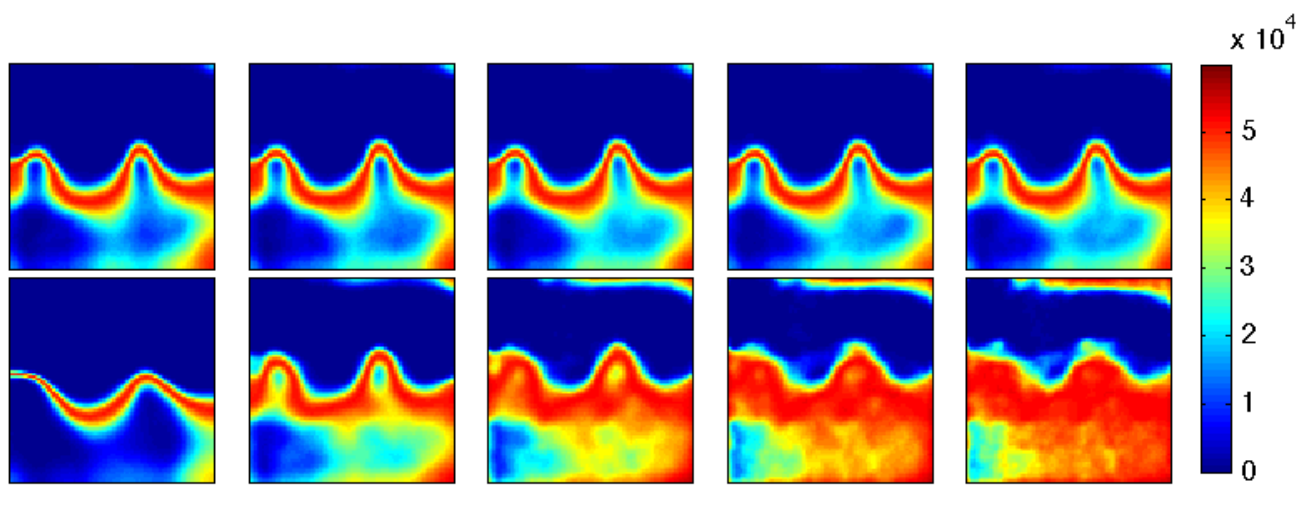}

\caption{(Groundwater flow model) Approximations of $\text{Var}\big(F(u,\tau)\big)$ using the hierarchical (top) and the non-hierarchical (bottom) MCMC.}
\label{fig:gw_3}
\end{center}
\end{figure*}

\begin{figure*}
\begin{center}
\subfloat[\mmd{(Top) Representative samples of the rescaled level-set function $\tau^4\cdot u$ and (bottom) approximations of $\mathbb{E}(\tau^4\cdot u)$ using the hierarchical method. From left-to-right, $\tau$ is initialized at $\tau = 10,30,50,70,90$.}]{\includegraphics[width=0.9\linewidth,trim=0cm 0cm 0.15cm 0cm,clip]{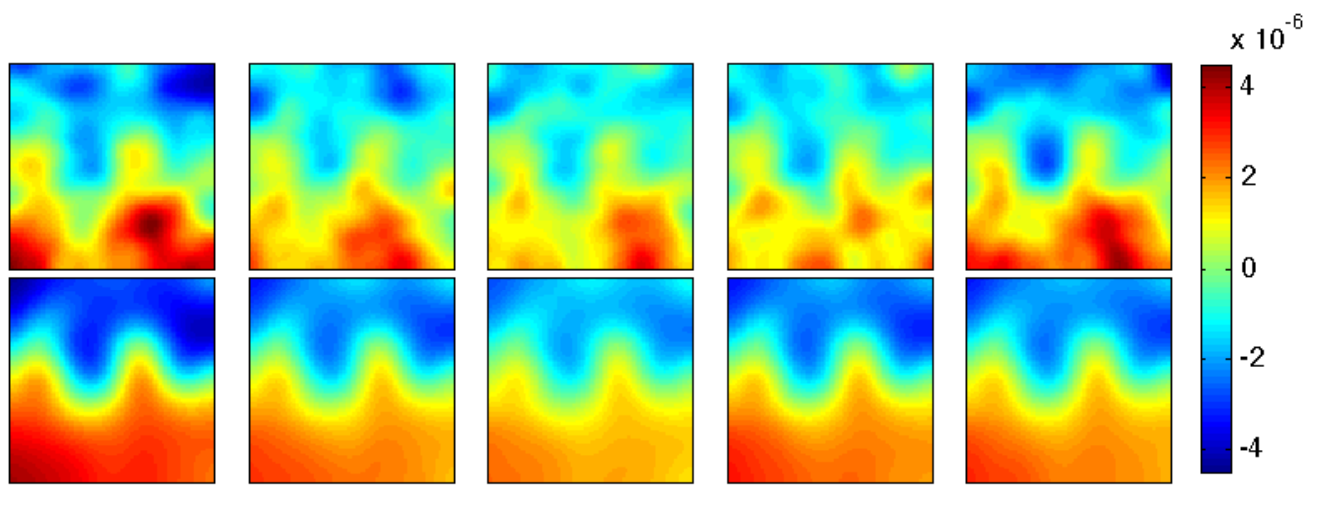}}\\
\subfloat[\mmd{As in (a), using the non-hierarchical method. From left-to-right, $\tau$ is fixed at $\tau = 10,30,50,70,90$.}]{\includegraphics[width=0.9\linewidth,trim=0cm 0cm 0.15cm 0cm,clip]{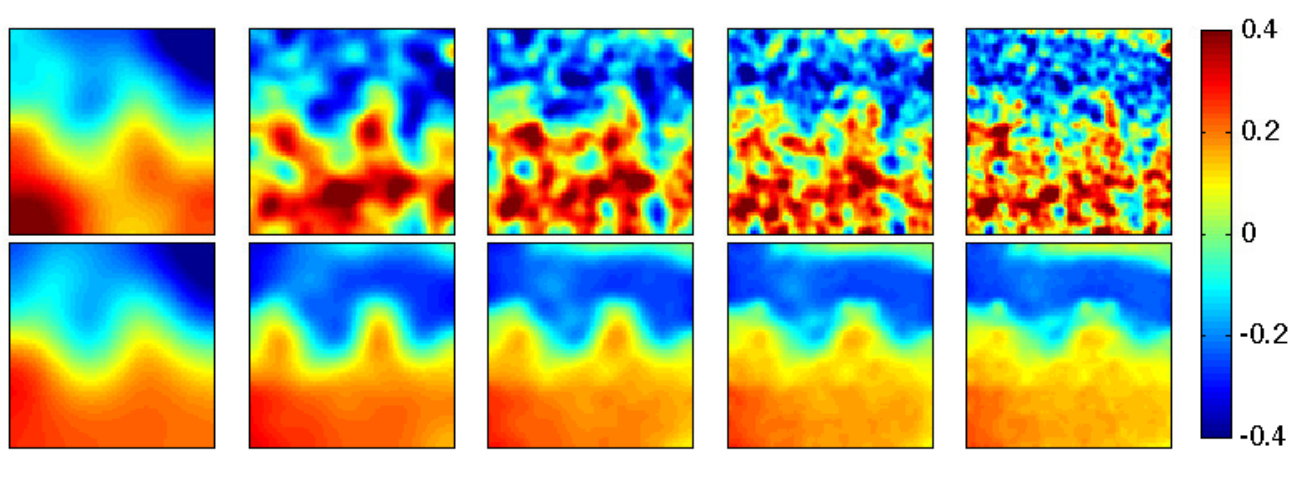}}\\
\caption{\mmd{(Groundwater flow model) Representative samples and sample means of the level set function. The rescaling $\tau^4$ means that the above quantities have the same approximate amplitude. True inverse length scale is $\tau = 15$.}}
\label{fig:gw_4}
\end{center}
\end{figure*}

\begin{figure*}
\begin{center}
\includegraphics[width=\linewidth,trim=2.3cm 0cm 3.5cm 0cm, clip]{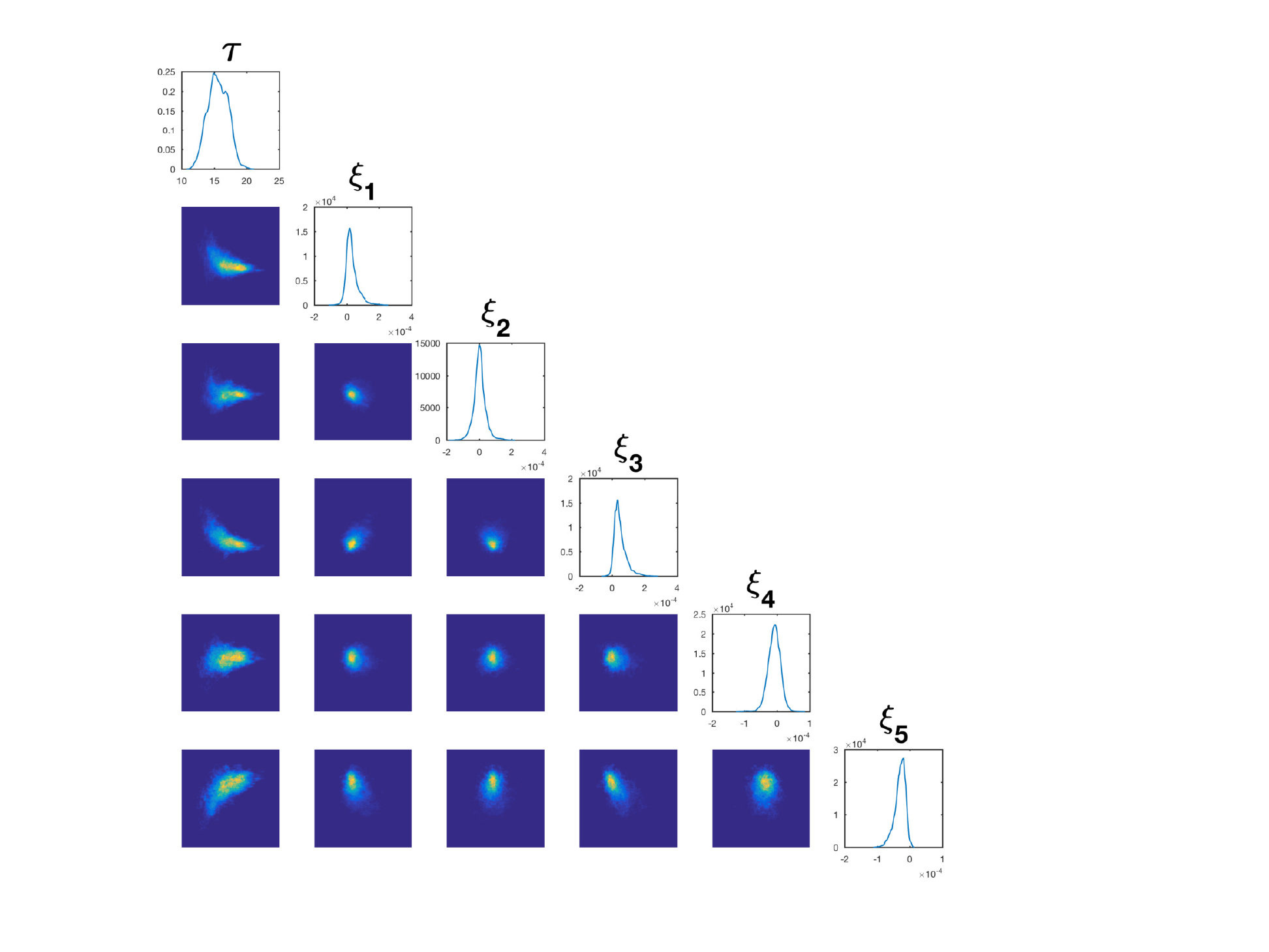}
\caption{(Groundwater flow model) (diagonal) Empirical densities of $\tau$ and the first five KL modes of $u$. (off-diagonal) Empirical joint densities.}
\label{fig:gw_5}
\end{center}
\end{figure*}

\subsection{Electrical Impedance Tomography}
Finally we consider the electrical impedance tomography (EIT) problem. This problem has previously been approached with a non-hierarchical Bayesian level set method \cite{eit}. In this subsection we show that the hierarchical approach outperforms the non-hierarchical approach in the case where the true conductivity is a binary field, given the same number of forward model evaluations.
 
\label{ssec:eit}
\subsubsection{The forward model}
\label{sssec:eit_fwd}
EIT is an imaging technique which attempts to infer the internal conductivity of a body from boundary voltage measurements. Typical applications include medical imaging, as well as subsurface imaging where it is known as electrical resistivity tomography (ERT). We utilize the complete electrode model (CEM), proposed in \cite{cheney}. This is a physically accurate model which has been shown to agree with experimental data up to measurement precision. The strong form of the PDE governing the model is given by
\begin{align*}
\begin{cases}
\displaystyle-\nabla\cdot(\kappa(x)\nabla v(x)) = 0 & x \in D\\[.8em]
\displaystyle\int_{e_l} \kappa\frac{\partial v}{\partial n}\,\dee S = I_l & l=1,\ldots,L\\[.8em]
\displaystyle\kappa(x)\frac{\partial v}{\partial n}(x) = 0 & x \in \partial D\setminus\bigcup_{l=1}^L e_l\\[.8em]
\displaystyle v(x) + z_l\kappa(x)\frac{\partial v}{\partial n}(x) = V_l & x \in e_l, l=1,\ldots,L.
\end{cases}
\end{align*}
Here $D\subseteq \mathbb{R}^2$ is the domain and $\{e_l\}_{l=1}^L \subseteq \partial D$ are electrodes on the boundary upon which currents $\{I_l\}_{l=1}^L$ are injected and voltages $\{V_l\}_{l=1}^L$ are read. The numbers $\{z_l\}_{l=1}^L$ represent the contact impedances of the electrodes. The field \mmd{$\kappa$} represents the conductivity of the body and $v$ represents the potential within the body\footnote{\mmd{In the EIT literature the conductivity field is often denoted $\sigma$, however we have already used this in denoting the marginal variance of random fields.}}. It should be noted that the solution of this PDE comprises both a potential $v \in H^1(D)$ and a vector $\{V_l\}_{l=1}^L$ of boundary voltage measurements.

The inverse problem we consider is the recovery of \mmd{$\kappa$} from a sequence of boundary voltage measurements. A number of (linearly independent) current stimulation patterns $\{I_l\}_{l=1}^L$ may be performed to provide more information; we assume that we perform the maximum $M = L-1$ measurements. Let $Z = L^p(D)$ for some $1\leq p < \infty$ and $Y = \mathbb{R}^{J}$ where $J = LM$. We can concatenate the boundary voltage measurements arising from different stimulation patterns to yield a map $S:Z\rightarrow Y$,
\mmd{\[
\kappa \mapsto (V^{(1)},V^{(2)},\ldots,V^{(M)})
\]
}where $V^{(m)} = \{V_l^{(m)}\}_{l=1}^L \in \mathbb{R}^L$, $m=1,\ldots,M$. 

For the experiments we work on a circular domain $D = \{x \in \mathbb{R}^2\;|\;|x| < 1\}$. 16 electrodes are spaced equally around the boundary providing 50\% coverage. All contact impedances are taken to be $z_l = 0.01$. Adjacent electrodes are stimulated with a current of 0.1, so that the matrix of stimulation patterns $I = \{I^{(j)}\}_{j=1}^{15} \in \mathbb{R}^{16\times 15}$ is given by
\[
\begingroup
\renewcommand{\arraycolsep}{0.18cm} 
I = 0.1\times\left(
\begin{array}{cccc}
+1 & 0 & \cdots & 0\\
-1 & +1 &\cdots & 0\\
0 & -1 & \ddots & 0\\
\vdots & \vdots & \ddots & +1\\
 0 & 0 & 0 & -1
\end{array}\right).
\endgroup
\]

We define our forward map $\mathcal{G}:X\times\mathbb{R}^+\rightarrow\mathbb{R}^J$ by $\mathcal{G} = S\circ F$. As in the groundwater flow example, assume that each $\kappa_i$ in the definition of the level set map is strictly positive. We do not have a continuity result for the map $S$ on $L^p$ for any $1\leq p < \infty$. However the almost-sure continuity of the map $\mathcal{G}$ can be seen via a modification of the proof of Proposition 3.5 in \cite{eit} to include the parameter $\tau$; this modification is almost
identical to the proof of Proposition \ref{p:app} given in the appendix. The uniform boundedness of $\mathcal{G}$ follows from a result in \cite{eit} similarly. Hence as was the case with the identity map example, the conclusions of Proposition \ref{p:app} follow, and we can deduce the conclusions of Theorem \ref{t:2}.

\subsubsection{Simulations and results}
\label{sssec:eit_results}

We fix a true conductivity \mmd{$\kappa^\dagger$}, shown in Figure \ref{fig:eit_means}. As with the groundwater flow experiments, this is constructed explicitly and does not have a true value of $\tau$ associated with it. We generate data $y$ as
\mmd{\[
y = S(\kappa^\dagger) + \eta,\;\;\;\eta \sim N(0,\Gamma)
\]
}where we take the noise covariance \mmd{$\Gamma = 0.0002^2\cdot I$ to be white}. The mean relative error on the generated data is approximately 12\%. The data is generated using a mesh of 43264 elements and simulations are performed used a mesh of 10816 elements, in order to avoid an inverse crime. \mmd{Forward solves are performed using the EIDORS software \cite{eidors}.} All \mmd{level set} field samples are defined on the square $[-1,1]^2$ and restricted to the domain $D$. This has the advantage of allowing for efficient sampling via the Fast Fourier Transform, though has the drawback of introducing possibly non-trivial boundary effects on the domain; no such
effects are observed in our problem, however. \mmd{The discretization on the square is performed via the DFT on a grid of $2^7\times 2^7$ points, and we retain all modes.}

The level set map $F$ is defined such that there are 2 phases, taking the constant values 1 and 10. We take the prior \mmd{level set} field mean to be zero, so that in this case $F$ (and hence $\Phi$) becomes independent of $\tau$. Thus a forward model evaluation is not required for the Gibbs update of $\tau$, and each sample of $(u,\tau)$ using the hierarchical method costs virtually the same as one of $u$ using the non-hierarchical method.

Similarly to the previous experiments, we initialize the hierarchical sampling from $\tau = 10,30,50,70,90$ to check for robustness of the method. We use a sharper prior on $\tau$ than was used previously. \mmd{We again use a Gaussian random walk proposal distribution for $\tau$}. We fix the smoothness parameter $\alpha = 5$ in the prior for $u$, \mmd{and again use Neumann boundary conditions}. We again wish to compare how the hierarchical method compares with the non-hierarchical method. We therefore also look at the 5 different posterior distributions that arise when using each of 5 fixed prior inverse length scales $\tau = 10,30,50,70,90$, which gives another 5 MCMC chains. For both the methods we produce $4\times 10^6$ samples for each chain, and discard the first $2\times10^6$ samples as burn-in when calculating quantities of interest.
 
The traces of the values of $\tau$ along the hierarchical chains are shown in Figure \ref{fig:eit_tautrace}. With the exception of the chain initialized at $\tau = 10$, the chains converge to the sample approximate value of $\tau$. Unlike in previous experiments, the traces have a relatively flat period before the approximate linear convergence to the common length scale. Initializing $\tau = 90$ requires an additional $10^6$ samples to converge, over the other converging chains. 
 
Figure \ref{fig:eit_means} shows the push forwards of the sample means from different chains under the level set map, along with approximations of $\mathbb{E}(F(u,\tau))$ and typical samples of $F(u,\tau)$ coming from the different posteriors. In both the hierarchical and non-hierarchical methods, the chains initialized/fixed at $\tau = 10$ fail to recover the true conductivity, similarly to what was observed with the identity map experiments when initializing at $\tau = 5$. The other chains for the hierarchical method produce very similar results to one another, whilst the effect of fixing the length scale to be too short is apparent in the figures for the non-hierarchical method.

In Figure \ref{fig:eit_var} we see approximations to Var($F(u,\tau)$) under the different posteriors. In both cases, variance is highest around the boundaries of the two inclusions. The difference between the hierarchical and non-hierarchical methods is more apparent here, with higher variance between the two inclusions when the length scale is fixed to be too short.

Again, we look at the level set function $u$ itself in Figure \ref{fig:eit_means_u}. In these plots, as before, we rescale the level set function by $\tau^{\alpha-d/2} = \tau^4$ so that they are all of approximately the same amplitude. As in the previous experiments, there is noticeable contrast between the means for the hierarchical and non-hierarchical methods, and yet more contrast between the typical samples.

Finally, in Figure \ref{fig:eit_densities}, we show the posterior densities
on the inverse length scale and the first five KL modes, as well as correlations
between them. As with the groundwater flow example, although there is no
``true'' inverse length scale, the data is sufficiently informative to define
a small range of values for this parameter under the posterior.

\begin{figure}
\begin{center}
\includegraphics[width=0.7\linewidth]{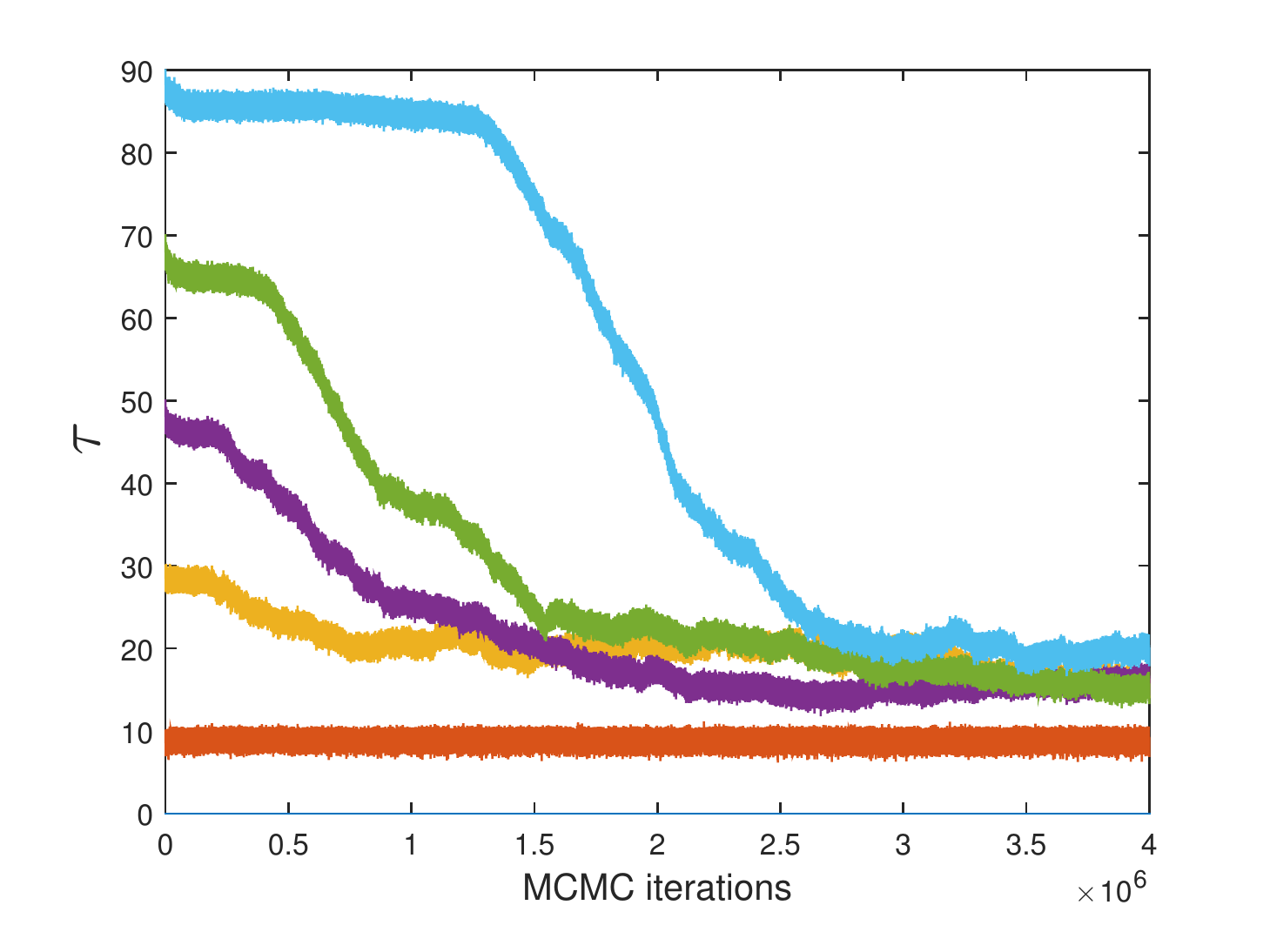}
\caption{(EIT model) The trace of $\tau$ along the MCMC chain, when initialized at the 5 different values $\tau = 10,30,50,70,90$.}
\label{fig:eit_tautrace}
\end{center}
\end{figure}

\begin{figure*}
\begin{center}
\begingroup
\captionsetup[subfigure]{width=0.9\textwidth}
\subfloat[\mmd{(Left) True conductivity field used to generate the data $y$. (Right) The entries $y_i$ of the data vector $y$, plotted against $i$.}]{\includegraphics[width=0.3\linewidth]{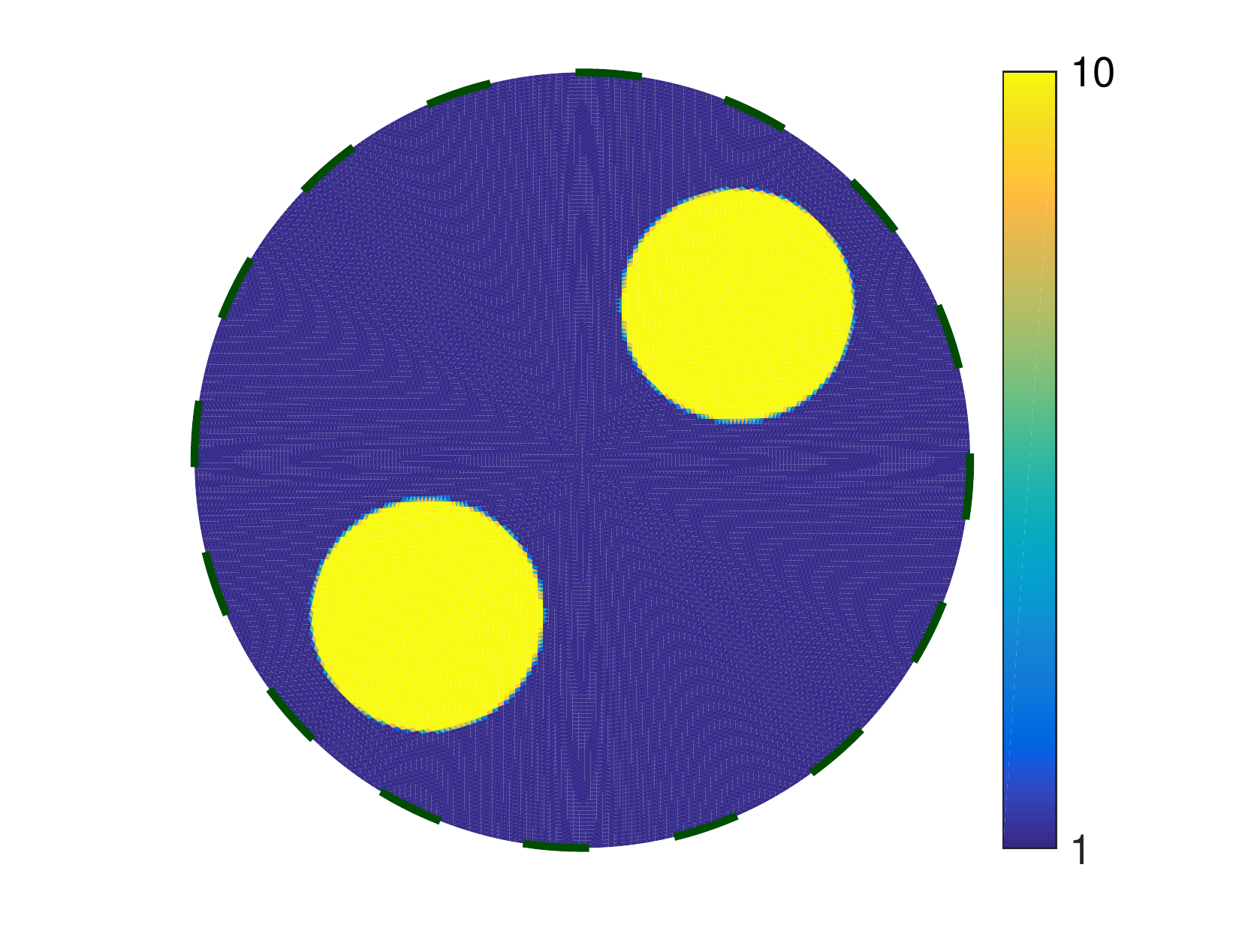}\includegraphics[width=0.3\linewidth]{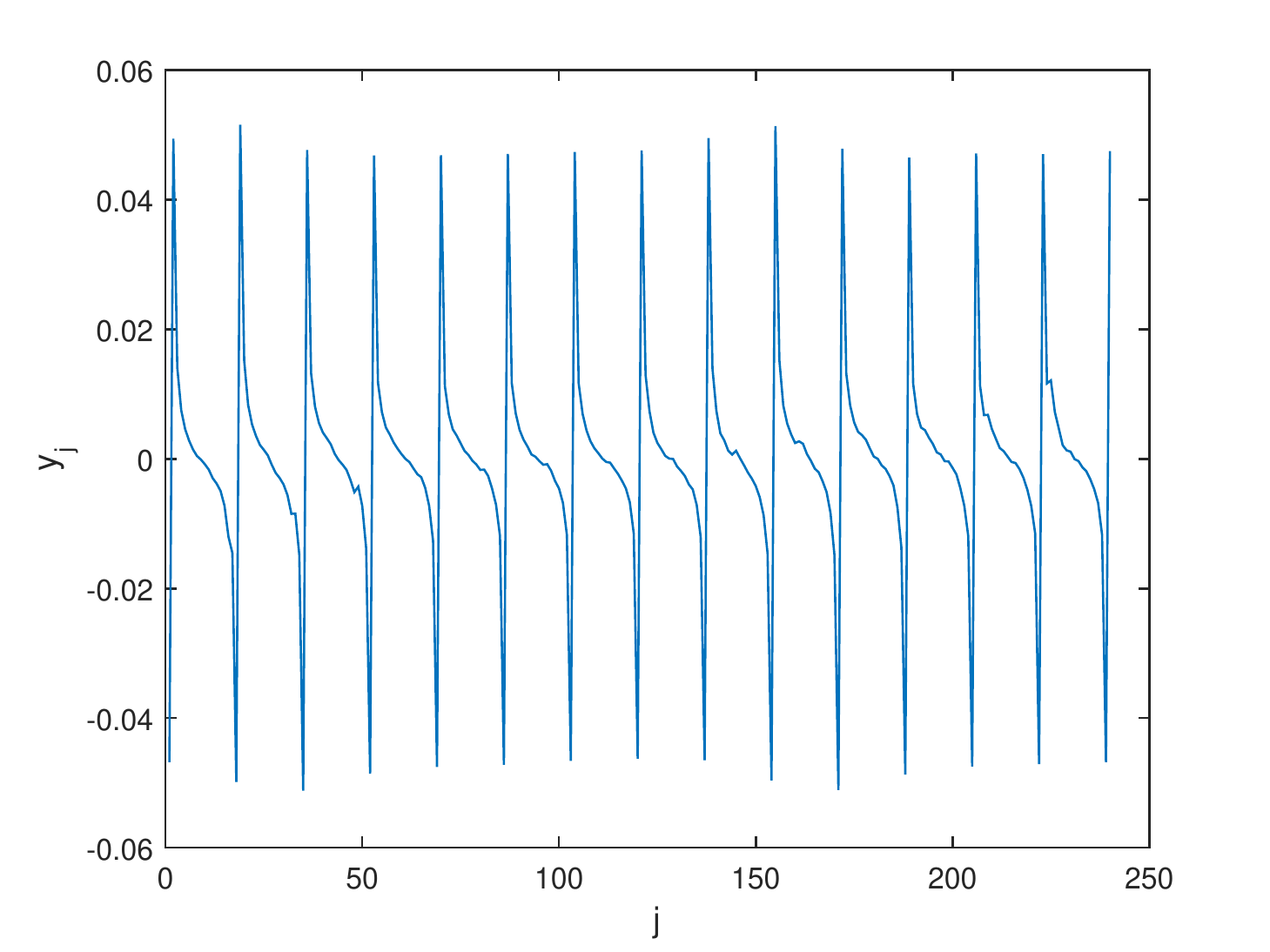}}\\
\endgroup
\subfloat[\mmd{(Top) Representative samples of $F(u,\tau)$ under the hierarchical posterior. (Middle) Approximations of $F(\mathbb{E}(u),\mathbb{E}(\tau))$. (Bottom) Approximations of $\mathbb{E}(F(u,\tau))$. From left-to-right, $\tau$ is initialized at $\tau = 10,30,50,70,90$.}]{\includegraphics[width=0.9\linewidth]{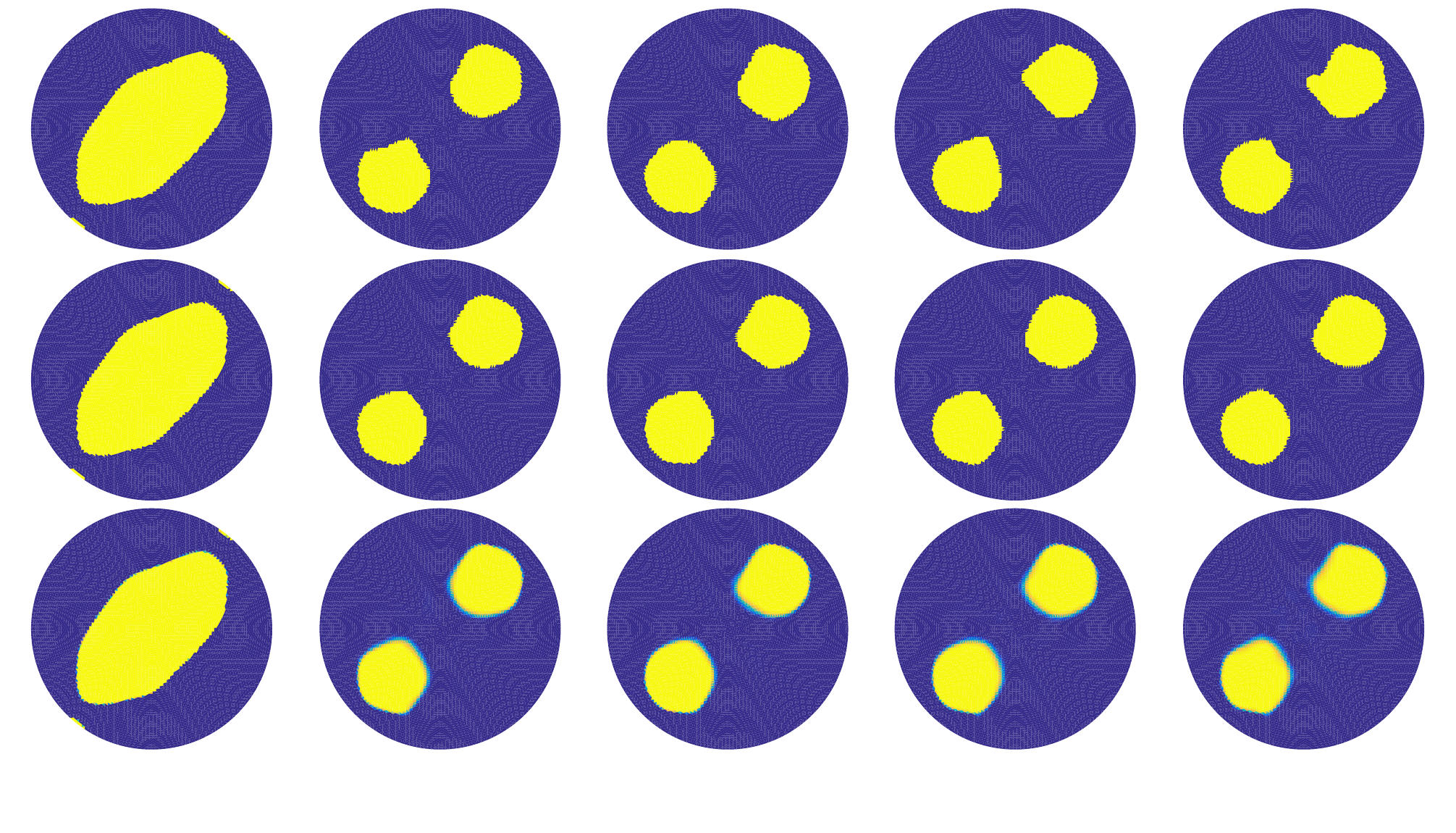}}\\
\subfloat[\mmd{As in (b), using the non-hierarchical method. From left-to-right, $\tau$ is fixed at $\tau = 10,30,50,70,90$.}]{\includegraphics[width=0.9\linewidth]{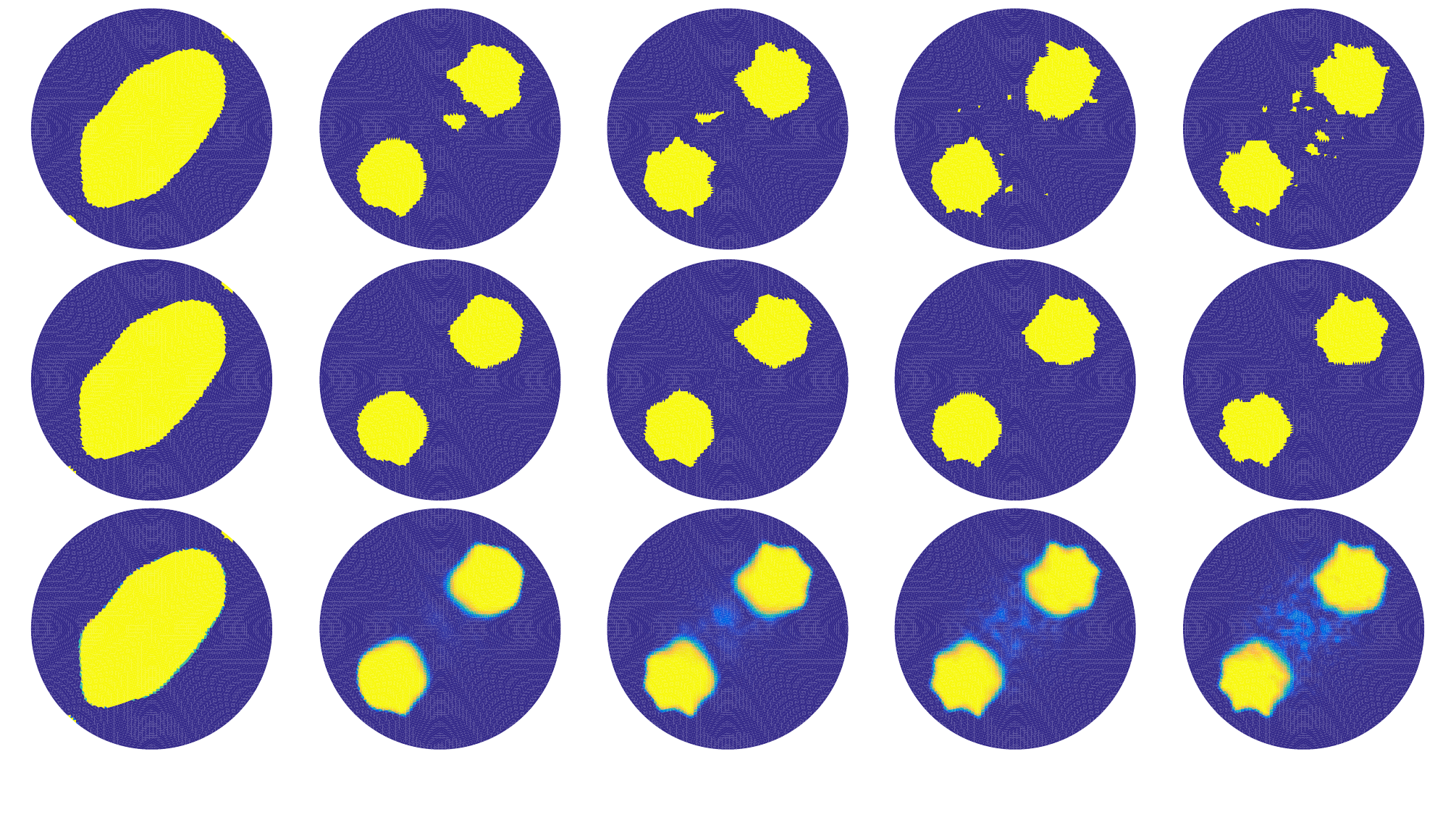}}
\caption{\mmd{Simulations for the EIT model.}}
\label{fig:eit_means}
\end{center}
\end{figure*}

\begin{figure*}
\begin{center}
\includegraphics[width=0.9\linewidth,trim=0cm 0cm 2cm 0cm,clip]{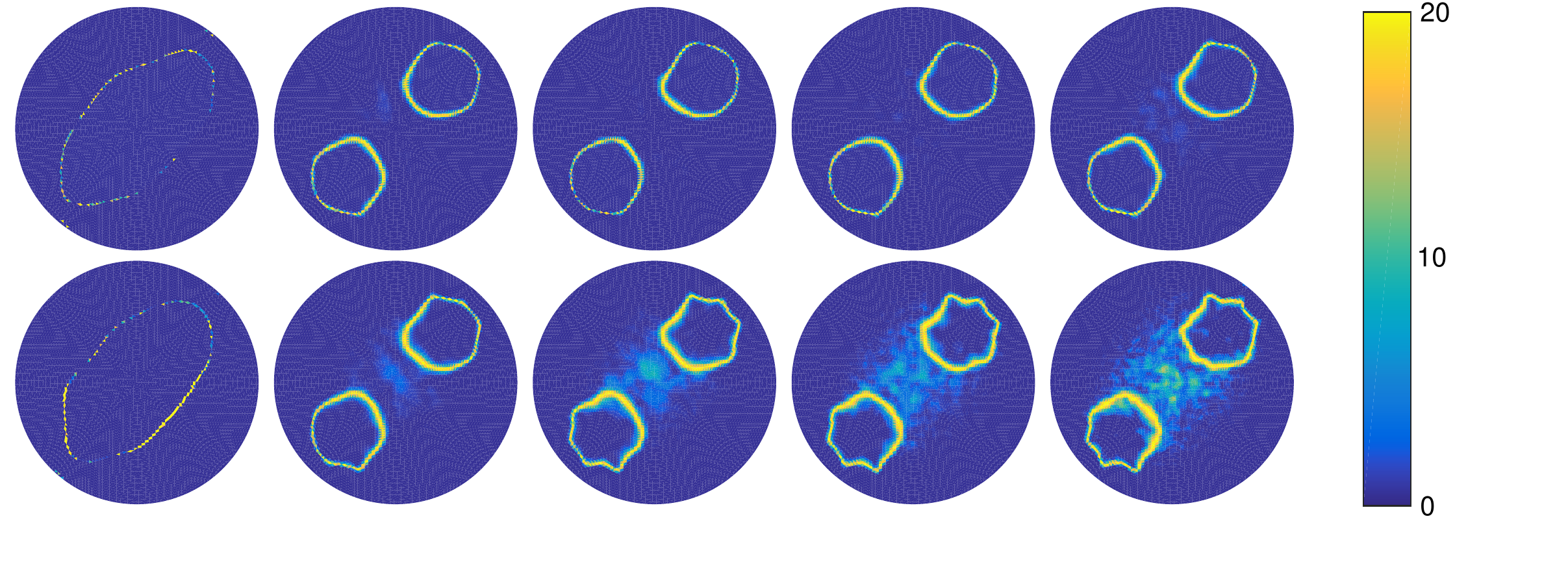}
\caption{(EIT model) Approximations of $\text{Var}(F(u,\tau))$ using the hierarchical (top) and fixed (bottom) priors, with $\tau$ initialized or fixed at $\tau = 10,30,50,70,90$, from left-to-right.}
\label{fig:eit_var}
\end{center}
\end{figure*}

\begin{figure*}
\begin{center}
\subfloat[\mmd{(Top) Representative samples of the rescaled level-set function $\tau^4\cdot u$ and (bottom) approximations of $\mathbb{E}(\tau^4\cdot u)$ using the hierarchical method. From left-to-right, $\tau$ is initialized at $\tau = 10,30,50,70,90$.}]{\includegraphics[width=0.9\linewidth,trim=0cm 0cm 2cm 0cm,clip]{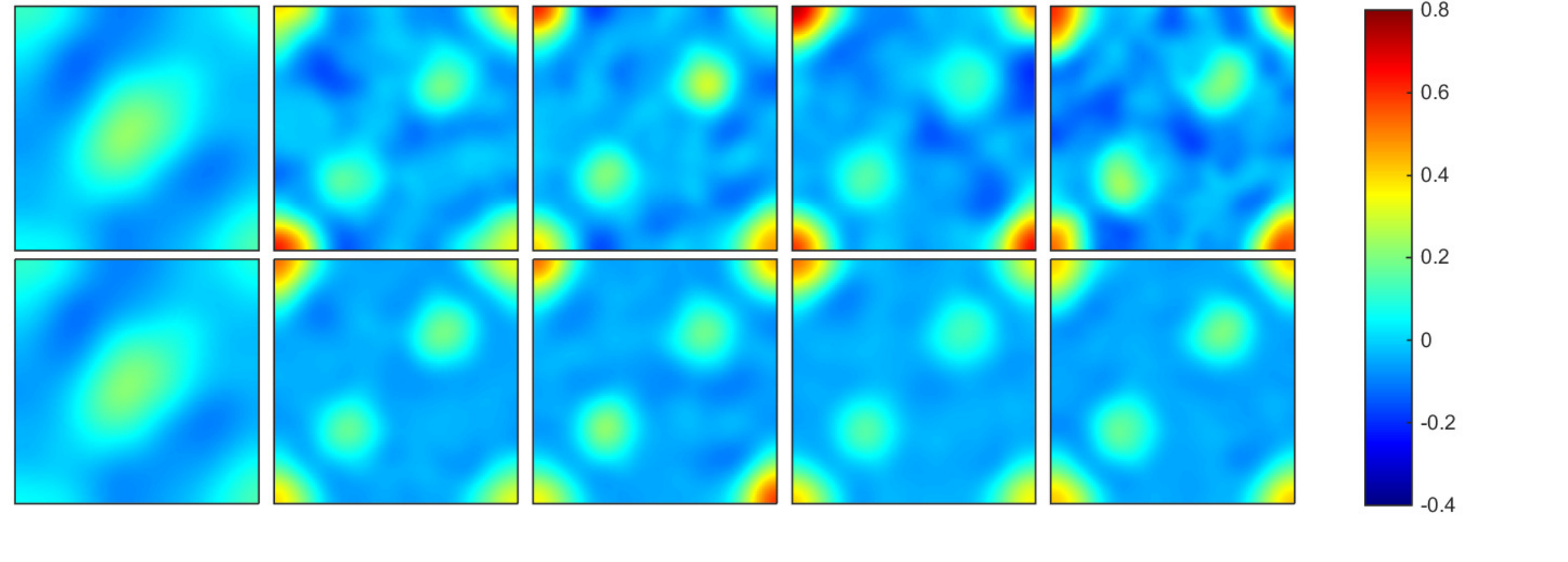}}\\
\subfloat[\mmd{As in (a), using the non-hierarchical method. From left-to-right, $\tau$ is fixed at $\tau = 10,30,50,70,90$.}]{\includegraphics[width=0.9\linewidth,trim=0cm 0cm 2cm 0cm,clip]{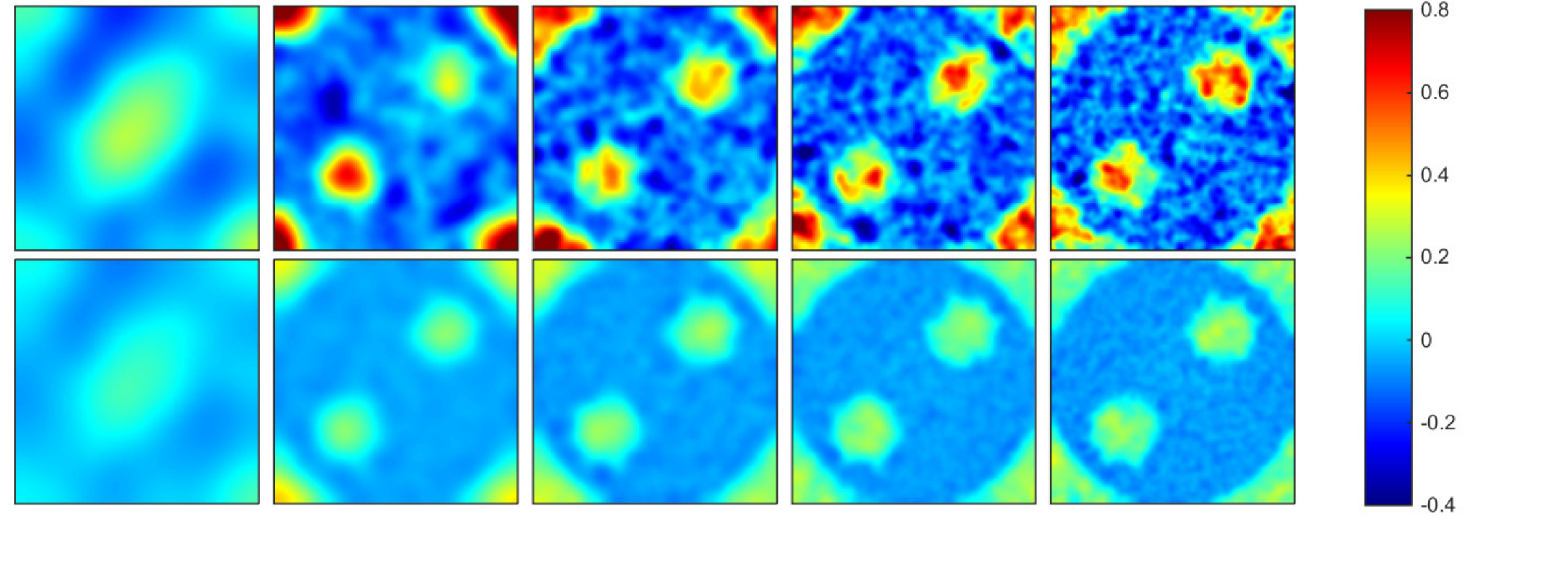}}\\
\caption{\mmd{(EIT model) Representative samples and sample means of the level set function. The rescaling $\tau^4$ means that the above quantities have the same approximate amplitude. True inverse length scale is $\tau = 15$.}}
\label{fig:eit_means_u}
\end{center}
\end{figure*}

\begin{figure*}
\begin{center}
\includegraphics[width=\textwidth]{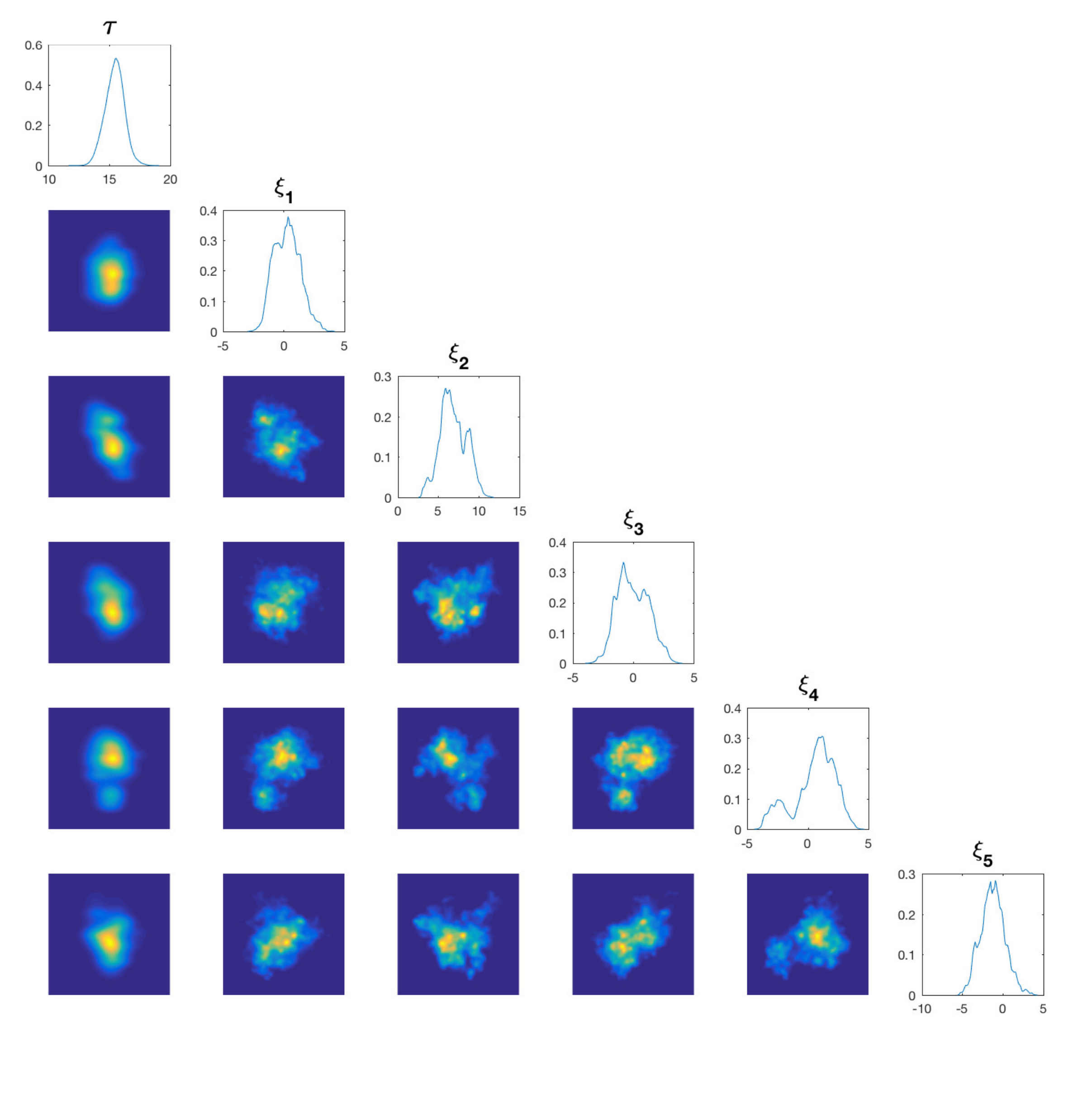}
\caption{(EIT model) (diagonal) Empirical densities of $\tau$ and the first five KL modes of $u$. (off-diagonal) Empirical joint densities.}
\label{fig:eit_densities}
\end{center}
\end{figure*}

\section{Conclusions}
\label{sec:conc}

The level set method is an attractive approach to inverse problems for
the detection of interfaces. Furthermore the Bayesian approach is particularly
desirable when there is a need to quantify uncertainty. In this paper
we have shown that Bayesian level set inversion is considerably enhanced
by a hierarchical approach in which the length scale of the underlying level
set function is inferred from the data. We have demonstrated this
by means of three examples of interest arising in, respectively, the
information, physical and medical sciences; however
many potential applications remain to be explored and this provides an
interesting avenue for future work.

We also developed the theoretical underpinnings for our hierarchical method.
Our work is based on a Metropolis-within-Gibbs approach which alternates
between updating the level set function and the length-scale. The
Metropolis method we use for the \mmd{level set} field update does not use
derivatives of the log-likelihood,
and could be improved by doing so, using the infinite
dimensional variants on MALA and HMC (which use first derivative
information, see the citations in \cite{Cotter2013})
or the manifold MALA and HMC methods, which use higher
order derivatives \cite{CaldGiro}. Another interesting
direction for future work is the design of 
methods with more informed proposals which 
exploit correlations in the level set function and its length-scale. 
And finally it would be interesting to consider pseudo-marginal methods
to sample the hierarchical parameter alone, as in \cite{Giro}.

\mmd{Assuming independence under the prior, it would require little further work to treat the thresholding levels $\{c_i\}$ and the values of the thresholded function $\{\kappa_i\}$ as part of the inference as well; we omitted this here for the sake of clarity. Such a model may be more realistic, and numerical studies of such models may prove interesting.  Another extension of interest may be to place a hyperprior upon the regularity parameter also, which may be useful for improving rates of convergence \cite{van2009adaptive}. This is a more challenging task, again related to singularity of measures. The paper \cite{ABPS14} discusses ways in which this may be done, however it is still an open question in terms of theory and optimal algorithms. Additionally, it may be of interest to overcome the restriction of the ordering of phases $\{\kappa_i\}$ by means of a vector level set method \cite{Tai2004}.}

Finally we mention that the use of a single length-scale within an isotropic
prior is a simple example of more sophisticated hierarchical approaches
which attempt to learn non-stationary and non-isotropic \cite{calvetti2007gaussian,calvetti2008hypermodels}
features of the
level set function from the data. This provides an interesting opportunity
for future work and for ideas from machine learning to play a role
in the solution of inverse problems for interfaces.

\begin{appendix}
\section{Appendix}
\label{sec:A}

\subsection{Proof of Theorems}
\label{ssec:wdp}

\begin{proof}[Theorem \ref{t:wmequiv}]
\begin{enumerate}[(i)]
\item Note that it suffices to show that $\mu_0^\tau \sim \mu^0_0$ for all $\tau > 0$. (Here $\sim$ denotes ``equivalent as measures''). It is known that the eigenvalues of $-\Delta$  on $\mathbb{T}^d$ grow like $j^{2/d}$, and hence the eigenvalues $\lambda_j(\tau)$ of $\mathcal{C}_{\alpha,\tau}$ decay like
\[
\lambda_j(\tau) \asymp (\tau^2 + j^{2/d})^{-\alpha},\;\;\;j\geq 1.
\]
Using Proposition \ref{prop:fhdensity} \mmd{below}, we see that $\mu_0^\tau \sim \mu_0^0$ if
\begin{align*}
\hspace{1.0cm} 
\sum_{j=1}^\infty\left(\frac{\lambda_j(\tau)}{\lambda_j(0)} - 1\right)^2 < \infty.
\end{align*}
Now we have
\begin{align*}
\left|\frac{\lambda_j(\tau)}{\lambda_j(0)} - 1\right|&\asymp \left|\left(1 + \frac{\tau^2}{j^{2/d}}\right)^{-\alpha} - 1\right|\\
&\leq \left|\exp\left(\frac{\alpha\tau^2}{j^{2/d}}\right) - 1\right|\\
&\leq C\frac{\alpha\tau^2}{j^{2/d}}.
\end{align*}
Here we have used that $(1+x)^{-\alpha}-1 \leq \exp(\alpha x)-1$ for all $x \geq 0$ to move from the first to the second line, and that $\exp(x)-1 \leq Cx$ for all $x \in [0,x_0]$ to move from the second to third line.
Now note that when $d \leq 3$, $j^{-4/d}$ is summable, and so it follows that $\mu_0^\tau\sim\mu_0^0$.

\item The case $\tau = 0$ is Theorem 2.18 in \cite{lecturenotes}; the general result follows from the equivalence above.

\item Let \mi{$v \sim N(0,\mathcal{D}_{\sigma,\nu,\ell})$} where $\mathcal{D}_{\sigma,\nu,\ell}$ is as given by (\ref{wmcov}). \mmd{Then we have
\begin{align*}
\mathcal{D}_{\sigma,\nu,\ell} &= \beta\ell^d(I - \ell^2\Delta)^{-\nu-d/2}\\
&= \beta\ell^{d}\ell^{-2\nu-d}(\ell^{-2}I - \Delta)^{-\nu-d/2}\\
&= \beta \tau^{2\alpha - d}(\tau^2I - \Delta)^{-\alpha}\\
&= \beta\tau^{2\alpha - d}\mathcal{C}_{\alpha,\tau}.
\end{align*}
Hence, letting \mi{$u \sim N(0,\mathcal{C}_{\alpha,\tau})$}, we see that
\mi{\begin{align*}
\mathbb{E}\|u\|^2 &= \tr(\mathcal{C}_{\alpha,\tau})\\
&= \frac{1}{\beta}\tau^{d-2\alpha}\tr(\mathcal{D}_{\sigma,\nu,\ell})\\
&= \frac{1}{\beta}\tau^{d-2\alpha}\mathbb{E}\|v\|^2.
\end{align*}}\qed}

\end{enumerate}
\end{proof}

\begin{proof}[Theorem \ref{t:2}]
Proposition \ref{p:app} which follows
shows that $\mu_0$ and $\Phi$ satisfy Assumptions 2.1 in \cite{levelset}, 
with $U = X\times\mathbb{R}^+$. Theorem 2.2 in \cite{levelset} then tells 
us that the posterior exists and is Lipschitz with respect to the data.\qed
\end{proof}

\begin{proposition}
\label{p:app}
Let $\mu_0$ be given by (\ref{eq:prior}) and $\Phi:X\times\mathbb{R}^+\rightarrow\mathbb{R}$ be given by (\ref{eq:phi}). Let Assumptions \ref{ass:forward} hold. Then
\begin{enumerate}[(i)]
\item for every $r > 0$ there is a $K = K(r)$ such that, for all $(u,\tau) \in X\times\mathbb{R}^+$ and all $y \in Y$ with $|y|_\Gamma < r$,
\[
0 \leq \Phi(u,\tau;y) \leq K;
\]
\item for any fixed $y \in Y$, $\Phi(\cdot,\cdot;y):X\times\mathbb{R}^+\rightarrow\mathbb{R}$ is continuous $\mu_0$-almost surely on the complete probability space $(X\times\mathbb{R}^+,\mathcal{X}\otimes\mathcal{R},\mu_0)$;
\item for $y_1,y_2 \in Y$ with $\max\{|y_1|_\Gamma,|y_2|_\Gamma\} < r$, there exists a $C = C(r)$ such that for all $(u,\tau) \in X\times\mathbb{R}^+$,
\[
|\Phi(u,\tau;y_1) - \Phi(u,\tau;y_2)| \leq C|y_1-y_2|_\Gamma.
\]
\end{enumerate}
\end{proposition}

\begin{proof}
\begin{enumerate}[(i)]
\item Recall the level set map $F$ defined by (\ref{lvlsetmap}) defined
via the finite constant values $\kappa_i$ taken on each 
subset $D_i$ of $\overline{D}$. We may bound $F$ uniformly:
\[
|F(u,\tau)| \leq \max\{|\kappa_1|,\ldots|\kappa_n|\} =: F_{\max}
\]
for all $(u,\tau) \in X\times\mathbb{R}^+$. Combining this with Assumption \ref{ass:forward}(ii) it follows that $\mathcal{G}$ is uniformly bounded on $X\times\mathbb{R}^+$. The result then follows from the continuity of $y\mapsto \frac{1}{2}|y-\mathcal{G}(u,\tau)|_\Gamma^2$.
\item Let $(u,\tau) \in X\times\mathbb{R}^+$ and let $D_i(u,\tau)$ be as defined by (\ref{eq:di}), and define $D_i^0(u,\tau)$ by
\mi{\begin{align*}
D_i^0(u,\tau) &= \overline{D}_i(u,\tau)\cap \overline{D}_{i+1}(u,\tau)\\
&= \{x \in D\,|\,u(x) = c_i(\tau)\},\;\;\; i=1,\ldots,n-1.
\end{align*}}We first show that $\mathcal{G}$ is continuous at $(u,\tau)$ whenever $|D_i^0(u,\tau)| = 0$ for $i=1,\ldots,n-1$.

Choose an approximating sequence $\{u_\eps,\tau_\eps\}_{\eps >0}$ of $(u,\tau)$ such that $\|u_\eps - u\|_\infty + |\tau_\eps-\tau| < \eps$ for all $\eps > 0$. We will first show that $\|F(u_\eps,\tau_\eps) - F(u,\tau)\|_{L^p(D)}\rightarrow 0$ for any $p \in [1,\infty)$. As in \cite{levelset} Proposition 2.4, we can write
\begin{align*}
F(u_\eps,\tau_\eps) - F(u,\tau) &= \sum_{i=1}^n\sum_{j=1}^n (\kappa_i - \kappa_j)\mathds{1}_{D_i(u_\eps,\tau_\eps)\cap D_j(u,\tau)}\\
&= \sum_{\substack{i,j=1\\i\neq j}}^n (\kappa_i - \kappa_j)\mathds{1}_{D_i(u_\eps,\tau_\eps)\cap D_j(u,\tau)}.
\end{align*}

From the definition of $(u_\eps,\tau_\eps)$,
\begin{align*}
u(x) - \eps < u_\eps(x) < u(x) + \eps,\;\;\;\tau - \eps < \tau_\eps < \tau + \eps 
\end{align*}
for all $x \in D$ and $\eps > 0$. \mmd{We claim that for $|i-j| > 1$ and $\eps$ sufficiently small, $D_i(u_\eps,\tau_\eps)\cap D_j(u,\tau) = \varnothing$. First note that
\mi{\begin{align*}
D_i(u_\eps,\tau_\eps) &= \big\{x \in D\;\big|\; \tau_\eps^{d/2-\alpha}c_{i-1} \leq u_\eps(x) < \tau_\eps^{d/2-\alpha}c_i\big\}\\
&= \big\{x \in D\;\big|\; c_{i-1} \leq \tau_\eps^{\alpha-d/2}u_\eps(x)  < c_i\big\}.
\end{align*}
Then we have that
\begin{align*}
D_i(u_\eps,\tau_\eps)\cap D_j(u,\tau) =\{x \in D\;|\;&c_{i-1} \leq \tau_\eps^{\alpha-d/2}u_\eps(x) < c_i,\\
&c_{j-1} \leq  \tau^{\alpha-d/2}u(x) < c_j\}.
\end{align*}
Now, since $u$ is bounded,
\begin{align*}
\tau^{\alpha-d/2}u(x) -\mathcal{O}(\eps) &< \tau_\eps^{\alpha-d/2}u_\eps(x) < \tau^{\alpha-d/2}u(x) + \mathcal{O}(\eps) 
\end{align*}
and so
\begin{align*}
D_i(u_\eps,\tau_\eps)\cap D_j(u,\tau) \subseteq \{x \in D\;|\;c_{i-1}-\mathcal{O}(\eps) &\leq \tau^{\alpha-d/2}u(x)< c_i + \mathcal{O}(\eps),\\
c_{j-1} &\leq \tau^{\alpha-d/2}u(x) < c_j\}.
\end{align*}
From the strict ordering of the $\{c_i\}_{i=1}^n$ we deduce that for $|i-j| > 1$ and small enough $\eps$, the right hand side is empty.
We hence look at the cases $|i-j| = 1$. With the same reasoning as above, we see that
\begin{align*}
D_i(u_\eps,\tau_\eps)\cap D_{i+1}(u,\tau) &\subseteq \big\{x\in D\;\big|\;c_i -\mathcal{O}(\eps) \leq \tau^{\alpha-d/2}u(x)< c_i + \mathcal{O}(\eps) \big\}\\
&\rightarrow \big\{x \in D\;\big|\; \tau^{\alpha-d/2}u(x) = c_i\big\}\\
&= \big\{x \in D\;\big|\; u(x) = \tau^{d/2-\alpha}c_i\big\}\\
&= D_i^0(u,\tau)
\end{align*}
and also
\begin{align*}
D_i(u_\eps,\tau_\eps)\cap D_{i-1}(u,\tau) &\subseteq \big\{x\in D\;\big|\; c_{i-1} - \mathcal{O}(\eps) < \tau^{\alpha-d/2}u(x) < c_{i-1}\big\}\rightarrow \varnothing.
\end{align*}}}

Assume that each $|D_i^0(u,\tau)| = 0$, then it follows that $|D_i(u_\eps,\tau_\eps)\cap D_j(u,\tau)|\rightarrow 0$ whenever $i \neq j$. Therefore we have that
\begin{align*}
\|F(u_\eps,\tau_\eps) - F(u,\tau)\|_{L^p(D)}^p &=  \sum_{\substack{i,j=1\\i\neq j}}^n \int_{D_i(u_\eps,\tau_\eps)\cap D_j(u,\tau)} |\kappa_i - \kappa_j|^p\,\dee x\\
&\leq (2F_{\max})^p \sum_{\substack{i,j=1\\i\neq j}}^n |D_i(u_\eps,\tau_\eps)\cap D_j(u,\tau)|\\
&\rightarrow 0.
\end{align*}
Thus $F$ is continuous at $(u,\tau)$. By Assumption \ref{ass:forward}(i) it follows that $\mathcal{G}$ is continuous at $(u,\tau)$.

We now claim that $|D_i^0(u,\tau)| = 0$ $\mu_0$-almost surely for each $i$. By Tonelli's theorem, we have that
\begin{align*}
\mathbb{E}|D_i^0(u,\tau)|&= \int_{X\times\mathbb{R}^+}|D_i^0(u,\tau)|\,\mu_0(\dee u,\dee \tau)\\
&=\int_{X\times\mathbb{R}^+}\left(\int_{\mathbb{R^d}}\mathds{1}_{D_i^0(u,\tau)}(x)\,\dee x\right) \mu_0(\dee u,\dee \tau)\\
&=\int_{\mathbb{R}^d}\left(\int_{X\times\mathbb{R}^+}\mathds{1}_{D_i^0(u,\tau)}(x)\,\mu_0(\dee u,\dee \tau)\right) \dee x\\
&=\int_{\mathbb{R}^d}\left(\int_0^\infty \left(\int_X \mathds{1}_{D_i^0(u,\tau)}(x)\,\mu_0^\tau(\dee u)\right)\,\pi_0(\dee \tau)\right) \dee x\\
&=\int_{\mathbb{R}^d}\left(\int_0^\infty \mu_0^\tau(\{u \in X\;|\;u(x) = c_i(\tau)\})\,\pi_0(\dee \tau)\right) \dee x.
\end{align*}
For each $\tau \geq 0$ and $x \in D$, $u(x)$ is a real-valued Gaussian random variable under $\mu_0^\tau$. It follows that $\mu_0^\tau(\{u \in X\;|\;u(x) = c_i(\tau)\}) = 0$, and so $\mathbb{E}|D_i^0(u,\tau)| = 0$. Since $|D_i^0(u,\tau)| \geq 0$ we have that $|D_i^0(u,\tau)| = 0$ $\mu_0$-almost surely. The result now follows.
\item For fixed $(u,\tau) \in X\times\mathbb{R}^+$, the map $y\mapsto\frac{1}{2}|y-\mathcal{G}(u,\tau)|_\Gamma^2$ is smooth and hence locally Lipschitz.\qed
\end{enumerate}
\end{proof}

\begin{proof}[Theorem \ref{t:exponent}]
Recall that the eigenvalues of $\mathcal{C}_{\alpha,\tau}$ satisfy $\lambda_j(\tau) \asymp (\tau^2 + j^{2/d})^{-\alpha}$. Then we have that
\begin{align*}
\left(\frac{\lambda_j(0)}{\lambda_j(\tau)}-1\right) \asymp (1+\tau^2 j^{-2/d})^\alpha - 1 = \mathcal{O}(j^{-2/d}).
\end{align*}
It follows that
\begin{align}
\label{eq:eigenbound}
\sum_{j=1}^\infty \left(\frac{\lambda_j(0)}{\lambda_j(\tau)}-1\right)^p < \infty\;\;\;\text{if and only if }d<2p.
\end{align}
\begin{enumerate}[(i)]
\item We first prove the `if' part of the statement. We have \mi{$u \sim N(0,\mathcal{C}_0)$}, and so \mi{$\mathbb{E}\langle u,\varphi_j\rangle^2 = \lambda_j(0)$}. Since the terms within the sum are non-negative, by Tonelli's theorem we can bring the expectation inside the sum to see that that
{\mi{\begin{align*}
\mathbb{E}\sum_{j=1}^\infty\left(\frac{1}{\lambda_j(\tau)} - \frac{1}{\lambda_j(0)}\right)\langle u,\varphi_j\rangle^2
&= \sum_{j=1}^\infty\left(\frac{\lambda_j(0)}{\lambda_j(\tau)}-1\right)
\end{align*}}}which is finite if and only if $d < 2$, i.e. $d=1$. It follows that the sum is finite almost surely.

For the converse, suppose that $d \geq 2$ so that the series in (\ref{eq:eigenbound}) diverges when $p=1$. Let $\{\xi_j\}_{j\geq 1}$ be a sequence of i.i.d. $N(0,1)$ random variables so that \mi{$\langle u,\varphi_j\rangle^2$} has the same distribution as $\lambda_j(0)\xi^2$. Define the sequence $\{Z_n\}_{n\geq 1}$ by
\begin{align*}
Z_n &= \sum_{j=1}^n \left(\frac{\lambda_j(0)}{\lambda_j(\tau)}-1\right)\xi_j^2\\
&= \sum_{j=1}^n \left(\frac{\lambda_j(0)}{\lambda_j(\tau)}-1\right) + \sum_{j=1}^n \left(\frac{\lambda_j(0)}{\lambda_j(\tau)}-1\right)(\xi_j^2-1)\\
&=: X_n + Y_n.
\end{align*}
Then the result follows if $Z_n$ diverges with positive probability. By assumption we have that $X_n$ diverges. In order to show that $Z_n$ diverges with positive probability it hence suffices to show that $Y_n$ converges with positive probability. Define the sequence of random variables $\{W_j\}_{j\geq 1}$ by 
\[
W_j = \left(\frac{\lambda_j(0)}{\lambda_j(\tau)}-1\right)(\xi_j^2-1).
\]
It can be checked that
\[
\mathbb{E}(W_j) = 0,\;\;\;\text{Var}(W_j) = 2\left(\frac{\lambda_j(0)}{\lambda_j(\tau)}-1\right)^2.
\]
The series of variances converges if and only if $d\leq 3$, using (\ref{eq:eigenbound}) with $p = 2$. We use Kolmogorov's two series theorem, Theorem 3.11 in \cite{varadhan2001}, to conclude that $Y_n = \sum_{j=1}^n W_j$ converges almost surely and the result follows.

\item Now we have
\begin{align*}
\log\left(\frac{\lambda_j(\tau)}{\lambda_j(0)}\right) &= -\log\left(1-\left(1-\frac{\lambda_j(0)}{\lambda_j(\tau)}\right)\right)\\
&=\left(1-\frac{\lambda_j(0)}{\lambda_j(\tau)}\right) + \frac{1}{2}\left(1-\frac{\lambda_j(0)}{\lambda_j(\tau)}\right)^2 + \text{h.o.t.}
\end{align*}

Let $\{\xi_j\}_{j\geq 1}$ be a sequence of i.i.d. $N(0,1)$ random variables, so that again we have that \mi{$\langle u,\varphi_j\rangle^2$} has the same distribution as $\lambda_j(0)\xi^2$. Then it is sufficient to show that the series
\[
I = \sum_{j=1}^\infty \left[\left(\frac{\lambda_j(0)}{\lambda_j(\tau)}-1\right)\xi_j^2 + \log\left(\frac{\lambda_j(\tau)}{\lambda_j(0)}\right)\right]
\]
is finite almost surely. We use the above approximation for the logarithm to write
\begin{align*}
I &= \sum_{j=1}^\infty\left(\frac{\lambda_j(0)}{\lambda_j(\tau)}-1\right)(\xi_j^2-1)\\
&\hspace{1cm}+ \sum_{j=1}^\infty\left[\frac{1}{2}\left(1-\frac{\lambda_j(0)}{\lambda_j(\tau)}\right)^2 + \text{h.o.t.}\right].
\end{align*}
The second sum converges if and only if $d<4$, i.e. $d\leq 3$. The almost sure convergence of the first term is shown in the proof of part (i).\qed
\end{enumerate}
\end{proof}

\begin{proposition}
\label{prop:point_obs_cts}
Let $D\subseteq\mathbb{R}^d$. Define the construction map $F:X\times\mathbb{R}^+\rightarrow \mathbb{R}^D$ by (\ref{lvlsetmap}). Given $x_0 \in D$ define $\mathcal{G}:X\times\mathbb{R}^+\rightarrow\mathbb{R}$ by $\mathcal{G}(u,\tau) = F(u,\tau)|_{x_0}$. Then $\mathcal{G}$ is continuous at any $(u,\tau) \in X\times\mathbb{R}^+$ with $u(x_0) \neq c_i(\tau)$ for each $i=0,\ldots,n$. In particular, $\mathcal{G}$ is continuous $\mu_0$-almost surely when $\mu_0$ is given by (\ref{eq:prior}). Additionally, $\mathcal{G}$ is uniformly bounded.
\end{proposition}
\begin{proof}
The uniform boundedness is clear. For the continuity, let $(u,\tau) \in X\times\mathbb{R}^+$ with  $u(x_0) \neq c_i(\tau)$ for each $i=0,\ldots,n$. Then there exists a unique $j$ such that
\begin{align}
\label{eq:cont1}
c_{j-1}(\tau) < u(x_0) < c_j(\tau).
\end{align}
Given $\delta > 0$, let $(u_\delta,\tau_\delta) \in X\times\mathbb{R}^+$ be any pair such that
\[
\|u_\delta - u\|_\infty + |\tau_\delta-\tau| < \delta.
\]
Then it is sufficient to show that for all $\delta$ sufficiently small, $x_0 \in D_j(u_\delta,\tau_\delta)$, i.e. that
\[
c_{j-1}(\tau_\delta) \leq u_\delta(x_0) < c_j(\tau_\delta).
\]
From this it follows that $G(u_\delta,\tau_\delta) = G(u,\tau)$.

Since the inequalities in (\ref{eq:cont1}) are strict, we can find $\alpha > 0$ such that
\begin{align}
\label{eq:cont2}
c_{j-1} + \alpha < u(x_0) < c_j(\tau) - \alpha.
\end{align}
Now $c_j$ is continuous at $\tau > 0$, and so there exists a $\gamma > 0$ such that for any $\lambda > 0$ with $|\lambda-\tau| < \gamma$ we have
\begin{align}
\label{eq:cont3}
c_j(\lambda) - \alpha/2 < c_j(\tau) < c_j(\lambda) + \alpha/2.
\end{align}
We have that $\|u_\delta - u\|_\infty < \delta$, and so in particular,
\begin{align}
\label{eq:cont4}
u(x_0) - \delta < u_\delta(x_0) < u(x_0) + \delta.
\end{align}
We can combine (\ref{eq:cont2})-(\ref{eq:cont4}) to see that, for $\delta < \gamma$,
\[
c_{j-1}(\tau_\delta) - \delta + \alpha/2 < u_\delta(x_0) < c_j(\tau_\delta) + \delta - \alpha/2
\]
and so in particular, for $\delta < \min\{\gamma,\alpha/2\}$,
\[
c_{j-1}(\tau_\delta) < u_\delta(x_0) < c_j(\tau_\delta).\hfill\qed
\]
\end{proof}

\subsection{Radon-Nikodym Derivatives in Hilbert Spaces}
\label{ssec:rnd}

The following proposition gives an
explicit formula for the density of one Gaussian with
respect to another and is used in defining the acceptance probability
for the length-scale updates in our algorithm.
Although we only use the proposition 
in the case where $H$ is a function space and the mean $m$ is \mi{zero}, 
we provide a proof in the more general case 
where $m$ is an arbitrary element of 
separable Hilbert space $H$ as this setting may be of
independent interest. 
\begin{proposition}
\label{prop:fhdensity}
Let $(H,\langle\cdot,\cdot\rangle,\|\cdot\|)$ be a separable Hilbert space, and let $A, B$ be positive trace-class operators on $H$. Assume that $A$ and $B$ share a common complete set of orthonormal eigenvectors $\{\varphi_j\}_{j\geq 1}$, with the eigenvalues $\{\lambda_j\}_{j\geq 1}$, $\{\gamma_j\}_{j\geq 1}$ defined by
\[
A\varphi_j = \lambda_j\varphi_j,\;\;\; B\varphi_j = \gamma_j\varphi_j
\] 
for all $j \geq 1$. Assume further that the eigenvalues satisfy
\[
\sum_{j=1}^\infty \left(\frac{\lambda_j}{\gamma_j} - 1\right)^2 < \infty.
\]
Let $m \in H$ and define the measures $\mu = N(m,A)$ and $\nu = N(m,B)$. Then $\mu$ and $\nu$ are equivalent, and their Radon-Nikodym derivative is given by
\[
\frac{\dee \mu}{\dee \nu}(u) = \prod_{j=1}^\infty\frac{\gamma_j}{\lambda_j}\cdot \exp\Bigg(\frac{1}{2}\sum_{j=1}^\infty\bigg(\frac{1}{\gamma_j} - \frac{1}{\lambda_j}\bigg)\langle u-m,\varphi_j\rangle^2\Bigg).
\]
\end{proposition}

\begin{proof}
The assumption on summability of the eigenvalues means that the Feldman-H\'ajek theorem applies, and so we know that $\mu$ and $\nu$ are equivalent. We show that the Radon-Nikodym derivative is as given above.

Define the product measures $\hat{\mu},\hat{\nu}$ on $\mathbb{R}^\infty$ by
\[
\hat{\mu} = \prod_{j=1}^\infty \hat{\mu}_j,\;\;\;\hat{\nu} = \prod_{j=1}^\infty \hat{\nu}_j
\]
where $\hat{\mu}_j = N(0,\lambda_j)$, $\hat{\nu}_j = N(0,\gamma_j)$. As a consequence of a result of Kakutani, see \cite{daprato} Proposition 1.3.5, we have that $\hat{\mu}\sim\hat{\nu}$ with
\begin{align*}
\frac{\dee \hat{\mu}}{\dee \hat{\nu}}(x) &= \prod_{j=1}^\infty\frac{\dee \hat{\mu}_j}{\dee \hat{\nu}_j}(x_j)\\
&= \prod_{j=1}^\infty\frac{\gamma_j}{\lambda_j}\cdot \exp\Bigg(\frac{1}{2}\sum_{j=1}^\infty\bigg(\frac{1}{\gamma_j} - \frac{1}{\lambda_j}\bigg)x_j^2\Bigg).
\end{align*}
We associate $H$ with $\mathbb{R}^\infty$ via the map $G:H\rightarrow\mathbb{R}^\infty$, given by
\[
G_ju = \langle u,\varphi_j\rangle,\;\;\;j\geq 1.
\]
Note that the image of $G$ is $\ell^2 \subseteq\mathbb{R}^\infty$, and $G:H\rightarrow\ell^2$ is an isomorphism. Since $A$ and $B$ are trace-class, samples from $\hat{\mu}$ and $\hat{\nu}$ almost surely take values in $\ell^2$. $G^{-1}$ is hence almost surely defined on samples from $\hat{\mu}$ and $\hat{\nu}$. Define the translation map $T_m:H\rightarrow H$ by $T_m u = u + m$. Then by the Karhunen-Lo\`eve theorem, the measures $\mu$ and $\nu$ can be expressed as the push-forwards
\[
\mu = T_m^\#(G^{-1})^\#\hat{\mu},\;\;\;\nu = T_m^\#(G^{-1})^\#\hat{\nu}.
\]
Now let $f:H\rightarrow\mathbb{R}$ be bounded measurable, then we have
\begin{align*}
\int_H f(u)\,\mu(\dee u) &= \int_H f(u)\,\big[T_m^\#(G^{-1})^\#\hat{\mu}\big](\dee u)\\
&= \int_{\mathbb{R}^\infty} f(G^{-1}x+m)\,\hat{\mu}(\dee x)\\
&= \int_{\mathbb{R}^\infty} f(G^{-1}x+m)\frac{\dee \hat{\mu}}{\dee \hat{\nu}}(x)\,\hat{\nu}(\dee x)\\
&= \int_H f(u)\frac{\dee \hat{\mu}}{\dee \hat{\nu}}(G(u-m))\,\big[T_m^\#(G^{-1})^\#\hat{\nu}\big](\dee u)\\
&= \int_H f(u)\frac{\dee \hat{\mu}}{\dee \hat{\nu}}(G(u-m))\,\nu(\dee u).
\end{align*}
From this is follows that we have
\begin{align*}
\frac{\dee \mu}{\dee \nu}(u) &= \frac{\dee\hat{\mu}}{\dee\hat{\nu}}(G(u-m))\\
&= \prod_{j=1}^\infty\frac{\gamma_j}{\lambda_j}\cdot \exp\Bigg(\frac{1}{2}\sum_{j=1}^\infty\bigg(\frac{1}{\gamma_j} - \frac{1}{\lambda_j}\bigg)\langle u-m,\varphi_j\rangle^2\Bigg).
\end{align*}\qed
\end{proof}
\begin{remark}
The proposition above, in the case $m=0$, is given as Theorem 1.3.7 in \cite{daprato} except that, there, the factor before the exponential is omitted. This is because it does not depend on $u$, and all measures involved are probability measures and hence normalized. We retain the factor as we are interested in the precise value of the derivative for the MCMC algorithm; in particular its dependence on
the length-scale.$\hfill\qed$
\end{remark}

\end{appendix}

\noindent {\bf Acknowledgements.}
AMS is grateful to DARPA, EPSRC and ONR for 
financial support. MMD is supported by the EPSRC-funded
MASDOC graduate training program. The authors are grateful to Dan Simpson
for helpful discussions. \mmd{The authors are also grateful for discussions with Omiros Papaspiliopoulos about links with probit. The authors would also like to the two anonymous referees for their comments that have helped improve the quality of the paper.} This research utilized Queen Mary's MidPlus computational facilities, supported by QMUL Research-IT and funded by EPSRC grant EP/K000128/1.

\bibliography{references}
\bibliographystyle{plain}

\end{document}